\documentclass[11pt]{amsart}

\usepackage{amsmath,amsthm,amscd,amssymb,amsfonts}
\usepackage[margin=1.14 in]{geometry}
\usepackage{subfig}
\usepackage{graphicx}   
\usepackage{hyperref}
\usepackage{enumitem}
\usepackage[labelfont=bf]{caption}
\allowdisplaybreaks
\usepackage{color}
\usepackage{mathtools}
\usepackage{bbm}
\usepackage{xcolor}
\usepackage{tikz,tikz-cd}

\allowdisplaybreaks

\theoremstyle{plain}

\newtheorem{theorem}{Theorem}[section]
\newtheorem*{theorem*}{Theorem}
\newtheorem{lemma}[theorem]{Lemma}

\newtheorem{proposition}[theorem]{Proposition}

\theoremstyle{definition}
\newtheorem{definition}[theorem]{Definition}

\theoremstyle{remark}
\newtheorem{remark}[theorem]{Remark}

\numberwithin{equation}{section}

\newcommand{\bA}{\mathbb{A}}

\newcommand{\bC}{\mathbb{C}}

\newcommand{\bE}{\mathbb{E}}

\newcommand{\bN}{\mathbb{N}}
\newcommand{\bP}{\mathbb{P}}
\newcommand{\bR}{\mathbb{R}}

\newcommand{\bW}{\mathbb{W}}

\newcommand{\cA}{\mathcal{A}}
\newcommand{\cB}{\mathcal{B}}
\newcommand{\cC}{\mathcal{C}}

\newcommand{\cF}{\mathcal{F}}

\newcommand{\cL}{\mathcal{L}}

\newcommand{\cU}{\mathcal{U}}

\newcommand{\NCPd}{{\rm NCP}_d }

\newcommand{\tr}{{\rm tr}\,}
\newcommand{\trn}{{\rm tr}_n\,}

\DeclareMathOperator{\Lip}{Lip}
\DeclareMathOperator{\Tr}{Tr}

\DeclareMathOperator{\sa}{sa}
\DeclareMathOperator*{\median}{\ast} 

\DeclarePairedDelimiter{\ip}{\langle}{\rangle}
\DeclarePairedDelimiter{\norm}{\lVert}{\rVert}

\title{Large $n$-limit of matrix control problems and non-commutative controls}

\author{Wilfrid Gangbo$^1$}
\address{$^1$Department of Mathmatics, University of California, Los Angeles}
\email{\href{mailto:wgangbo@math.ucla.edu}{wgangbo@math.ucla.edu}}

\author{David Jekel$^2$}
\address{$^2$Department of Mathematical Sciences, University of Copenhagen}
\email{\href{mailto:daj@math.ku.dk}{daj@math.ku.dk}}

\author{Kyeongsik Nam$^3$}
\address{$^3$Department of Mathematical Sciences, Korea Advanced Institute of Science and Technology}
\email{\href{mailto:ksnam@kaist.ac.kr}{ksnam@kaist.ac.kr}}

\author{Aaron Z. Palmer$^4$}
\address{$^4$Department of Mathmatics, University of California, Los Angeles}
\email{\href{mailto:azp6@cornell.edu}{azp6@cornell.edu}}


\begin{document}
	
	\begin{abstract} 
	Building on the free‐probability stochastic control framework introduced in \cite{2025viscosity}, we connect optimal control problems for \(n\times n\) random‐matrix ensembles with their infinite‐dimensional, free‐probability analogues. Under natural convexity hypotheses, we prove that the non‐commutative value function captures the large-\(n\) limit of the corresponding finite-matrix control problems. As an application, we give a new perspective on the Laplace principle for convex functionals in the theory of large deviations for random matrices.
	\end{abstract}

	\maketitle

	\section{Introduction} \label{sec: introduction}

	Large random matrices naturally appear in diverse fields such as quantum mechanics and quantum chaos.
	In recent years, a growing body of work has highlighted the rich interplay between matrix optimization problems and control theory, particularly in applications such as the statistical mechanics of large complex systems. The study of finite-dimensional matrix control problems has revealed subtle structures that emerge in the large-size limit, motivating us to show that the finite \(n \times n\) matrix control problem converges, as \(n\to\infty\). More precisely, in finite dimensions, control problems are typically formulated in terms of optimizing a cost functional on the space of \(n \times n\) matrices; however, as \(n\) grows to infinity, classical tools may fail to capture emergent features such as non-local interactions and complex symmetry properties inherent in large random matrices. In this paper, by drawing on the abstract Hamilton--Jacobi framework from our earlier work in free probability setting \cite{2025viscosity}, we rigorously analyze the limit and show that the resulting problem is naturally set in a tracial von Neumann algebra---a non-commutative analog of classical measure spaces.

	The passage from finite-dimensional matrix control to its infinite-dimensional counterpart is not merely a technical convenience. In this limit, the empirical spectral distributions of the matrices, previously described by classical probability  measures when the matrices were self-adjoint, are replaced by non-commutative distributions. This transition necessitates a shift in perspective, from classical PDEs defined on Euclidean spaces   to the PDEs on spaces of non-commutative laws. These new types of equations not only encapsulate the dynamics of large random matrix models but also pave the way for various applications where the interplay of randomness, non-commutativity, and optimization comes into play.
	
	Our work thus addresses several key questions: How does the matrix control problem behave as the matrix dimension increases to infinity? What are the appropriate asymptotic limits, and how can they be characterized in the language of non-commutative analysis? By rigorously establishing that the finite-dimensional control problem converges to a well-defined Hamilton--Jacobi equation in the infinite-dimensional setting, we provide both a theoretical foundation and a set of analytical tools for studying large-scale matrix optimization.

	~

	Before formulating the control problem precisely, we recall fundamental results from random matrix theory and free probability.  Wigner’s semicircle law guarantees that the empirical spectral distribution  of \textup{GUE} converges to the semicircle law supported on \([-2,2]\). More generally, by the seminal work of Dan Voiculescu   \cite{voiculescu1991limit, voiculescu1998strengthened}, the \textup{GUE}-Brownian motion $(\widehat W^n_t)_{t \ge 0}$  converge in distribution to the \emph{free Brownian motion} (or \emph{free semicircular process}) $(S_t)_{t \ge 0}$  on a non-commutative probability space \((\mathcal{A}, \tau)\). Precisely, for any $t>0$ and a non-commutative polynomial \( p \),  as \( n \to \infty \),
	\[
	\frac{1}{n} \operatorname{tr}\big(p(\widehat W^n_t)\big) \to \tau\big(p(S_t)\big).
	\] 
	When considering a \(d\)-tuple of independent \textup{GUE} ensembles, Voiculescu’s asymptotic freeness theorem further guarantees joint convergence  to a freely independent semicircular family \((S^1,\dots,S^d)\). This collection of results  are manifestation of the concentration of measure, and   justifies replacing the classical randomness of the \textup{GUE}-Brownian motion by the deterministic structure of the free semicircular process in the large-\( n \) limit.

	The state space of our stochastic control problem will be 
	$M_n(\mathbb{C})_{\sa}^d,$
	the set of \(d\)-tuples of self-adjoint \(n \times n\) matrices. While the \(d=1\) case reduces to spectral measures, the multi-matrix setting (\(d>1\)) requires the language of \emph{non-commutative laws} (see \S \ref{subsec: von Neumann algebras}), since the joint distribution of non-commuting variables cannot be captured by a classical probability measure on $\bR^d$ (see \cite[\S 5.2.1]{anderson2010introduction}).  
	
	Let us now precisely describe the stochastic control problems.
	Our framework incorporates two sources of noise: the \emph{common noise} and the \emph{free individual noise}. The free individual noise is a \textup{GUE}-Brownian motion \( (\widehat{W}_t^n)_{t \ge 0}\) on the space of \(d\)-tuples of $n\times n$ self-adjoint matrices.   In contrast, the common noise is given by a single classical Brownian motion \((W_t^0)_{t \ge 0}\) which acts uniformly across the matrices. As \(n\) tends to infinity, while the free individual noise converges to free Brownian motion, the  nature of the common noise persists.
	Given two nonnegative parameters \(\beta_C\) and \(\beta_F\) governing the strength of the common and free individual noise respectively, we consider a time-dependent random control \((\alpha^n_t)_{t \ge 0}\) taking values in \(M_n(\mathbb{C})_{\sa}^d\) and the controlled dynamics
	\begin{equation}\label{eq:d1}
		dX_t^{n,j}[\alpha^n] =  \alpha_t^{n,j}\, dt + \beta_C\, \mathbf{1}_{M_n(\mathbb{C})}\, dW_t^0 + \beta_F\, d\widehat{W}^{n,j}_t,\quad j=1,\dots,d,
	\end{equation}
	with initial condition 
	\[
	X_{t_0}^{n}[\alpha^n] = x_0^n \in M_n(\mathbb{C})_{\sa}^d.
	\]
	Associated to these dynamics is a cost functional with running cost \(L_{M_n(\mathbb{C})}\) and terminal cost \(g_{M_n(\mathbb{C})}\). The value function is then defined as
	\begin{align}\label{eq:d2}
		\widehat {V}_{M_n(\mathbb{C})}&(t_0,x_0^n) \nonumber \\
		&:= \inf_{\alpha^n \in \widehat{\mathbb{A}}_{M_n(\mathbb{C})}^{t_0,T}} \left\{ \mathbb{E}\!\left[\int_{t_0}^T L_{M_n(\mathbb{C})}\big(X^n_t[\alpha^n],\alpha^n_t\big)\,dt + g_{M_n(\mathbb{C})}\big(X^n_T[\alpha^n]\big) \right] :\, X^n_{t_0}[\alpha^n] = x_0^n \right\},
	\end{align}
	where  $\widehat{\mathbb{A}}_{M_n(\mathbb{C})}^{t_0,T}$ denotes the collection of admissible controls.
	
	A central challenge addressed in this paper is to study the large-\(n\) limit of the value functions \(\widehat{V}_{M_n(\mathbb{C})}(t_0,x_0^n)\). As mentioned above, when $d >1$, the joint distribution of non-commuting variables is described by a {non-commutative law}. 
	More precisely, suppose that for each \(n\), the initial condition \(x_0^n\) is a \(d\)-tuple of $n\times n$ self-adjoint matrices with uniformly bounded operator norms, and that the corresponding non-commutative laws converge (in the weak-\(*\) topology). Then the natural question is: Do the value functions \(\widehat{V}_{M_n(\mathbb{C})}(t_0,x_0^n)\) converge as \(n\to\infty\), and if so, how can we characterize the limit? The answer is provided by our formulation in the language of free probability and tracial von Neumann algebras.
	
	In the infinite-dimensional analogue, we work with a tracial von Neumann algebra \(\mathcal{A} = (A,\tau)\). The limiting stochastic dynamics are then described by the equation
	\begin{equation}\label{eq:d1noncom}
		dX_t^j = \alpha_t^j\, dt + \beta_C\, \mathbf{1}_{\mathcal{A}}\, dW_t^0 + \beta_F\, dS_t^j,\quad j=1,\dots,d,
	\end{equation}
	with the initial condition \(X_{t_0}[\alpha] = x_0 \in L^2(\mathcal{A})_{\sa}^d\). Here, the free Brownian motion \((S_t^j)_{t \ge 0}\) captures the deterministic limiting behavior of the individual noise \( (\widehat{W}_t^{n,j})_{t \ge 0}\), while the common noise $(W_t^0)_{t \ge 0}$ remains a classical  Brownian motion.

	While our previous paper \cite{2025viscosity} set up and studied the stocastic control problems and value functions in the free probability setting, this paper will give a rigorous treatment of convergence for multi‐matrix stochastic control problems to the free limits. In particular, we show how the convergence of non‐commutative laws---replacing classical empirical spectral measures---yields a canonical infinite‐dimensional limit in the setting of tracial von Neumann algebras. 
	By bridging the finite and infinite-dimensional settings, our approach not only reinforces the deep connections between mean field games, random matrix theory, and control theory, but also opens up new avenues for studying stochastic control problems where non-commutative randomness shows up.

	\subsection{Organization}
	
	\begin{itemize}
		\item \S \ref{sec: setup} gives the details of the setup for the von Neumann algebraic and random matrix control problems, and states our main result.
		\item \S \ref{sec: preliminaries} reviews necessary ingredients for the proof of the main theorem, including properties of the value function from our previous paper \cite{2025viscosity} and a general form of Voiculescu's asymptotic freeness theorem.
		\item \S \ref{sec: discretization} introduces approximations for the von Neumann algebraic and matrix control problems obtained by discretizing the time domain and the classical probability space which the common noise inhabits.
		\item \S \ref{sec: convergence} proves the main theorem on convergence of the finite-dimensional value functions to the free limit under the assumption of $E$-convexity.
		\item \S \ref{sec: large deviations} applies the main result to a problem from large deviation theory.  Namely, we prove a Laplace principle for tuples of GUE matrices for the case of convex functionals (compare \cite{biane2003large,Dabrowski2017Laplace}).
		\item The appendices contain technical details for some of the background results, namely a proof of the general asymptotic freeness theorem (\S \ref{sec: asymptotic freeness proof}), estimates for discretization of the common noise (\S \ref{sec: normal estimates}), and the relationship between the matrix Laplacian and the free Laplacian on an important class of test functions (\S \ref{sec: Laplacians}).
	\end{itemize}

	\subsection{Acknowledgements}
	
	W.G. was supported by NSF grant DMS-2154578 and Air Force grant FA9550-18-1-0502. D.J. was partially supported by the National Sciences and Engineering Research Council (Canada) grant RGPIN-2017-05650, the Independent Research Fund of Denmark grant 1026-00371B, and the EU Horizon Marie Sk{\l}odowska-Curie Action FREEINFOGEOM, grant 101209517. K.N. was supported by the National Research Foundation of Korea (RS-2019-NR040050). A.Z.P. also acknowledges the support of Air Force grant FA9550-18-1-0502. We thank Dimitri Shlyakhtenko for numerous conversations that motivated this work. 
	
	\section{Setup and main results} \label{sec: setup}
	
	\subsection{Von Neumann Algebras} \label{subsec: von Neumann algebras}
	
	We briefly recall the
	framework of non-commutative probability spaces, which was also used in \cite{2025viscosity}. For a more thorough exposition, see for instance \cite[\S 5.2.3]{anderson2010introduction} and \cite[\S 2]{gangbo2022duality}.  A \emph{$W^*$--algebra} (or equivalently a von Neumann algebra) if is a unital 
	$C^*$--algebra $A$ together 
	with an operator norm $\|\cdot\|_\infty$ such 
	that $A$ as a 
	Banach space is a dual of some Banach space $A_*$.  A \emph{tracial $\mathrm{W}^*$-algebra} or \emph{non-commutative probability space} is a $\mathrm{W}^*$-algebra $A$ together with a faithful normal tracial state $\tau \in A_*$.  Here \emph{state} means that $\tau(x^*x) \geq 0$ and $\tau(1) = 1$; \emph{faithful} means that $\tau(x^*x) = 0$ if and only if $x = 0$; \emph{normal} means that $\tau \in A_*$ or $\tau$ is weak-$*$ continuous; \emph{tracial} means that $\tau(xy) = \tau(yx)$.  These properties are analogous to those of the expectation or integral as a functional on $L^\infty$ of a classical probability space.
	
	The GNS construction \cite{segal1947irreducible} produces a 
	Hilbert space $L^2(\cA)$ as follows:  
	A pre-inner product can 
	be defined on $\cA$ by
	$$
	\langle X,Y\rangle_{L^2(\cA)} := \tau(X^*\, Y) , \hbox{ for } X,Y\in L^2(\cA).
	$$
	Faithfulness of $\tau$ guarantees this inner 
	product is non-degenerate. Completing $\cA$ under 
	the induced norm gives the Hilbert 
	space $L^2(\cA)$; we continue to write 
	$\langle\cdot,\cdot\rangle_{L^2(\cA)}$ for its 
	extension. For $d$-tuples $X=(X^1,\dots,X^d)$ and $Y=(Y^1,\dots,Y^d)$ in $L^2(\cA)^d$,  
	the inner product is defined as
	$$
	\langle X,Y\rangle_{L^2(\cA)} := \sum_{j=1}^d\tau({X^j}^*\, Y^j),
	$$
	and the Hilbert space norm is defined as
	$$
	\|X\|_{L^2(\cA)} : = \sqrt{ \langle X,X\rangle_{L^2(\cA)}}.
	$$
	We write $L^\infty(\cA)\subset L^2(\cA)$ for the collection of elements that remain 
	bounded in the operator norm. Since $A\subset L^\infty(\cA)$, this also 
	induces a norm on tuples,
	$$
	\|X\|_\infty := \max_{j\in \{1,\ldots,d\}} \|X^j\|_\infty, \hbox{ for } X = (X^1,\cdots,X^d) \in L^\infty(\cA)^d.
	$$
	Let $\mathbf1_\cA$ denote the unit in $L^2(\cA)$, and write $L^2(\cA)^d_{\textup{sa}}$ for the self-adjoint 
	$d$-tuples. In particular, we write $\mathbbm1_\cA=(\mathbf1_\cA,\dots,\mathbf1_\cA)\in L^2(\cA)^d_{\textup{sa}}$, and each standard basis vector $\mathbf e^j_\cA\in L^2(\cA)^d_{\textup{sa}}$ has $\mathbf1_\cA$ in its $j$-th entry 
	and zero elsewhere.
	
	\subsection{Free products}
	
	Free independence and free products of tracial von Neumann algebras are crucial to this paper, since our stochastic optimization problems are defined using free Brownian motion.  Free independence is a non-commutative form of independence that describes the large-$n$ behavior of many random matrix models (see \S \ref{sec: asymptotic freeness} below).  It also generally provides a way to embed two given tracial von Neumann algebras into a larger one.  For background, we recommend \cite{voiculescu1992freerandom}, and further citations are given in \cite[\S 2.2]{2025viscosity} of our previous paper.
	
	We recall the definition quickly here.  If $\cA=(A, \tau)$ is a tracial $\mathrm{W}^*$--algebra and $\{\cA_j=(A_j,\tau): j\in J\}$ are tracial $\mathrm{W}^*$-subalgebras of $\cA$ with an index set $J$, we say that $\{\cA_j: j\in J\}$ are \emph{freely independent} if for all positive integers $n$ and $j:\{1,\ldots,n\}\rightarrow J$ such that $j(k)\not= j(k+1)$  for $k=1,\cdots,n-1,$
	$$
	\tau\bigg( \prod_{k=1}^{n}\big(a_k-\tau(a_k)\big)\bigg)=0, \qquad \hbox{for all }(a_1,\ldots,a_{n}) \in A_{j(1)}\times\ldots \times A_{j(n)},
	$$
	where the terms in the product are understood to be multiplied in order from left to right; see \cite{voiculescu1985symmetries,voiculescu1992freerandom}.  More generally, if $\{\cA_j: j \in J\}$ is a collection of tracial $\mathrm{W}^*$-algebras in $\cA$ containing a common subalgebra $\cB$ and if $E_{\cB}$ denote the trace-preserving conditional expectation $\cA \to \cB$, then we say that $\{\cA_j: j \in J\}$ are \emph{freely independent with amalgamation over $\cB$} or \emph{freely independent over $\cB$} if for all positive integers $n$ and $j:\{1,\ldots,n\}\rightarrow J$ such that $j(k)\not= j(k+1)$   for $k=1,\cdots,n-1,$ 
	$$
	E_{\cB} \bigg( \prod_{k=1}^n\big(a_k-E_{\cB}(a_k)\big)\bigg)=0, \qquad \hbox{for all }(a_1,\ldots,a_{n}) \in A_{j(1)}\times\ldots \times A_{j(n)}.
	$$
	
	Given any indexed family $(\cA_j)_{j \in J}$ and $\cB$ be tracial $\mathrm{W}^*$-algebras, and tracial $\mathrm{W}^*$-embeddings $\varphi_j: \cB \to \cA_j$, there exists a unique (up to isomorphism) free product von Neumann algebra generated by copies of $\cA_j$ that are freely independent over $\cB$.  More precisely (see e.g.\ \cite[Lemma 2.1]{2025viscosity}), there exists a tracial $\mathrm{W}^*$-algebra $\cC$ and tracial $W^*$--embeddings $\iota_j: \cA_j \to \cC$ such that
	\begin{enumerate}
		\item $\varphi = \iota_j \circ \varphi_j: \cB \to \cC$ is independent of $j$.
		\item The images $(\iota_j(\cA_j))_{j \in J}$ are freely independent with amalgamation over $\varphi(\cB)$.
		\item $\cC$ is generated by $(\iota_j(\cA_j))_{j \in J}$.
	\end{enumerate}
	Moreover, if $\widetilde{\cC}$ and $\widetilde{\iota}_j$ are another tracial $\mathrm{W}^*$-algebra and  tracial $W^*$--embeddings satisfying these properties, then there exists a unique isomorphism $\Phi: \cC \to \widetilde{\cC}$ such that $\Phi \circ \iota_j = \widetilde{\iota}_j$ for $j \in J$.
	
	\subsection{Non-commutative laws}\label{subsecNon-commutativeLaws} 
	
	As in our previous paper \cite{2025viscosity}, we need the notion of a ``law'' or ``joint distribution'' for $d$–tuples of non-commutative 
	random variables in a tracial von Neumann algebra, which is given for instance in \cite[\S 5.2]{anderson2010introduction}. For convenience, we fix throughout the paper a set $\bW$ of representatives 
	for each isomorphism class of separable-predual tracial $W^*$-algebras, so that every such algebra appears exactly once in $\bW$.
	
	We denote by  ${\rm NCP}_d:=\bC\langle x_1, \cdots, x_d \rangle$, the universal 
	unital algebra generated by variables $x_1 , \cdots , x_d$. 
	Let $\Sigma_{d,R}$ be 
	the set of linear functionals 
	$\lambda: {\rm NCP}_d\rightarrow \bC$  satisfying
	\begin{align*}
		\lambda(1)=1,\quad \lambda(pp^*)\geq 0,\quad  \lambda(pq)=\lambda(qp) \qquad \forall p,q\in {\rm NCP}_d
	\end{align*}
	that are $R$-exponentially 
	bounded, i.e. for any $k \in \mathbb{N}$ and a 
	monomial $\phi$ of degree $k$,
	$$
	|\lambda(\phi)| \leq R^k.
	$$
	We equip $\Sigma_{d,R}$ with the weak-$*$ topology as linear functionals on ${\rm NCP}_d$, which is metrizable since ${\rm NCP}_d$ has a countable basis and the evaluation of each polynomial $p$ is uniformly bounded for $\lambda$ in $\Sigma_{d,R}$.  For $\cA\in \bW$, there is a canonical map 
	$\lambda:\{X\in L^\infty(\cA)_{\textup{sa}}:\|X\|_\infty\leq R\}\rightarrow \Sigma_{d, R}$ 
	given by
	$$
	\lambda_X(p) := \tau\big(p(X)\big) \quad \hbox{ for }p\in {\rm NCP}_d.
	$$
	Conversely, every non-commutative law in $\Sigma_{d,R}$ arises as $\lambda_X$ for some $d$-tuple in some tracial von Neumann algebra (see \cite[Theorem 5.2.4]{anderson2010introduction}).  Let $\Sigma_d^\infty$ be the union of $\Sigma_{d,R}$ over 
	all $R>0$ with the 
	natural equivalence.
	
	In this paper, we work with self-adjoint elements in the non-commutative $L^2$ space $L^2(\cA)$ associated to a tracial von Neumann algebra $\cA$, and hence we want a notion of non-commutative laws $\Sigma_d^2$ of self-adjoint tuples in $L^2(\cA)$.  One approach to defining this is to complete $\Sigma_d^\infty$ in the $L^2$-Wasserstein metric.  The Wasserstein 
	metric is defined as follows: for $\lambda_1,\lambda_2\in \Sigma_{d,R}$,
	$$
	d_W^2(\lambda_1,\lambda_2) := \inf \Big\{\|X_1-X_2\|_{L^2(\cA)}^2: \cA\in \bW,\ X_1,X_2\in L^\infty(\cA)_{\textup{sa}}^d,\ {\lambda}_{X_1}=\lambda_1,\  {\lambda}_{X_2}=\lambda_2\Big\}.
	$$
	We define ${\Sigma}_d^2$ be the 
	closure of $\Sigma_d^\infty$ with respect to 
	the Wasserstein metric.  We showed in \cite[Lemma 2.3]{2025viscosity} that the elements of $\Sigma_d^2$ can always be realized by elements in $L^2(\cA)$ for some tracial von Neumann algebra $\cA$.
	
	Due to the fact that $\Sigma_d^2$ is not separable with respect to Wasserstein distance \cite[Theorem 1.8]{gangbo2022duality}, it will also be useful to have a suitable weak-$*$ topology on $\Sigma_d^2$.  Following \cite[\S A]{gangbo2022duality}, we define a map $\iota: \Sigma_d^2 \to \Sigma_{d,\pi/2}$ by sending the law of a  $d$-tuple $X = (X_1,\dots,X_d)$ of $L^2$ elements to the law of $\arctan(X) = (\arctan(X_1),\dots,\arctan(X_d))$, where $\arctan(X_j)$ is defined by continuous functional calculus.  We define the weak-$*$ topology on $\Sigma_d^2$ as the pullback via $\iota$ of the weak-$*$ topology on $\Sigma_{d,\pi/2}$ described above.  Equivalently, if $X^{(k)}$ is a self-adjoint $d$-tuple in $L^2(\cA_k)$ and $X$ is a self-adjoint $d$-tuple in $L^2(\cA)$, then the non-commutative laws $\lambda_{X^{(k)}}$ converge in the weak-$*$ topology of $\Sigma_d^2$ to $\lambda_X$ if and only if $\tau_{\cA_k}(p(\arctan(X^{(k)})) \to \tau_{\cA}(p(\arctan(X))$ for all $p \in {\rm NCP}_d$.  The test functions given by applying the trace to $p \circ \arctan$ are examples of the \emph{cylindrical test functions} that we discuss in \S \ref{sec: Laplacians}.  Note that the weak-$*$ topology is metrizable since for instance, we can list the monomials $(p_k)_{k \in \bN}$ and define the metric
	\begin{align} \label{metric}
		d(\lambda_1,\lambda_2) = \sum_{k \in \bN} \frac{1}{2^k (\pi/2)^{\deg(p_k)}} |\lambda_1(p_k \circ \arctan) - \lambda_2(p_k \circ \arctan)|.
	\end{align}

	We remark that for any $R>0$, the space $\Sigma_{d,R}$ can be viewed naturally as a subset of $\Sigma_d^2$ since every bounded $d$-tuple is an example of a $d$-tuple in $L^2$.  To see that the map $\Sigma_{d,R} \to \Sigma_d^2$ is injective, suppose that $X$ and $Y$ are $d$-tuples from $\cA$ and $\cB$ respectively.  Suppose $\tau_{\cA}(p \circ \arctan(X)) = \tau_{\cB}(p \circ \arctan(Y))$ for all $p \in {\rm NCP}_d$.  Note that $\norm{\arctan(X_j)} \leq \arctan(R) < \pi/2$.  Moreover, $\tan$ can be uniformly approximated by a sequence of polynomials $q_k$ on $[-\arctan(R),\arctan(R)]$, and so $p \circ q_k \circ \arctan(X) \to p(X)$ and $p \circ q_k \circ \arctan(Y) \to p(Y)$ in operator norm.  Hence, $\tau_{\cA}(p(X)) = \tau_{\cB}(p(Y))$.
	
	Furthermore, we claim that the weak-$*$ topology on $\Sigma_{d,R}$ coincides with the restriction of the weak-$*$ topology on $\Sigma_d^2$.  To see this, suppose that $\lambda_{X^{(k)}}$ converges to $\lambda_X$ in $\Sigma_{d,R}$.  To show convergence in $\Sigma_d^2$, consider a non-commutative polynomial $p$; in fact, by linearity, it suffices to consider the case where $p$ is a monomial.  For every $\varepsilon > 0$, there is a single-variable polynomial $q$ that approximates $\arctan$ within $\varepsilon$ uniformly on $[-R,R]$, and so $\norm{p \circ q(X^{(k)}) - p \circ \arctan(X^{(k)})} \leq \deg(p) \varepsilon$, and the same holds for $X$.  By assumption, $\tau_{\cA^{(k)}}(p \circ q(X^{(k)})) \to \tau_{\cA}(p \circ q(X))$, and so by taking $q$ arbitrarily close to $\arctan$ on $[-R,R]$, we obtain $\tau_{\cA^{(k)}}(p \circ \arctan(X^{(k)})) \to \tau_{\cA}(p \circ \arctan(X))$.  Since $\Sigma_{d,R}$ is compact and $\Sigma_d^2$ is Hausdorff and the inclusion map is injective, it is a homeomorphism onto its image.
	
	\subsection{Setup}
	
	We introduce a framework for control policies and cost functionals in the setting of general non-commutative stochastic optimization. 
	
	\subsubsection{Classical and free Brownian motions}
	Let $(\Omega,\mathcal{F},(\mathcal{F}_t)_{0\le t\le T},\mathbb{P})$ be a non-atomic,
	complete filtered probability 
	space supporting a standard Brownian motion
	$(W_t^0)_{t\in [0,T]}\in C([0,T]; L^2(\Omega, \mathcal{F},\bP))$:  
	\begin{itemize}
		\item[(a)]  $W_0^0=0$.
		\item[(b)]  For $0\leq s\leq t\leq T$, $W_t^0 - W_s^0$ 
		is normally distributed with 
		mean 0 and variance $t-s$.
		\item[(c)]  For any sequence of times  $0=t_0\le t_1 \le t_2 \le \cdots \le t_{k-1}  \le t_k= T,$ the 
		collection of increments  $W_{t_{j+1}}^0-W_{t_j}^0$ 
		for $j\in \{0,1,\ldots, k-1\}$ are (mutually) independent. 
	\end{itemize}
	We assume that  $\mathcal F= \sigma (W^0_s : 0\le s\le T)$ 
	and  $\mathcal F_t= \sigma (W^0_t : 0\le s\le t)$ for $t\in [0,T]$.

	In what follows, we identify scalar Brownian motions with their canonical embeddings into the von Neumann algebra by writing $\mathbf{1}_{\cA}\, W_t^0\in L^\infty(\cA)_{\textup{sa}}$ or $\mathbbm{1}_{\cA}\, W_t^0\in L^\infty(\cA)_{\textup{sa}}^d$.

	~

	Analogously to filtrations of $\sigma$-algebras, an increasing family of sub-von Neumann algebras $(\mathcal A_t)_{t\in [t_0,t_1]}$ in $\mathcal A\in\bW$ is called a \emph{free filtration}. Given $0\le t_0\le t_1\le T$ and a 
	free filtration $(\mathcal A_t)_{t\in [t_0,t_1]}$, 
	a $d$-dimensional process 
	$
	S = (S_t)_{t\in [t_0,t_1]} \in C\bigl([t_0,t_1];L^\infty(\mathcal A)_{\textup{sa}}^d\bigr)$
	is called a 
	\emph{free Brownian motion} (or \emph{free semicircular process}) compatible 
	with  $(\mathcal A_t)_{t\in [t_0,t_1]} $ if it 
	satisfies the following properties:

	\begin{itemize}
		\item[(a)] $S_{t_0} = 0$.
		
		\item[(b)] For $t_0\leq s\leq t\leq t_1$ and $l\in \{1,\ldots, d\}$, the 
		increment $S_t^l - S_s^l$ is semi-circularly 
		distributed with mean 0 and variance $t-s$, 
		and the components $\{S_t^l - S_s^l\}_{l=1}^d$ are freely independent.
		\item[(c)] $S_t\in L^\infty(\cA_t)_{\textup{sa}}^{d}$ for all $t\in [t_0,t_1]$. 
		\item[(d)] For $t_0 \leq s\leq t\leq t_1$,   $S_t - S_s$ is freely independent of $\cA_s$.
	\end{itemize} 
	We collect assumptions of our framework—specified in Sections \ref{sec 2.2.2}, 
	and \ref{sec:cost_functions}—which we call \textbf{Assumption A}.

	
	\subsubsection{Control policies} \label{sec 2.2.2}
	We assume that  for $\cA \in \bW$, controls in $\cA$  
	belong to some $\bA_\cA\subseteq L^2(\cA)_{\textup{sa}}^d$ which satisfies 
	\begin{itemize}
		\item[(a)] $\bA_\cA\subset L^2(\cA)_{\textup{sa}}^d$ is closed and convex. 
		\item[(b)]  $0 \in \bA_\cA$. 
		\item[(c)]  For any $\cB\in \bW$ and a tracial 
		$W^*$--embedding $\iota:\cA \to \cB$ (with its adjoint $E$), we have  
		$$\iota\, \bA_\cA \subset \bA_\cB \quad {\rm and} \quad E\, \bA_\cB \subset \bA_\cA.$$
	\end{itemize}
	Given $\cA\in \bW$, $[t_0, t_1] \subset [0, T]$,
	and $x_0\in L^2(\cA)_{\textup{sa}}^d$, we let  
	$\bA_{\cA,x_0}^{t_0,t_1}$ be the collection of 
	\emph{admissible} control policies 
	$$
	\widetilde{\alpha} = \Big((\alpha_t)_{t\in [t_0,t_1]},(\cA_t)_{t\in [t_0,t_1]}, (S_t)_{t\in [t_0,t_1]}\Big)
	$$
	that satisfy the 
	following properties:
	\begin{itemize}
		\item[(a)]  $(\cA_t)_{t\in [t_0,t_1]}$ is a free 
		filtration, i.e., an increasing sequence 
		of tracial $W^*$ subalgebras 
		of $\cA$, such that $x_0\in L^2(\cA_{t_0})_{\textup{sa}}^d$.
		\item[(b)]   $(S_t)_{t\in[t_0,t_1]}$ is a $d$-dimensional 
		free semi-circular 
		process compatible with 
		$(\cA_t)_{t\in [t_0,t_1]}$.
		\item[(c)]   For every $s\in [t_0,t_1]$, we have 
		$(\alpha_t)_{t\in [t_0,s]}\in L^2\big([t_0,s]\times (\Omega, \mathcal{F}_s,\mathbb{P}); \bA_{\cA_s}\big)$, 
		i.e., $(\alpha_t)_{t\in [t_0,t_1]}$ is progressively measurable 
		and freely progressive.  
	\end{itemize}

	\subsubsection{Cost functions}\label{sec:cost_functions}
	We consider a 
	running cost $L_{\cA}:L^2(\cA)_{\textup{sa}}^d\times \mathbb{A}_{\cA}\rightarrow \mathbb{R}$ and a 
	terminal cost $g_{\cA}:L^2(\cA)_{\textup{sa}}^d \rightarrow   \mathbb{R}$  
	satisfying the following assumptions.
	\begin{itemize}
		\item[(a)] Both $(L_{\cA})_{\cA\in \bW}$ and $(g_{\cA})_{\cA\in \bW}$ 
		are tracial $W^*$--functions. Equivalently, 
		$(L_{\cA})_{\cA\in \bW}$ may be defined on the 
		space of joint non-commutative laws in $\Sigma_{2d}^2$ and $(g_{\cA})_{\cA\in \bW}$ is 
		a function on the space 
		of non-commutative laws in $\Sigma_d^2$.
		\item[(b)]   $(L_\cA)_{\cA\in \bW}$ is \emph{$E$-convex} 
		in the control variable, 
		meaning that for any $\cA\in \bW$ and $X\in L^2(\cA)^d_{\textup{sa}}$, 
		$\alpha \mapsto L_\cA(X,\alpha)$ is convex, and  
		for any tracial $W^*$--embedding  $\iota:\cA\rightarrow \mathcal{B}$ with its adjoint 
		$E:L^2(\mathcal{B})_{\textup{sa}}^d\rightarrow L^2(\mathcal{A})_{\textup{sa}}^d$, 
		\begin{align}\label{eqn:E_convex}
			L_\cA( X, {E}\, \alpha) \leq  L_{\mathcal{B}}(\iota\, X, \alpha) \hbox{ for all }  \alpha\in\bA_\mathcal{B}.
		\end{align}
		\item[(c)] Similar to the drift,
		we assume that $(L_{\cA})_{\cA\in \bW}$ is uniformly 
		continuous on bounded sets.
		\item[(d)]  There exists a constant $C_1>0$ 
		such that for all $\cA\in \bW$, $X\in L^2(\cA)^d_{\textup{sa}}$ 
		and $\alpha \in \bA_{\cA}$, 
		\begin{align}\label{eqn:lower_and_upper_bounds}
			-C_1 +\frac{1}{C_1}\|\alpha\|_{L^2(\cA)}^2\leq&\  L_{\cA}(X,\alpha)\leq C_1\big(1+ \|X\|_{L^2(\cA)} + \|\alpha\|_{L^2(\cA)}^2\big),\\
			-C_1\leq&\  g_{\cA}(X)\leq C_1\big(1+ \|X\|_{L^2(\cA)}\big).\nonumber
		\end{align}
		Also  $(L_{\cA})_{\cA\in \bW}$ and 
		$ (g_{\cA})_{\cA\in \bW}$ are Lipschitz with 
		respect to $X$:  There exists a constant $C_2>0$ 
		such that for all $\cA\in \bW$, $X_1,X_2\in L^2(\cA)^d_{\textup{sa}}$ 
		and $\alpha \in \bA_{\cA}$, 
		\begin{align}\label{eqn:Lipschitz_bounds}
			|L_{\cA}(X_1,\alpha) -L_{\cA}(X_2,\alpha)| \leq&\ C_2\|X_1-X_2\|_{L^2(\cA)}, \\
			|g_{\cA}(X_1) -g_{\cA}(X_2)|\leq&\ C_2\|X_1-X_2\|_{L^2(\cA)}. \nonumber 
		\end{align}
	\end{itemize}

	\subsubsection{Assumption B} \label{sec: Assumption B}
	The following conditions are referred to as \textbf{Assumption B}. 
	\begin{enumerate}[label=(\alph*)] \item The control set is $\mathbb{A}_{\cA} = L^2(\cA)_{\textup{sa}}^d$. 
		\item The Lagrangian  $(L_{\cA})_{\cA\in \bW}$ is 
		jointly $E$-convex in 
		$(X,\alpha)$ and the terminal 
		cost $(g_{\cA})_{\cA\in \bW}$ is $E$-convex, in the sense that  for any $\cA\in \bW$, 
		the maps $(X,\alpha) \mapsto L_\cA(X,\alpha)$ and 
		$X\mapsto g_{\cA}(X)$ are convex, and  for any tracial 
		$W^*$--embedding  $\iota:\cA\rightarrow \mathcal{B}$ with its 
		adjoint $E:L^2(\mathcal{B})_{\textup{sa}}^d\rightarrow L^2(\mathcal{A})_{\textup{sa}}^d$, 
		\begin{align*}
			L_\cA( E\, X, {E}\, \alpha) \leq  L_{\mathcal{B}}(X, \alpha)& \hbox{ for all }  X\in L^2(\cB)_{\textup{sa}}^d \hbox{ and }\alpha\in\bA_\mathcal{B},\\
			g_\cA( E\, X) \leq  g_{\mathcal{B}}(X)& \hbox{ for all }  X\in L^2(\cB)_{\textup{sa}}^d.
		\end{align*}
		
	\end{enumerate}

	We now finally introduce an additional condition, \textbf{Assumption C},  which was not included in our previous work. This new assumption is essential for establishing the rigorous convergence of the value functions.
	
	\subsubsection{Assumption C}  \label{c}
	The following final  condition is refer to as \textbf{Assumption C.}
	
	
	Let $L_{\cA}$ be a Lagrangian of the form $$L_{\cA}(X,\alpha) = L_{\cA}^0(X,\alpha) + c \norm{\alpha}_{L^2(\cA)}^2,$$ where $c \geq 0$ and $L_0$ is $\kappa$-Lipschitz  ($\kappa>0$) with respect to $\norm{\cdot}_{L^1(\cA)}$ in both $X$ and $\alpha$, in the sense that  for any $X_1,\alpha_1,X_2,\alpha_2\in L^2(\cA)_{\textup{sa}}^d$,
	\begin{align}\label{eqn:L-1-Lipschitz}
		\big|L^0_{\cA}(X_1,\alpha_1) -L^0_{\cA}(X_2,\alpha_2)\big| \leq&\ \kappa \big(\|X_1-X_2\|_{L^1(\cA)} + \|\alpha_1-\alpha_2\|_{L^1(\cA)}\big).
	\end{align}
	We also assume that $(L^0_{\cA})_{\cA\in \mathbb{W}}$ is weak* continuous in the sense that if $(X^i,\alpha^i)\in L^\infty(\cA^i)_{\textup{sa}}^{2d}$ is a sequence converging in noncommutative law to $(X,\alpha)\in L^\infty(\cA)_{\textup{sa}}^{2d}$ as $i\rightarrow \infty$, then
	$$
	\lim_{i\rightarrow\infty} L^0_{\cA^i}(X^i,\alpha^i) = L^0_{\cA}(X,\alpha).
	$$
	
	\vspace{3mm}
	
	In summary, the assumptions in Sections \ref{sec 2.2.2}–\ref{sec:cost_functions} will be collected as \textbf{Assumption A}, those in Section \ref{sec: Assumption B} as \textbf{Assumption B}, and the condition of Section \ref{c} as \textbf{Assumption C}.

	\subsection{Value functions} \label{subsec: value function definitions}

	We now define the notion of  value functions.
	
	\subsubsection{Value functions on von Neumann algebras}
	
	Let $\beta_C\geq0$ and  $\beta_F\geq 0$ be 
	diffusion coefficients.
	Let $\cA\in \bW$ and $[t_0, t_1] \subset [0, T]$. For $x_0\in L^2(\cA)^d_{\textup{sa}}$ and $\widetilde{\alpha}\in \mathbb{A}_{\mathcal{A},x_0}^{t_0,t_1}$, we consider the SDE on  $ L^2(\cA)^d_{\textup{sa}}$:  
	\begin{align}\label{eqn:common_noise}
		\begin{cases} dX_t = \alpha_t\, dt +  \beta_C\, \mathbbm{1}_{\cA}\, dW_t^0 +  \beta_F\, dS_t,\\
			X_{t_0}=x_0.
		\end{cases}
	\end{align}  
	We denote by $X_t[t_0,x_0,\widetilde{\alpha}]$ the solution of (\ref{eqn:common_noise}) on $[t_0,t_1]$, or $X_t[\widetilde{\alpha}]$ when $(t_0,x_0)$ 
	are clear from the context. 
	Then we consider the value function on 
	$L^2(\cA)^d_{\textup{sa}}$, defined as follows: For $t_0 \in [0,T]$ and  $x_0\in L^2(\cA)_{\textup{sa}}^d$, 
	\begin{align}\label{eqn:l2_value}
		\widetilde{V}_{\cA}(t_0,x_0) := \inf_{\widetilde{\alpha} \in \bA_{\cA,x_0}^{t_0,T}}\Big\{\bE\Big[\int_{t_0}^T L_{\cA} (X_t[\widetilde{\alpha}], \alpha_t)dt + g_{\cA}(X_T[\widetilde{\alpha}])\Big]:  X_{t_0}[\widetilde\alpha]=x_0 \Big\},
	\end{align}
	where  $X_t[\widetilde \alpha]$ denotes a solution to \eqref{eqn:common_noise} and $\bE$ denotes expectation with respect to  the common noise. Note that if $\cA$ does not support a free Brownian motion $(S_t)_{t\in [t_0,T]}$ freely independent of $x_0$, i.e. if   $\bA_{\cA,x_0}^{t_0,T}$ is an empty set,  then this definition will result in $+\infty$.
	
	The eventual value function  is defined as follows: For $\cA \in \bW,$
	\begin{align} \label{def:bar}
		\overline{V}_{\cA}(t_0, x_0):= \inf_{\iota: \mathcal A \to \mathcal B} \widetilde{V}_{\mathcal B}(t_0, \iota\, x_0),
	\end{align}
	where the infimum is performed over the set of $(\cB, \iota)$ such that $\cB \in \bW$ and $\iota: \mathcal A \to \mathcal B$ is a tracial $W^*$--embedding.  In our previous paper \cite[Lemma 3.9]{2025viscosity}, it is shown that this value function can be viewed as a function on the space of non-commutative laws, i.e.\ the non-commutative Wasserstein space (see Lemma \ref{lem:nov18.2023.2} below). In other words, for $t_0 \in [0,T]$ and  $\lambda \in \Sigma_{d}^2,$ take  any $\cA \in \bW$ and $x_0\in \cA$ such that $\lambda_{x_0}=\lambda$, one can define the value function 
	\begin{align}\label{eqn:value}
		\overline{V}(t_0, \lambda) := \overline{V}_\cA(t_0, x_0).
	\end{align}

	\subsubsection{Value functions on finite-dimensional matrix algebras}
	
	Next, we  define the value function in the finite-dimensional matrix algebra setting. 
	The finite-dimensional matrix algebra corresponds to the case of $\cA = M_n(\bC)$.  
	We consider an orthonormal basis  $\{e_{\widehat{i}}\}_{\widehat{i}=1}^n$ of $\bR^n$  so that $\{ e_{\widehat{i}}\otimes e_{\widehat{j}}\}_{\widehat{i},\widehat{j}=1}^n$ form a basis of $M_n(\bC)$, where the inner product is given by the normalized trace $\tr_n$.
	We get a real orthonormal basis  of $M_n(\bC)_{\textup{sa}}$ for $\widehat{i}\in\{1,\ldots, n\}$ and $\widehat{j}\in \{1,\ldots, n\}$ as
	\begin{align}\label{eqn:n_basis}
		E_{\widehat{i}\widehat{j}} := \begin{cases}
			\sqrt{n}(e_{\widehat{i}} \otimes e_{\widehat{j}}) & \hbox{ if } \widehat{i} =\widehat{j},\\
			\frac{\sqrt{n}}{\sqrt{2}} \big(e_{\widehat{i}} \otimes e_{\widehat{j}}+e_{\widehat{j}}\otimes e_{\widehat{i}}\big)& \hbox{ if }  \widehat{i} < \widehat{j},\\
			\frac{{\rm i}\, \sqrt{n}}{\sqrt{2}} \big(e_{\widehat{i}} \otimes e_{\widehat{j}}- e_{\widehat{j}} \otimes e_{\widehat{i}}\big) &  \hbox{ if }  \widehat{i} > \widehat{j}.
		\end{cases}
	\end{align} 
	In order to approximate the free Brownian motion, we define $d$ many GUE($n$) Brownian motions to be comprised of $d N^2$ independent Brownian motions $(\widehat{W}^{\widehat{i}\widehat{j},l}_t)_t$ for $\widehat{i},\widehat{j}\in \{1,\ldots, n\}$ and $l\in \{1,\ldots d\}$ as 
	\begin{align} \label{gue}
		\widehat{W}^{n,l}_t := \frac{1}{n}\sum_{\widehat{i}=1}^n\sum_{\widehat{j}=1}^n  E_{\widehat{i}\widehat{j}} \widehat{W}^{\widehat{i}\widehat{j},l}_t,\qquad l\in \{1,\ldots, d\}.
	\end{align}
	We require that $(\widehat{W}_t^{n,l}: n  \ge 1, l \in \{1,\dots,d\}, t \in [0,T])$ is independent of the common noise $(W^0_t)_{t \in [0,T]}$.
	From now on, we consider the product probability space $   (\overline \Omega ,\overline \bP) =(\Omega ,\bP) \times (\widehat{\Omega}, \widehat{\bP})$, where the former one is w.r.t the common noise and the  latter one is w.r.t. the the GUE Brownian motion $(\widehat{W}_t^{n,l}: n  \ge 1, l \in \{1,\dots,d\}, t \in [0,T])$. We denote by $ \overline \bE, \bE, \widehat \bE $ the expectation w.r.t. the probability measures $\overline \bP, \bP, \widehat \bP$ respectively. Note that we do not require any coupling conditions on $\text{GUE}(n)$ for different $n$'s.

	When considering the value function $\widetilde{V}_{M_n(\bC)}(t_0,x_0)$ as in (\ref{eqn:l2_value}),  its  value will be $+\infty$ as the strategy set is empty since the von Neumann algebra does not support the free Brownian motions. 
	However, one can consider an approximation using the GUE($n$) Brownian motions defined above.  We denote this as follows: For  $t_0 \in [0,T]$ and $x_0^n \in M_n(\bC),$
	\begin{align}\label{valuematrix}
		\widehat{V}_{M_n(\bC)}(t_0,x_0^n) :=  \inf_{\widetilde{\alpha}^n \in \widehat{\bA}_{M_n(\bC)}^{t_0,T}}\Big\{\overline \bE\Big[\int_{t_0}^T L_{M_n(\bC)} ({X}^n_t[\widetilde{\alpha}^n], \alpha^n_t)dt + g_{M_n(\bC)}(X^n_T[\widetilde{\alpha}^n])\Big]\Big\},
	\end{align}
	where the modified strategy set $\widehat{\bA}_{M_n(\bC)}^{t_0,T}$ corresponds to the progressively measurable control policies,
	{i.e. denoting by $(\overline{\cF}^n_t)_t$ the filtration generated by the common noise and GUE($n$) Brownian motion, $(\alpha^n_t)_{t\in [t_0,s]}\in L^2\big([t_0,s]\times (\Omega,\overline{\cF}^n_s,\mathbb{P}); M_n(\bC) \big)$,} and  $(X_t^n[\widetilde{\alpha}^n])_t$ solves
	\begin{align}\label{eqn:finite_dynamics}
		\begin{cases} dX_t^{n,i} = \alpha^{n,i}_t\, dt + \beta_C\mathbf{1}_{M_n(\mathbb{C})}\, dW_t^0 + \beta_F\, d\widehat{W}^{n,i}_t,\ i\in \{1,\ldots,d\},\\
			X_0^n=x_0^n,
		\end{cases}
	\end{align}
	where $(\widehat{W}^{n,i}_t)_{t \ge 0}$ for $i=1,\cdots,d$ are GUE-Brownian motions defined in \eqref{gue}. Here, we used the assumption $b_\cA(x,\alpha)  = \alpha $ (see Assumption \textbf{B}). In other words, 
	$$
	X_t^{n,i}[\widetilde{\alpha}^n] = x_0^{n,i} + \int_{t_0}^t \alpha_s^{n,i}\, ds +\beta_C\, \mathbf{1}_{M_n(\mathbb{C})} (W_t^0-W_{t_0}^0) + \beta_F\, (\widehat{W}^{n,i}_t-\widehat{W}^{n,i}_{t_0}),\ i\in\{1,\ldots,d\}.
	$$

	\subsection{Main results}
	
	We now state the main result of this paper.  
	\begin{theorem}\label{main thm}
		Suppose that Assumptions \textbf{A}, \textbf{B} and \textbf{C} hold. For any sequence of $x_0^n\in M_n(\bC)_{\textup{sa}}^d$, such that operator norms are uniformly bounded in $n$ and the laws converge weakly* to $\lambda_0\in \Sigma_{d}^{2}$ as $n\rightarrow \infty$, we have
		\begin{align*}
			&\ \lim_{n\rightarrow \infty} \widehat{V}_{M_n(\bC)}(t_0,x_0^n) = \overline{V}(t_0,\lambda_0).
		\end{align*}
	\end{theorem}
	
	This theorem states that the  large-\(n\) limit of the finite-matrix control problems is described by the non‐commutative value function.
	
	\section{Preliminary results} \label{sec: preliminaries}
	
	Here we recall several ingredients needed for our main proof, namely properties of the value function from \cite{2025viscosity} and a general form of Voiculescu's asymptotic freeness theorem for random matrices (see \cite{voiculescu1991limit}).
	
	\subsection{Properties of the value functions}
	
	In this section, we briefly review the results obtained in the previous paper \cite{2025viscosity}.
	The following lemma states a tracial property of the value function $\overline V_\cA$. 
	\begin{lemma}[Lemma 3.9 in \cite{2025viscosity}]\label{lem:nov18.2023.2} Suppose that Assumption \textbf{A} holds. 
		Then for all $t_0 \in [0,T]$, $(\overline{V}_\cA(t_0,\cdot))_{\cA \in \bW}$ is a tracial $W^*$--function. 
	\end{lemma}

	The next lemma shows that the two definitions of the value function—namely \eqref{eqn:l2_value} and \eqref{def:bar}—coincide whenever $\cA$ admits a free Brownian motion that is freely independent of the initial condition.
	\begin{lemma}[Lemma 3.10 in \cite{2025viscosity}]\label{lem:decreasing}
		Suppose that Assumptions \textbf{A} and \textbf{B} hold. 
		
		1. {Let $\cA \in \mathbb{W}$, and let $x_0 \in L^2(\cA)_{\sa}^d$ and $t_0 \in [0,T]$.  Suppose that $\cA$ admits a $d$-variable free Brownian motion $(S_t^0)_{t \in [t_0,T]}$ freely independent of $\mathrm{W}^*(x_0)$. For  $t \in [t_0,T],$ let
			\[
			\cA_t^0 := \mathrm{W}^*(x_0, (S^0_s)_{s \in [t_0,t]}).
			\]
			Then we have
			\[
			\overline{V}_{\cA}(t_0,x_0) = \widetilde{V}_{\cA}(t_0,x_0),
			\]
			and the infimum in \eqref{eqn:l2_value} is witnessed by control policies $\widetilde{\alpha}\in \mathbb{A}_{\cA,x_0}^{t_0,T}$  that use the given filtration $(\cA_t^0)_{t \in [t_0,T]}$ and free Brownian motion $(S_t^0)_{t \in [t_0,T]}$.
			
			2. Let $\cA \in \mathbb{W}$. Let $\mathcal{C}$ be a tracial von Neumann algebra generated by a $d$-variable free Brownian motion $(S^1_t)_{t \in [t_0,T]}$, and define $\mathcal{C}_{t_0,t} := \mathrm{W}^*(S^1_s: s \in [t_0,t])$ for $t \in [t_0,T]$. Let $\iota_1: \cA \to \cA \median \cC$ and $\iota_2: \cC \to \cA * \cC$ be the inclusions associated to the free product.  Then for any  $x_0 \in L^2(\cA)_{\sa}^d$,
			\[
			\overline{V}_{\cA}(t_0,x_0) = \widetilde{V}_{\cA * \cC}(t_0,\iota_1(x_0)),
			\]
			and the infimum is witnessed by control polices that use the fixed filtration $(\iota_1(\cA) \vee \iota_2(\cC_{t_0,t}))_{t \in [t_0,T]}$ and free Brownian motion $(\iota_2(S^1_t))_{t \in [t_0,T]}$.
		} 
	\end{lemma}

	We remark that under Assumptions \textbf{A} and \textbf{B}, for $\epsilon \in (0,1),$ any $\epsilon$-optimal control policy $(\alpha_t)_t$ in \eqref{eqn:l2_value} satisfies a priori bound   
	\begin{align} \label{l2 bound}
		\bE\Big[\int_{t_0}^T \|\alpha_t\|_{L^2(\cA)}^2dt\Big] \leq   C( \|x_0\|_{L^2(\cA)}+ T+1) .
	\end{align}
	Indeed, for $\alpha_t  \equiv 0,$ the corresponding trajectory $(X_t)_t$ becomes $X_t=x_0+ \beta_C \mathbbm{1}_{\cA}(W^0_t-W^0_{t_0}) +\beta_F(S_t-S_{t_0}),$ implying  
	\[
	\bE[\|X_t\|_{L^2(\cA)}]  \leq \|x_0\|_{L^2(\cA)}+ (\beta_C+\beta_F) \sqrt{t}.
	\]
	By the upper bound condition on $\cL_\cA$ and $g_\cA$ in (\ref{eqn:lower_and_upper_bounds}),  we have $\widetilde{V}_{\cA}(t_0,x_0) \le   C( \|x_0\|_{L^2(\cA)}+ T)$. Hence  using the lower bound condition \eqref{eqn:lower_and_upper_bounds}, we deduce \eqref{l2 bound}.

	~

	The next observation, which will be used in \S \ref{sec: large deviations}, is that when there is no common noise, we do not need to consider classical randomness in our control process either.  In fact, most examples from free probability do not include a common noise (in fact, \cite[\S 5.2-5.3]{2025viscosity} considered two examples without common or individual noise), so the next lemma shows that for these cases, one does not need to perform the more complicated analysis of a classical filtration and non-commutative filtration at the same time that was used in our general theory in \cite{2025viscosity}.
	
	\begin{lemma} \label{lem: reducing to deterministic controls}
		Suppose that Assumptions \textbf{A} and \textbf{B} hold, and assume that $\beta_C = 0$, that is, the process $(X_t)_t$ in \eqref{eqn:common_noise} has no common noise term.  Then the value function $\widetilde{V}_{\cA}(t_0,x_0)$ is the infimum of
		\[
		\int_{t_0}^T L_{\cA} (X_t[\widetilde{\alpha}], \alpha_t)dt + g_{\cA}(X_T[\widetilde{\alpha}])
		\]
		over control processes $\alpha \in \bA_{\cA,x_0}^{t_0,T}$ that are also classically deterministic.
	\end{lemma}
	
	\begin{proof}
		Fix $\varepsilon > 0$ and let $\widetilde{\alpha} = ((\alpha_t)_{t \in [t_0,T]}, (\cA_t)_{t \in [t_0,T]}, (S_t)_{t \in [t_0,T]}) \in \bA_{\cA,x_0}^{t_0,T}$ such that  
		\[
		\mathbb{E} \left[ \int_{t_0}^T L_{\cA} (X_t[\widetilde{\alpha}], \alpha_t)dt + g_{\cA}(X_T[\widetilde{\alpha}]) \right] < \widetilde{V}_{\cA}(t_0,x_0) + \varepsilon.
		\]
		Let $\Omega$ denote the underlying classical probability space.  By the Fubini-Tonelli theorem, we know that for almost every sample $\omega$ in the probability space, the process $t \mapsto \alpha_t(\omega)$ is measurable and is adapted to the non-commutative filtration $(\mathcal{A}_t)_t$, and since $\beta_C = 0$ it satisfies $dX_t = \alpha_t\,dt + \beta_F dS_t$.  Hence, for almost every $\omega$, the deterministic process $t\mapsto \alpha_t(\omega)$ defines an element $\widetilde{\alpha}(\omega)$ of $\bA_{\cA,x_0}^{t_0,T}$.  Recall that for every integrable (real-valued) random variable $Z$, the probability that $Z \leq \mathbb{E}[Z]$ is strictly positive.  Therefore, there is a set of samples $\omega$ with positive probability such that
		\[
		\int_{t_0}^T L_{\cA} (X_t[\widetilde{\alpha}(\omega)], \alpha_t(\omega))dt + g_{\cA}(X_T[\widetilde{\alpha}(\omega)]) < \widetilde{V}_{\cA}(t_0,x_0) + \varepsilon.
		\]
		Hence, the resulting infimum is the same if we restrict to classically deterministic elements in $\bA_{\cA,x_0}^{t_0,T}$.
	\end{proof}
	
	\subsection{Asymptotic freeness} \label{sec: asymptotic freeness}
	
	In this section, we review Voiculescu’s theorem on the asymptotic freeness of GUE ensembles and related random matrices, and recall key properties of the GUE. We begin with Wigner’s semicircle law for GUE matrices (see, e.g., \cite[\S2]{anderson2010introduction}).
	
	\begin{lemma}
		Let $S^{(n)}$ be $n\times n$ \textup{GUE} matrix.  Then almost surely, for all polynomials $f$,
		\[
		\lim_{n \to \infty} \tr_n[f(S^{(n)})] = \frac{1}{2 \pi} \int_{-2}^2 f(x) \sqrt{4 - x^2}\,dx.
		\]
		Here,  $\tr_n$ denotes the normalized trace.
	\end{lemma}
	
	This immediately implies that $\liminf_{n\to\infty}\|S^{(n)}\|_\infty\ge 2.$
	However, the semicircle law alone does not exclude an  outlier eigenvalue above \(2\). 
	As for bounds on the operator norm, 
	the largest eigenvalue of the GUE matrix has been studied in depth, and in particular the largest eigenvalue converges to $2$ and its  fluctuations are described by the Tracy-Widom distribution (see e.g.\ \cite[\S 2.1.6, 2.6.2, 3.1.1]{anderson2010introduction}).  Indeed, Ledoux-Rider \cite{ledoux2010small} showed the following estimate for the operator norm:  For some constants $c, C > 0$,
	\[
	{{\mathbb P}}\big(\norm{S^{(n)}}_\infty \geq 2 + \epsilon\big) \leq C e^{-cn \epsilon^{3/2}},\qquad \forall \epsilon>0.
	\]
	A short argument based on concentration estimates combined with moment bounds was given by Ledoux \cite[equation (9)]{ledoux2003remark} (note that GUE normalization here   differs by a factor of $1/2$).  In particular, one can deduce from the Borel-Cantelli lemma the following.
	
	\begin{lemma} \label{lem: GUE operator norm}
		Let $S^{(n)}$ be a standard  $n\times n$ 
		\textup{GUE} matrix.  Then $\lim_{n \to \infty} \norm{S^{(n)}}_\infty = 2$ almost surely.
	\end{lemma}

	In order to state Voiculescu’s theorem on the asymptotic freeness of GUE ensembles (and more general random matrices), we first recall the non-commutative analog of classical independence.
	Let $\cA=(A, \tau)$ be a tracial 
	$\mathrm{W}^*$--algebra and $\{\cA_j=(A_j,\tau): j\in J\}$ are tracial $\mathrm{W}^*$-subalgebras of $\cA$ 
	with an index set $J$. We say that $\{\cA_j: j\in J\}$ are \emph{freely independent} 
	if for all positive integers $n$ and $j:\{1,\ldots,n\}\rightarrow J$ such 
	that $j(k)\not= j(k+1)$  for $k=1,\cdots,n-1,$
	$$
	\tau\bigg( \prod_{k=1}^{n}\big(a_k-\tau(a_k)\big)\bigg)=0, \qquad \hbox{for all }(a_1,\ldots,a_{n}) \in A_{j(1)}\times\ldots \times A_{j(n)}
	$$
	(see \cite{voiculescu1985symmetries,voiculescu1992freerandom}  
	for the references). 
	
	To state the results about asymptotic freeness of GUE matrices, it is helpful to formulate asymptotic freeness without any assumption on the existence of the limiting distribution of random matrices.
	
	\begin{definition}
		Let $J$ be a positive integer and $m: \{1,\cdots,J\} \rightarrow \mathbb{N}$ be any function.
		For $j = 1$, \dots, $J$, let $X_j^{(n)} = (X_{j,1}^{(n)},\dots,X_{j,m(j)}^{(n)})$ be a tuple of  $n\times n$ matrices.  We say that $X_1^{(n)}$, \dots, $X_J^{(n)}$ are \emph{asymptotically free} if for any $\ell \in \mathbb{N}$, indices $j_1 \neq j_2 \neq \dots \neq j_\ell$ in $\{1,\cdots,J\}$ and any polynomials $f_i$ in $m(j_i)$ variables (for $i = 1$, \dots, $\ell$), we have
		\[
		\lim_{n \to \infty} \tr_n\left[ (f_1(X_{j_1}^{(n)}) - \tr_n[f_1(X_{j_1}^{(n)})]) \dots (f_\ell(X_{j_\ell}^{(n)}) - \tr_n[f_\ell(X_{j_\ell}^{(n)})]) \right] = 0.
		\]
		When  $X_1^{(n)}$, \dots, $X_J^{(n)}$ 
		are \emph{random} matrix tuples (defined on some common  probability space), we say that they are \emph{almost surely asymptotically free} if the above limit holds almost surely.  
	\end{definition}

	For instance, for any positive integer $m$, independent $n\times n$ GUE matrices $S_1^{(n)}$, \dots, $S_m^{(n)}$ are almost surely asymptotically free (see  \cite[Theorem 5.4.2]{anderson2010introduction}). Note that the sequence $(S_1^{(n)}, \dots, S_{m}^{(n)})$ converges almost surely in the weak* topology of non-commutative laws, as $n\rightarrow \infty$.

	Now, we state the following useful result on the asymptotic freeness.
	
	\begin{theorem} \label{thm: asymptotic freeness}
		Let $m, m'$ be  positive integers. Let  $Y^{(n)} = (Y_1^{(n)}, \dots, Y_m^{(n)})$ be a $m$-tuple of $n\times n$ random self-adjoint matrices, and assume that $\sup_{n} \sup_{j=1,\cdots,m} \norm{Y_j^{(n)}}_\infty <\infty$ almost surely.  Let $S_1^{(n)}, \dots, S_{m'}^{(n)}$ be independent $n\times n$ GUE matrices, independent of $\{Y_j^{(n)}\}_{j=1,\cdots,m}$.  Then $Y^{(n)}$,  $S_1^{(n)}, \dots, S_{m'}^{(n)}$ are almost surely asymptotically free. 
		{Furthermore, if $Y^{(n)} = (Y_1^{(n)}, \dots, Y_m^{(n)})$ converge almost surely in the weak* topology of non-commutative laws, then the joint law of $(Y_1^{(n)}, \dots, Y_m^{(n)},S_1^{(n)}, \dots, S_{m'}^{(n)})$ also converges almost surely in the weak* topology.
		}
	\end{theorem}
	
	In the special case when the law of $Y^{(n)}$ converges almost surely, the above theorem is stated in \cite[Chapter 4, Theorem 5]{mingo2017free}. Also when  $Y^{(n)}$'s are deterministic, Theorem \ref{thm: asymptotic freeness} is stated in \cite{anderson2010introduction}.  However, similar to \cite[proof of Corollary 2.5]{voiculescu1998strengthened}, the assumption of a limit distribution does not make much difference in the proof.  The method of proof of Theorem \ref{thm: asymptotic freeness} is thus not much different than the other variants:  One can use the convergence of expectations together with concentration of measures.  We include a proof in the appendix \S \ref{sec: asymptotic freeness proof} for completeness.


	\section{Discretization of the control problem} \label{sec: discretization}
	
	In order to pass to the limit, it is useful to discretize the problem both in time and in the common noise.  This is necessary to gather all of the controls in a single von Neumann algebra, and must be done with uniform error over the $n$-dimensional problems.
	
	We fix two discretization parameters $N$ and $K$ for the common noise and time, as well as another parameter $R$ that restricts the operator norm of the control. When unambiguous, we do not include these parameters in the notation. We choose the simplest time discretization of $[t_0, T]$ with $t_{i}=t_0 + \frac{i}{K}(T-t_0)$ for $i\in \{0,1,\ldots,K\}$. Let
	\begin{align}\label{delta}
		\delta:=\frac{T-t_0}{K}
	\end{align}
	be the time increment. 
	

			Next we discretize the increments of the common noise Brownian motion, using equally spaced intervals for the increments. Let $N \ge 2$ be a positive integer.
			For $i\in \{1,\cdots,K\}$, define the events
			\[
			o_{i,j} := \Big\{ \frac{j}{N} < W^0_{t_i} - W^0_{t_{i-1}}  \le \frac{j+1}{N} \Big\}  \ \  \text{ for } j=-N,-N+1,\cdots,N-1
			\]
			and
			\[
			o_{i,-N-1}:=\Big\{W^0_{t_i} - W^0_{t_{i-1}}  \le -1 \Big\}, \qquad 
			o_{i,N}:=\Big\{W^0_{t_i} - W^0_{t_{i-1}}  > 1 \Big\}.
			\]

			For $i\in \{1,\cdots,K\}$, let $\cF^N_{i, \ast}$ be the $\sigma$-algebra generated by  $\big\{o_{i,j}\big\}_{j=-N-1}^{N}$, and  $\cF^N_{0, \ast}$ be the trivial $\sigma$-algebra. Since  $\big\{o_{i,j}\big\}_{j=-N-1}^{N}$ forms a partition of  the probability space  $\Omega$, every element of $\cF^N_{i, \ast}$ is a finite union of elements of  $\big\{o_{i,j}\big\}_{j=-N-1}^{N}$.
			Next, let $\mathcal{F}_{i}^N$ be the $\sigma$-algebra generated by $\cF^N_{0, \ast}, \cdots, \cF^N_{i, \ast}$. Note that $\{\mathcal{F}_{i}^N\}_{i=0,\cdots,K}$ forms a filtration.  We set $[N] := \{-N-1,\ldots, N\}$ so that $[N]^i$ denotes the space of multi-indices with $i$ components. Any element of $\mathcal{F}_{i}^N$ is a union of sets of the form $o_{1,m_1} \cap \cdots \cap o_{i, m_i}$, where $m_1,\cdots, m_i \in [N]$. 
			
			For $i\in \{1,\ldots, K\}$ and $J:=(j_1, \cdots, j_K)\in [N]^K$, define the event
			\[
			O_{i,J}:=\cap_{p=1}^i o_{p,j_{p}}.
			\]
			Note that if $J:=(j_1, \cdots, j_K)\in [N]^K$ and $J' = (j_1', \cdots, j_K')\in [N]^K$ satisfy that $j_p = j_p'$ for any $p\in \{1,\ldots, i\}$, then $O_{i,J} = O_{i,J'}$. Throughout this section, by abusing the notations, we will also use the notation $O_{i, J}$ for a multi-index $J\in [N]^i$ to denote the common value of $O_{i,\tilde{J}}$ when $\tilde{J}$ is an element in $[N]^K$ whose first $i$ coordinates are $J$.
			\begin{remark} \label{conditional}
				For $i \in \{1, \cdots, K\}$, as $\cF^N_{i, \ast}$ and $\cF^N_{i-1}$ are independent,  for any function $\Phi:\mathbb{R}\rightarrow \mathbb{R},$
				\[ \bE\big[ \Phi(W^0_{t_i}-W^0_{t_{i-1}})  \mid \cF^N_{i}\big] = \bE\big[ \Phi(W^0_{t_i}-W^0_{t_{i-1}}) \mid \cF^N_{i, \ast}\big] .
				\]
				Similarly, as future increments are independent of  $\cF^N_{i}$,
				\begin{align*}
					\bE\big[ \Phi(W^0_{t_i}-W^0_{t_{i-1}}) \mid 
					\mathcal{F}_K^N \big] =\bE\big[ \Phi(W^0_{t_i}-W^0_{t_{i-1}}) \mid  \cF^N_{i}\big].
				\end{align*}
			\end{remark}
			For any $i\in \{1,\ldots, K\}$ and $J = (j_1,\cdots,j_K)\in [N]^K$, we define
			\begin{align} \label{omega}
				\omega_{i,J} := \mathbb{E}\big[W_{t_i}^0 - W_{t_{i-1}}^0 \mid  O_{i,J}],
			\end{align}
			or equivalently the value  $
			\mathbb{E}\big[W_{t_i}^0 - W_{t_{i-1}}^0 \mid   \mathcal{F}_i^N] = \mathbb{E}\big[W_{t_i}^0 - W_{t_{i-1}}^0 \mid   \mathcal{F}_K^N]$
			on the event  $O_{i,J}.$ Note that for sufficiently small $\delta>0,$ for all $i$ and $J$,
			\begin{align} \label{605}
				|    \omega_{i,J}| \le 2. 
			\end{align}
			This is an easy consequence of Gaussian computation. Indeed, when $-N \le j_i \le N-1$, recalling the definition of $o_{i,j}$, we have $   |    \omega_{i,J}| \le \frac{1}{N}$. In the case $j_i=N,$ as $(W^0_{t_i}-W^0_{t_{i-1}}) / \sqrt{\delta}$ is distributed as the standard Gaussian which we call $Z$ (see \eqref{delta} for the definition of time increment $\delta$), using  Lemma \ref{gauss} in the appendix,
			\begin{align*}
				0\le    \omega_{i,J} = \mathbb E [ \sqrt{\delta}Z \mid  \sqrt{\delta}Z >1] = \sqrt{\delta}  \mathbb  E \Big [ Z  \ \big\vert \ Z >\frac{1}{\sqrt{\delta}}\Big]  \le \sqrt{\delta}  \cdot \frac{2}{\sqrt{\delta}}=2. 
			\end{align*}
			The case $j_i = -N-1$ similarly follows.


			We claim that 
			\begin{align} \label{6000}
				\bE [|W_{t_i}^0 - W_{t_{i-1}}^0 - w_{i,J} | \mid O_{i,J}] \le 
				\begin{cases}
					N^{-1}&\qquad j_i \in \{-N,\cdots,N-1\}, \\
					\sqrt{\delta} &\qquad j_i = -N-1 \text{ or } N.
				\end{cases}
			\end{align}
			The first case is obvious since the oscillation of $W^0_{t_i} - W^0_{t_{i-1}} $ is at most $N^{-1}.$ Considering the second case, by H\"older's inequality,
			\begin{align}\label{601}
				\bE [|W_{t_i}^0 - W_{t_{i-1}}^0 - w_{i,J} | \mid  O_{i,J} ] \le    \sqrt{\text{Var} (W_{t_i}^0 - W_{t_{i-1}}^0 \mid O_{i,J}  } )  .
			\end{align}
			By Remark \ref{conditional}, 
			\begin{align*}
				\text{Var} (W_{t_i}^0 - W_{t_{i-1}}^0 \mid O_{i,J})= \text{Var} (W_{t_i}^0 - W_{t_{i-1}}^0 \mid o_{i,j_i} ).
			\end{align*} 
			As $(W^0_{t_i}-W^0_{t_{i-1}}) / \sqrt{\delta}$ is distributed as the standard Gaussian,  for $j_i=-N-1$ or $j_i=N$,  
			\begin{align} \label{603}
				\text{Var} (W_{t_i}^0 - W_{t_{i-1}}^0 \mid o_{i,j_i} ) = \delta  \cdot \text{Var} \Big({W^0_{t_i}-W^0_{t_{i-1}} \over \sqrt \delta} \ \big\vert \ o_{i,j_i} \Big)   \le \delta ,
			\end{align}
			where we used  Lemma \ref{cond variance} (in Appendix) in the above inequality.
			This along with \eqref{601} proves \eqref{6000}.

	~

	
	Now, we define the discretized control space $\mathbb{A}_{\cA, t_0,x_0}^{K,N,R}$   as follows.

	{{
			%
			\begin{definition}\label{defn:july09.2024.1}
				Let $\cA \in \bW$, $N, K \in \bN$ and $R>0$. Given $t_0\in [0,T)$ and $x_0 \in L^2(\cA)^d_{\textup{sa}}$, we call 
				\[ 
				\widetilde {\alpha}^{\textsf{dis}}=\Big( (\alpha_{i,J})_{i,J},  {(\cA_{t_i})_{i=0}^K}, {(S_{t_{i}})_{i=0}^K} \Big)
				\]
				a discrete policy belonging to $\mathbb{A}_{\cA, t_0,x_0}^{K,N,R}$ if the following conditions hold:  
				\begin{enumerate}[label=(\alph*)]
					\item {$(\cA_{t_i})_{i=0}^K$ is a free filtration  in $\cA$ and $x_0 \in L^2(\cA_{t_0})$.}
					\item  {$(S_{t_{i}} - S_{t_{i-1}})_{i=1}^K$ is a free semcircular family, such that $S_{t_{i}} - S_{t_{i-1}} \in \mathcal{A}_{t_{i}}$ has a mean zero and variance $t_{i} - t_{i-1}$, freely independent of $\mathcal{A}_{t_{i-1}}$.}
					\item  $\alpha:   \{1, \cdots, K\} \times [N]^K \to L^\infty(\cA)^{d}_{\textup{sa}}$ is such that for every $i \in \{1, \cdots, K\}$ and $J=(j_1, \cdots, j_K) \in [N]^K$, we have $\alpha_{i,J} \in  L^\infty(\cA_{t_i})^{d}_{\textup{sa}}$. If $J'=(j_1', \cdots, j_K') \in [N]^K$ is such that $j_p=j_p'$ for all $p \in \{1, \cdots, i\},$ then $\alpha_{i,J} =\alpha_{i,J'}.$ Hence we will also write $\alpha_{i,J}$ for $J\in [N]^i$, which is well defined. Furthermore, $\sup_{i,J}\|\alpha_{i,J}\|_{L^\infty(\cA)}\leq R.$ We call $(\alpha_{i,J})_{i,J}$ a \emph{discretized control} which we identify with 
					\begin{equation}\label{eq:discret-controlWG}
						(t, w) \to  \alpha_t(w)=\sum_{i=1}^{K-1}  \mathbbm{1}_{(t_{i}, t_{i+1}]}(t) \sum_{J \in [N]^i}\alpha_{i,J} \mathbbm{1}_{O_{i,J}}(w).
					\end{equation}
					Observe that $  \alpha_{t_i} \in L^2(\cF_{t_i}^N)_{\textup{sa}}^d$ for every $i$.
				\end{enumerate}

				
				
				
				
			\end{definition}
			Note that (c) expresses an adaptedness property, and $\alpha_{i, J}$ is the control value on $O_{i, J}$.
	}}
	
	~
	
	We similarly define $W_{i, J}^0$ to be the discrete version of the Brownian motion, i.e., for  $i \in \{1, \cdots, K\}$  and  $J=(j_1, \cdots, j_i)\in [N]^i$, 
	\begin{align} \label{disbm}
		W^0_{i, J} := \sum_{i'=1}^i \omega_{i',J} ,
	\end{align}
	where $\omega_{i',J}$ is defined in \eqref{omega}.
	In this way, we treat the common noise as if it were deterministic. 
	The corresponding discretized dynamics in the von Neumann algebra $L^2(\cA)^d_{\textup{sa}}$ is defined by
	\begin{align}\label{eqn:discrete_dynamics}
		X_{i,J} = x_0 + \sum_{i'=1}^i \alpha_{i',J}\, \delta + \beta_C\, \mathbbm{1}_{\cA} W_{i,J}^0 + \beta_F\, (S_{t_i} -{{S_{t_0})}}.
	\end{align}
	We then define the discretized value function cost in a similar way to the continuous-time version by
	$$
	V^{K,N,R}_{\cA}(t_0,x_0) := \inf_{\alpha \in \mathbb{A}^{K,N,R}_{\cA,t_0,x_0}}\Bigg\{ \sum_{i=1}^K \sum_{J\in [N]^i} \bP(O_{i,J}) \, L_{\cA}(X_{i,J}, \alpha_{i,J})\, \delta + \sum_{J\in [N]^K}\bP(O_{K,J})\, g_{\cA}(X_{K,J})\Bigg\},
	$$
	Note that if $\cA \in \bW$ does not admit any {discrete free filtration $(\cA_{t_i})_{i=0}^K$} then by convention $V^{K,N,R}_{\cA}(t_0,x_0)=+\infty$. 
	
	{In order to compare the discrete and continuous problems, we first explain how to extend the discrete filtration into a continuous filtration on a potentially larger tracial von Neumann algebra.  In fact, we can even arrange that the conditional expectations with respect to the filtration on the larger algebra extend those from the discrete filtration (point (2) of the lemma below); while this is not needed for our argument in \S \ref{sec: convergence}, we include it for the sake of future use.
		
		\begin{lemma}  \label{lem: extend filtration}
			Let $(\cA_{t_i})_{i=0}^K$ be a discrete free filtration  in $\cA \in \bW$, and let $(S_{t_{i}})_{i=0}^K$ be a discretized semi-circular process satisfying Definition \ref{defn:july09.2024.1} (b).  Then there exist a tracial $W^*$-algebra $\cB \supseteq \cA$, a filtration $(\cB_t)_{t \in [t_0,T]}$ in $\cB$, and a semi-circular process $(Z_t)_{t \in [t_0,T]}$ compatible with this filtration such that
			\begin{enumerate}
				\item $Z_{t_i} = S_{t_i}$ for $i = 0$, \dots, $K$.
				\item For $a \in \cA$ and $i = 0$, \dots, $K$, we have $E_{\cA_{t_i}}[a] = E_{\cB_{t_i}}[a]$.
			\end{enumerate}
			In fact, we can take $\mathcal{B} = \mathcal{A} * \mathcal{C}$ where $\mathcal{C}$ is a tracial von Neumann algebra generated by an infinite free semi-circular family.
		\end{lemma}
		
		\begin{proof}
			Consider the algebra $\mathrm{W}^*(Z_t: t \in [t_0,T])$ generated by a free semi-circular Brownian motion.  View $\mathrm{W}^*(S_{t_i}: i = 0, \dots, K)$ as a subalgebra of $\mathrm{W}^*(Z_t: t \in [t_0,T])$ by identifying $S_{t_i}$ with $Z_{t_i}$, and then take the amalgamated free product
			\[
			\mathcal{B} := \cA \;\;\, \median_{\mathrm{W}^*(S_{t_i}: i = 0, \dots, K)}\;\;\, \mathrm{W}^*(Z_t: t \in [t_0,T]).
			\]
			Furthermore, let
			\[
			\mathcal{B}_t := \mathrm{W}^*(\mathcal{A}_{t_i}, Z_s: s \leq t) \text{ for } t \in [t_i,t_{i+1}).
			\]
			In order to prove the property (2), we first recall that there is a von Neumann algebra $\mathcal{C}_i$ such that
			\[
			\mathrm{W}^*(S_t - S_{t_i}: t \in [t_i,t_{i+1}]) = \mathrm{W}^*(S_{t_{i+1}} - S_{t_i}) * \mathcal{C}_i;
			\]
			this follows from the construction of the free Brownian motion using the free Gaussian process over a Hilbert space (which is analogous to the construction of classical Brownian motion using an orthonormal basis of $L^2[0,T]$); see \cite[\S 3]{voiculescu1985symmetries}, \cite[\S 2.6]{voiculescu1992freerandom}.  We thus have
			\[
			\mathcal{B} = \cA *_{\mathrm{W}^*(S_{t_i}: i = 0, \dots, K)} (\mathrm{W}^*(S_{t_i}: i = 0, \dots, K) * \mathcal{C}_1 * \dots * \mathcal{C}_K) \cong \cA * \mathcal{C}_1 * \dots * \mathcal{C}_K;
			\]
			this follows from \cite[Proposition 4.1]{houdayer2007freeproducts}.
			We also write $\mathcal{C} = \mathcal{C}_1 * \dots * \mathcal{C}_K$.
			
			Now fix $i$.  Since $\mathcal{A}$ is freely independent of the $\mathcal{C}_j$'s, we have that
			\[
			\mathcal{B}_{t_i} = \cA_{t_i} * \mathcal{C}_1 * \dots * \mathcal{C}_i
			\]
			Thus, for $a \in \cA$, we have
			\[
			E_{\cB_{t_i}}[a] = E_{\cA_{t_i} * \mathcal{C}_1 * \dots * \mathcal{C}_i}[a] = E_{\cA_{t_i}}[a];
			\]
			indeed, this follows because $a - E_{\cA_{t_i}}[a]$ is orthogonal to $\cA_{t_i} * \mathcal{C}_1 * \dots * \mathcal{C}_i$ by the proof of  \cite[Lemma C.4]{2025viscosity}.
		\end{proof}
	}

	{{
			Similarly, we introduce a discretized version of the random matrix control problem where the individual noise is GUE($n$) Brownian motion.  The set of control policies $\widehat{\mathbb{A}}_{M_n(\mathbb{C}),t_0}^{K,N,R}$ for this problem is defined as follows.
			\begin{definition}\label{defn:july09.2024.2} Let $N, K, n \in \bN$ and $R>0$. Given $ t_0 \in [0,T] $, we call 
				\[ 
				\widetilde \alpha^{\textsf{dis},n}=\Big( (\alpha_{i,J}^n)_{i,J},  ( \widehat{\mathcal{F}}^n_{i})_{i=1}^K \Big)
				\]
				a discrete policy belonging to $\widehat{\mathbb{A}}_{M_n(\mathbb{C}),t_0}^{K,N,R}$ if   the following conditions hold: 
				\begin{enumerate}[label=(\alph*)]
					\item Denoting by  $(\widehat{W}^n_t)_{t \ge 0}$  the GUE($n$) Brownian motion, for $i \in \{1,\cdots,K\},$ $\widehat{\mathcal{F}}^n_{i}$ is the $\sigma$--algebra  generated by  
					$\{\widehat{W}^n_{t_1}-\widehat{W}^n_{t_0}, \ldots, \widehat{W}^n_{t_i}-\widehat{W}^n_{t_{i-1}}\}$.
					\item 
					For $i \in \{1, \cdots, K\}$ and $J=(j_1, \cdots, j_K) \in [N]^K$, we have $\alpha^n_{i,J} \in  L^\infty\big(\widehat{\Omega}, \widehat{\mathcal{F}}^n_{{i}}, \widehat \bP; M_n (\bC)^d_{\textup{sa}}\big)$. If $J'=(j_1', \cdots, j_K') \in [N]^K$ is such that $j_p=j_p',$ for all $p \in \{1, \cdots, i\},$ then $\alpha^n_{i,J} =\alpha^n_{i,J'}.$ Furthermore, $\sup_{i,J}\|\alpha^n_{i,J}\|_{M_n(\mathbb{C})_{\textup{sa}}^d}\leq R$, $\widehat \bP$-a.s (here, $\|\cdot\|_{M_n(\mathbb{C})_{\textup{sa}}^d}$ denotes the maximum of $L^2$-norms in $d$ coordinates). 
				\end{enumerate}
			\end{definition}
	}}
	The discrete dynamics  in $M_n(\mathbb{C})_{\textup{sa}}^d$ is similarly then given by
	\begin{align}\label{eqn:discrete_dynamics_matrix}
		X_{i,J}^{n} = x_0^n + \sum_{i'=1}^i \alpha_{i',J}^{n}\, \delta + \beta_C\,\mathbbm{1}_{M_n(\mathbb{C})}  W_{i, J}^0 + \beta_F\, (\widehat{W}^n_{t_i}- \widehat{W}^n_{t_0}),
	\end{align}
	for  $i \in \{1, \cdots, K\}$  and  $J\in [N]^i$.
	
	~
	
	
	We then define the cost for the discretized random matrix problem as follows:  For  $x_0^n \in M_n(\mathbb{C})_{\textup{sa}}^d$,
	\begin{align*}
		\widehat{V}^{K,N,R}_{M_n(\mathbb{C})}(t_0,x_0^n) := \inf_{\alpha^n \in \widehat{\mathbb{A}}^{K,N,R}_{M_n(\mathbb{C}), t_0}}   \Bigg\{ \widehat {\mathbb{E}} \Big[ 
		\sum_{i=1}^K  \sum_{J\in [N]^i}   \bP(O_{i,J}) & \, L_{M_n(\mathbb{C})}(X_{i,J}^{n}, \alpha_{i,J}^{n})\, \delta  \\
		&+  \sum_{J\in [N]^K}   \bP(O_{K,J})\, g_{M_n(\mathbb{C})}(X^n_{K,J}) \Big] \Bigg\}.
	\end{align*}

	{{
			\begin{remark}\label{exo:july24.2024.1} The following facts will be useful later.
				\begin{enumerate}
					\item[(1)] If the discrete control in $\mathbb{A}_{\cA, t_0,x_0}^{K,N,R}$ takes null values, then  the corresponding discretized dynamics is 
					\[
					X_{i,J}= x_0 +  \beta_C\, \mathbbm{1}_{\cA} W_{i,J}^0 + \beta_F\, (S_{t_i} - S_{t_0}).
					\] 
					This can be used to deduce the following bound on the value function
					\begin{equation}\label{eq:july23.2024.9}
						V^{K,N,R}_{\cA}(t_0,x_0) \leq \Big(C_1+\|x_0\|_{L^2(\cA)}+ (\beta_C+\beta_F) \sqrt{T}\Big)(T+1).
					\end{equation}
					\item[(2)] If $(\alpha_{i,J})_{(i,J)}$  is a control in $\mathbb{A}_{\cA, t_0,x_0}^{K,N,R}$, then using \eqref{eqn:lower_and_upper_bounds}, 
					\begin{align}
						& \sum_{i=1}^{K}\sum_{J\in [N]^i} \bP(O_{i, J})\| \alpha_{i,J}\|^2_{L^2(\cA)}\delta -2C_1^2T \nonumber\\
						\leq &C_1\bigg( \sum_{i=1}^K\sum_{J\in [N]^i} \bP(O_{i, J}) \sum_{i=1}^K L_{\cA}(X_{i,J}, \alpha_{i,J})\delta + \sum_{J\in [N]^K} \bP(O_{K, J}) g_{\cA}(X_{K,J})\bigg) .\label{eq:july26.2024.2}
					\end{align}
					In particular, if $( \alpha_{i,J})_{(i,J)}$ is an $\epsilon$-optimal discrete control in $\mathbb{A}_{\cA, t_0,x_0}^{K,N,R}$ for $\epsilon>0$, then
					\begin{equation}\label{eq:july26.2024.2c}
						\sum_{i=1}^{K}\sum_{J\in [N]^i} \bP(O_{i, J})\| \alpha_{i,J}\|^2_{L^2(\cA)}\delta  \leq C_1\Big(\epsilon+2TC_1+ V^{K,N,R}_{\cA}(t_0,x_0) \Big).
					\end{equation}
					
				\end{enumerate}
			\end{remark}
			
	}}
	
	
	
	As the first step toward proving convergence for the discretization, the following lemma analyzes the operator-norm cutoff.
	
	\begin{lemma}[Uniform operator norm truncation]\label{lem:operator_norm_truncation}
		Let $L_{\cA}$ be a Lagrangian of the form $L_{\cA}(x,\alpha) = L_{\cA}^0(x,\alpha) + c \norm{\alpha}_{L^2(\cA)}^2$, where $c \geq 0$ and $L_0$ is {{$\kappa$}}-Lipschitz ($\kappa>0$) with respect to $\norm{\cdot}_{L^1(\cA)}$ in both $x$ and $\alpha$, i.e., satisfying Assumption \textbf{C}. For $R>0$, let $\phi_R(t) := \max(\min(t,R),-R)$.  Then for any continuous $Y = (Y_t)_t: [0,T] \rightarrow L^2(\cA)$ and $\alpha = (\alpha_t)_t \in L^2([0,T],L^2(\cA))$,
		\begin{align*}
			\int_0^T L_{\cA}\Big(Y_t + \int_0^t \phi_R(&\alpha_s)\,ds,  \phi_R(\alpha_t)\Big)\,dt \\
			&\leq \int_0^T L_{\cA}\Big(Y_t + \int_0^t \alpha_s\,ds, \alpha_t \Big) \,dt +{{(1+T) \kappa \over R}} \int_0^T \norm{\alpha_t}_{L^2(\cA)}^2\,dt.    
		\end{align*}
		
	\end{lemma}
	
	\begin{proof}
		Let $y$ be any self-adjoint element in $\cA$ and $\mu$ be its spectral measure with respect to the trace.  Note that
		\[
		|s - \phi_R(s)| \leq \mathbf{1}_{\bR \setminus (-R,R)}(s) |s|.
		\]
		Thus,
		\[
		R \norm{y - {\phi_R(y) } }_{L^1(\cA)} = R \int_{\bR} |s - \phi_R(s)|\,d\mu(s) \leq \int_{\bR \setminus (-R,R)} |s|^2\,d\mu(s) \leq \norm{y}_{L^2(\cA)}^2,
		\]
		and so
		\[
		\norm{y - \phi_R(y)}_{L^1(\cA)} \leq \frac{\norm{y}_{L^2(\cA)}^2}{R}.
		\]
		Similar, and easier, reasoning shows that $\norm{\phi_R(y)}_{L^2(\cA)} \leq \norm{y}_{L^2(\cA)}$.  Now applying this to the process $(\alpha_t)_t$,  
		\[
		\int_0^T \norm{\alpha_t - \phi_R(\alpha_t)}_{L^1(\cA)}\,dt \leq \int_0^T \frac{1}{R} \norm{\alpha_t}_{L^2(\cA)}^2\,dt.
		\]
		In particular, letting $A_t := \int_0^t \alpha_s\,ds$ and $\widetilde{A}_t := \int_0^t \phi_R(\alpha_s)\,ds$,  
		\[
		\norm{A_t - \widetilde{A}_t}_{L^1(\cA)} \leq \int_0^t \norm{\alpha_s - \phi_R(\alpha_s)}_{{L^1(\cA)}} \le  \frac{1}{R} \int_0^t \norm{\alpha_s}_{L^2(\cA)}^2\,ds \le  \frac{1}{R} \int_0^T \norm{\alpha_s}_{L^2(\cA)}^2\,ds.
		\]
		Since $L_{\cA}^0$ is Lipschitz in $L^1(\cA)$, for any $t\in [0,T],$
		\begin{align*}
			L_{\cA}^0(Y_t + \widetilde{A}_t, \phi_R(\alpha_t)) &\leq L_{\cA}^0(Y_t + A_t, \alpha_t) + {{\kappa}} \norm{A_t - \widetilde{A}_t}_{L^1(\cA)} + {{\kappa}} \norm{\alpha_t - \phi_R(\alpha_t)}_{L^1(\cA)} \\
			&\leq L_{\cA}^0(Y_t + A_t, \alpha_t) + {{ {\kappa\over R}}} \int_0^T \norm{\alpha_s}_{L^2(\cA)}^2\,ds +  {{{\kappa\over R}}}
			\norm{\alpha_t}_{L^2(\cA)}^2.
		\end{align*}
		Integrating this over $t \in [0,T]$, we  obtain the asserted statement when $c = 0$.  In the case when $c > 0$, we combine the above argument with the fact that
		\[
		\norm{\phi_R(\alpha_t)}_{L^2(\cA)}^2 \leq \norm{\alpha_t}_{L^2(\cA)}^2.
		\]
	\end{proof}

	Now, we establish the following key proposition, which states that the discretized optimization problem approximates the original value function. Recall that for $\cA \in \bW$, $\widetilde{V}_{\cA}$ denotes the value function defined in \eqref{eqn:l2_value}. 
	\begin{proposition}\label{prop:discretization}
		Suppose that Assumptions \textbf{A}, \textbf{B} and \textbf{C} hold.
		Then there exists a constant $C>0$ such that the following holds: For all $t_0\in [0,T]$,  $K, N \in \bN$, $R>0$, $\cA \in \bW$ containing a free semi-circular process $(S_t)_{t\in [t_0,T]}$ compatible with a free filtration
		$(\cA_t)_{t \in [t_0, T]}$ and $x_0\in L^\infty(\cA)_{\textup{sa}}^d$,
		\begin{align} \label{311}
			{V}^{K,N,R}_{\cA}(t_0,x_0)\leq \widetilde{V}_{\cA}(t_0,x_0)+ C \Big( \frac{1}{\sqrt{K}} + \frac{K}{N} + \frac{1}{R} \Big).
		\end{align}
		Also for {any $\cA \in \bW$ containing  a discretized semi-circular process,} 
		denoting by $\mathcal{C}$ a tracial von Neumann algebra generated by an infinite free semi-circular family (see Lemma \ref{lem: extend filtration}),
		\begin{align} \label{312}
			V^{K,N,R}_{\cA}(t_0,x_0) \geq \widetilde{V}_{\cA * \mathcal{C}}(t_0,x_0) - C \Big( \frac{1}{\sqrt{K}} + \frac{K}{N} + \frac{1}{R} \Big).
		\end{align}
		Here, $C$ is an increasing function of $\|x_0\|_{L^2(\cA)},$ $T$ and $\beta_F+\beta_C$.
		
		
		Similarly, for the finite-dimensional matrix problem, for all $x_0^n\in M_n(\bC)_{\textup{sa}}^d$,
		\begin{align} \label{313}
			\bigg| \widehat{V}^{K,N,R}_{M_n(\mathbb{C})}(t_0,x_0^n) - \widehat V_{M_n(\mathbb{C})}(t_0, x_0^n)\bigg|  \leq  C \Big( \frac{1}{\sqrt{K}} + \frac{K}{N} + \frac{1}{R} \Big).
		\end{align}
		Here, $ C $ is  increasing in $\|x_0^n\|_{L^2(M_n(\mathbb{C}))},$ $T$ and $\beta_F+\beta_C$, \emph{not} depending on the matrix size $n$.
	\end{proposition}

	This proposition and Lemma \ref{lem:decreasing}, along with the fact that $(\overline V_\cA)_{\cA \in \bW}$ is a tracial function, imply that for any $\cA \in \bW$ containing a free-semi-circular process {freely independent of $x_0$,}
	\begin{align}\label{629}
		\bigg| \overline{V}_{\mathcal A}(t_0, x_0) - 
		\inf_{\iota: \mathcal A \to \mathcal B}
		V^{K,N,R}_{\cB}(t_0,\iota x_0)\bigg| \leq C \Big( \frac{1}{\sqrt{K}} + \frac{K}{N} + \frac{1}{R} \Big).
	\end{align}

	\begin{proof}
		\textbf{Von Neumann algebra case: Proof of \eqref{311}.} 
		For $\epsilon>0$, let $\widetilde{\alpha}\in \mathbb{A}_{\cA,x_0}^{t_0,T}$ denote an $\epsilon$-optimal control for the $\widetilde{V}_{\cA}(t_0,x_0)$. By Lemma \ref{lem:operator_norm_truncation} along with a bound on the $L^2$-norm of ${\alpha}$ (see \eqref{l2 bound}),   we may take the control $\alpha$ to have operator norm bounded by $R$ with a penalty of $\frac{C}{R}$ in the cost. In other words, $\sup_t \norm{\alpha_t}_{L^2(\cA)} \le R$ and
		\begin{align}\label{511}
			\bE\Big[\int_{t_0}^T L_{\cA} (X_t[\widetilde{\alpha}], \alpha_t)dt + g_{\cA}(X_T[\widetilde{\alpha}])\Big]  \le \widetilde{V}_{\cA}(t_0,x_0) + \frac{C}{R} + \epsilon. 
		\end{align}
		We then define a discretized control 
		$(\alpha_{i,J})_{(i,J)}$ (see Definition \ref{defn:july09.2024.1})  by   
		\begin{align} \label{614}
			\alpha_{i,J} := \frac{1}{\delta}\int_{t_{i-1}}^{t_{{i}}} \mathbb{E}\big[ \alpha_t\mid O_{i,J}\big] dt .
		\end{align}
		Note that  
		this control is admissible for the discrete problem. 
		Also, for $i' \le i$,  
		$$
		{{\int_{t_{i'-1}}^{t_{{i'}}} \mathbb{E}\big[ \alpha_t\mid     O_{i,J}\big]dt}} = \int_{t_{i'-1}}^{t_{{i'}}} \mathbb{E}\big[ \alpha_t\mid    O_{i',J}\big]dt = \delta\, \alpha_{i',J}.
		$$
		This is because for $i'<i$, $(\alpha_t)_{t\in [t_{i'-1},t_{i'}]}$ is independent of $W^0_{t_{i}}-W^0_{t_{i-1}}$. Hence, setting $(X_{i,J})_{i,J}$ to be the corresponding solution of the control $(\alpha_{i,J})_{(i,J)}$ (see \eqref{eqn:discrete_dynamics}),
		\begin{align} \label{612}
			\mathbb{E}\big[ X_{t_i}\mid O_{i,J}]& =  x_0 + \sum_{i'=1}^i \mathbb{E}\Big[ \int_{t_{i'-1}}^{t_{i
					'}}\alpha_{t}\mid O_{i,J}\Big]dt + \beta_C\, \mathbbm{1}_{\cA} \sum_{i'=1}^i 
			\mathbb{E}[W_{t_{i'}}^0 - W_{t_{i'-1}}^0  \mid O_{i,J}] + \beta_F(S_{t_i}-S_{t_0}) \nonumber  \\
			&=  x_0 + \delta \sum_{i'=1}^i \alpha_{i',J}  + \beta_C\, \mathbbm{1}_{\cA} W_{i,J}^0 + \beta_F\, (S_{t_i} -{{S_{t_0})}} = X_{i,J}.
		\end{align}
		Then by Jensen's inequality for the conditional expectation (recall that $L_\cA$ is jointly convex, see Assumption \textbf{B}), 
		\begin{align*}
			V^{K,N,R}_{\cA}&(t_0,x_0)\leq \sum_{i=1}^K \sum_{J\in [N]^i}\bP(O_{i,J})L_{\cA}(X_{i,J}, \alpha_{i,J})\delta + \sum_{J\in [N]^K}\bP(O_{K,J}) g_{\cA}(X_{K,J})\\
			&\overset{\eqref{612}}{=}  \ \bE \Big[ \sum_{i=1}^K \sum_{J\in [N]^i}
			L_{\cA}( \mathbb{E}\big[ X_{t_i}| O_{i,J}], \alpha_{i,J})  \delta  \cdot \mathbbm{1}_{O_{i,J}}  \Big] + \bE \Big[ \sum_{J\in [N]^K} g_{\cA}(  \mathbb{E}\big[ X_{T}\mid O_{K,J}] ) \mathbbm{1}_{O_{i,J}}  \Big] \\
			&\overset{\eqref{614}}{\le}    \ \bE \Big[ \sum_{i=1}^K\sum_{J\in [N]^i} \int_{t_{i-1}}^{t_i} 
			\mathbb{E} [ L_{\cA}( X_{t_i} , \alpha_{t})  | O_{i,J} ]  dt \cdot \mathbbm{1}_{O_{i,J}}  \Big] + \bE \Big[ \sum_{J\in [N]^K} \mathbb{E} [g_{\cA} ( X_{T} ) | O_{K,J}]   \mathbbm{1}_{O_{i,J}}  \Big] \\
			&=  \ \mathbb{E}\bigg[\sum_{i=1}^K \int_{t_{i-1}}^{t_i} L_{\cA}(X_{t_i}, \alpha_t)dt + g_{\cA}(X_{T})\bigg]\\
			&\overset{\eqref{511}}{\le} \widetilde{V}_{\cA}(t_0,x_0)+\frac{C}{R} + \epsilon +\mathbb{E}\Bigg[\sum_{i=1}^K \int_{t_{i-1}}^{t_i}\big| L_{\cA}(X_{t_i}, \alpha_t) -  L_{\cA}(X_{t}, \alpha_t) \big|dt\Bigg].
		\end{align*}
		We control the last term. Note that for $t\in [t_{i-1},t_i],$
		\begin{align} \label{321}
			\|X_{t_i} -  X_{t} \|_{L^2(\cA)} \le C \Big( \int_{t}^{t_i} \|\alpha_s  \|_{L^2(\cA)} ds + \beta_C |W_{t_i}^0  -W_{t}^0| + \beta_F \sqrt{t_i-t}\Big).
		\end{align}
		Taking the expectation, noting that $0\le t_i - t \le \delta,$
		\begin{align}
			\bE [\|X_{t_i} -  X_{t} \|_{L^2(\cA)}] \le C \Big( \int_{t}^{t_i} \bE \|\alpha_s  \|_{L^2(\cA)} ds + \sqrt{\delta}\Big).
		\end{align}
		Thus, using the Lipschitz condition on ${L}_\cA$ (see Assumption \textbf{A}),
		\begin{align} \label{632}
			\mathbb{E}\Bigg[\sum_{i=1}^K \int_{t_{i-1}}^{t_i}\big| L_{\cA}(X_{t_i}, \alpha_t) -  L_{\cA}(X_{t}, \alpha_t) \big|dt\Bigg] & \le C \mathbb{E}\Bigg[\sum_{i=1}^K \int_{t_{i-1}}^{t_i} \|X_{t_i} -  X_{t} \|_{L^2(\cA)} dt\Bigg] 
			\nonumber  \\
			&\le  C  \sum_{i=1}^K \int_{t_{i-1}}^{t_i} \Big(\int_{t}^{t_i} \mathbb{E} \|\alpha_s  \|_{L^2(\cA)} ds +\sqrt{\delta} \Big) dt \nonumber  \\
			&\le C  \delta \int_{t_0}^T  \mathbb{E}  \|\alpha_s  \|_{L^2(\cA)} ds + C T \sqrt{\delta} \nonumber  \\
			&\overset{\eqref{l2 bound}}{\le} C  \delta 
			( \| x_0\|_{L^2(\cA)}+T+1)+  C T \sqrt{\delta} \le C \frac{1}{\sqrt{K}} .
		\end{align}
		Therefore, as $\epsilon>0$ is arbitrary, we obtain the upper bound.

		~

		\textbf{Von Neumann algebra case: Proof of \eqref{312}.} For $\epsilon>0$, we similarly choose an $\epsilon$-optimal discrete control $ (\alpha_{i,J})_{i,J}$ in $\mathbb{A}_{\cA, t_0,x_0}^{K,N,R}$. Then consider the enlarged algebra $\cA * \cC$, extending the discretized semi-circular process $(S_{t_{i}})_{i=0}^K$ to the continuous semi-circular process  $(S_t)_{t \ge t_0} $ (see Lemma \ref{lem: extend filtration}).
		
		We extend  $ (\alpha_{i,J})_{i,J}$ to an admissible control for the continuous-time problem, with a shift in time step to satisfy the measurability requirement.  Precisely, we set $\alpha_t := 0$ for $t\in [t_0,t_1]$, and for $t\in (t_i,t_{i+1}]$ with $i \in \{1,\cdots,K-1\}$,  
		$$
		\alpha_t(\omega) := \alpha_{i,J}  \ \hbox{ where $J \in [N]^i$ is a (unique) multi-index such that }\omega\in O_{i+1,J},
		$$
		{{which is a discretized control defined in \eqref{eq:discret-controlWG}}.} We regard $\alpha_t(\omega)$ as an element in  $L^2 (\cA * \cC)_{\textup{sa}}^d$.
		Let $X_t$ be the corresponding solution in $L^2 (\cA * \cC)_{\textup{sa}}^d$. Then  for $t \in [t_0,t_1],$  
		\begin{align} \label{626}
			X_t = x_0 +\beta_C \mathbbm{1}_{\cA * \cC} \big(W^0_{t} -W^0_{t_{0}}\big)  +  \beta_F (S_t - S_{t_0}).
		\end{align}
		Next,
		for any $i\in \{1,\cdots,K-1\}$, $t\in (t_{i},t_{i+1}]$ and a multi-index $J \in [N]^K$,
		\begin{align*}
			X_{t} - X_{i,J} = (t-t_i-\delta) \alpha_{i,J}+ {{ \beta_C \mathbbm{1}_{\cA * \cC} \big(W^0_{t} -W^0_{t_{0}}- W^0_{i,J}\big)}}  +  \beta_F (S_t - S_{t_i}) \qquad \text{on $O_{i+1,J}$}.
		\end{align*}
		We claim that for any   $t\in (t_{i},t_{i+1}]$,  on the event $O_{i+1,J}$,
		\begin{align} \label{611}
			\mathbb{E}& \Big[\|X_{t} - X_{i,J}\|_{L^2(\cA  * \cC)} \mid \mathcal{F}_K^N \Big]  \nonumber \\
			&\leq  (\delta - (t-t_i)) \|\alpha_{i,J} \|_{L^2(\cA  * \cC)}   +   
			C \beta_C \Big (     \zeta_{i,J} + \sqrt{\delta}  +
			\bE [ |W_{t_{i+1}}^0 -W_{t_{i}}^0 |\mid  \mathcal{F}_K^N ]  \Big) + C\sqrt{\delta} 
			\beta_F,
		\end{align}
		where for $J = (j_1,\cdots,j_K),$  
		\begin{align} \label{zeta}
			\zeta_{i,J}:=  \begin{cases}
				\frac{K}{N}&\qquad  j_1,j_2,\cdots,j_i  \in  \{-N,\cdots,N-1\}, \\
				\frac{K}{N}+ \sqrt{TK}   &\qquad  \text{otherwise}.
			\end{cases}
		\end{align}
		To see this, recalling \eqref{disbm}, note that
		\begin{align} \label{620}
			\|X_{t} - X_{i,J}\|_{L^2(\cA  * \cC)} &\le   (\delta - (t-t_i)) \|\alpha_{i,J} \|_{L^2(\cA  * \cC)} + {{ \beta_C    |W^0_{t} -W^0_{t_{0}}- W^0_{i,J}| }}  + C \sqrt{\delta} \beta_F \nonumber \\
			&\le    (\delta - (t-t_i)) \|\alpha_{i,J} \|_{L^2(\cA  * \cC)} \nonumber \\
			&\qquad +  \beta_C \Big( \sum_{i'=1}^{i}   |W_{t_{i'}}^0 -W_{t_{i'-1}}^0-\omega_{i',j_{i'}} |    +|W_t^0-W_{t_{i}}^0| \Big)  + C \sqrt{\delta}   \beta_F.   
		\end{align} 
		By \eqref{6000}, the conditional expectation of each term in the summation above is controlled as
		\begin{align*}
			\bE [|W_{t_{i'}}^0 - W_{t_{i'-1}}^0 - w_{i',j_{i'}}  | \mid \mathcal{F}_K^N] \le     \begin{cases}
				N^{-1}&\qquad j_{i'} \in \{-N,\cdots,N-1\}, \\
				\sqrt{\delta} &\qquad j_{i'}  = -N-1 \text{ or } N.
			\end{cases}
		\end{align*}
		Also by Lemma \ref{cond brownian},
		\begin{align} \label{610}
			\bE [ |W_t^0-W_{t_{i}}^0|   \mid  W_{t_{i+1}}^0 -W_{t_{i}}^0 ] &\le   \frac{t-t_i}{t_{i+1}-t_{i}} |W_{t_{i+1}}^0 -W_{t_{i}}^0|  + \sqrt{t-t_i} \nonumber \\
			&\le  |W_{t_{i+1}}^0 -W_{t_{i}}^0| +\sqrt{\delta} .
		\end{align} 
		Thus by the tower property of conditional expectations,
		\begin{align*}
			\bE [ |W_t^0-W_{t_{i}}^0|  \mid  \mathcal{F}_K^N ] &=   \bE [ |W_t^0 -W_{t_{i}}^0|  \mid  \mathcal{F}_{i+1,*}^N ] \\
			&= \bE [ \bE [ |W_t^0 -W_{t_{i}}^0|  \mid W_{t_{i+1}}^0 -W_{t_{i}}^0] \mid  \mathcal{F}_{i+1,*}^N ]  \\
			&\overset{\eqref{610}}{\le}  
			\bE [ |W_{t_{i+1}}^0 -W_{t_{i}}^0 |\mid  \mathcal{F}_{i+1,*}^N ] +\sqrt{\delta} =  
			\bE [ |W_{t_{i+1}}^0 -W_{t_{i}}^0 |\mid  \mathcal{F}_K^N ] +\sqrt{\delta} .
		\end{align*} 
		Thus, by plugging the above two estimates into \eqref{620}, using the fact $i \le K$ and $\delta = \frac{T}{K}$, we obtain \eqref{611}.

		~

		Now, we compare the Lagrangian part: for any multi-index $J$, on the event $O_{K,J},$
		\begin{align*}
			\mathbb{E}\Bigg[\int_{t_1}^T L_{\cA * \cC}(X_t,\alpha_t) dt &  \  \Big\vert  \   \mathcal{F}_K^N \Bigg]-\delta \sum_{i=1}^{K-1}  L_{\cA * \cC}(X_{i,J}, \alpha_{i,J})\nonumber \\
			& =  
			\sum_{i=1}^{K-1} \bE \Big[ \int_{t_{i}}^{t_{i+1}} \Big(L_{\cA * \cC}(X_t,\alpha_{i,J})-L_{\cA * \cC}(X_{i,J}, \alpha_{i,J})\Big)dt  \  \Big\vert  \   \mathcal{F}_K^N   \Big]     \\
			&\le C\sum_{i=1}^{K-1}   \bE \Big[ \int_{t_{i}}^{t_{i+1}} \|X_t - X_{i,J} \|_{L^2(\cA  * \cC)}   dt  \  \Big\vert  \   \mathcal{F}_K^N   \Big] .
		\end{align*} 
		Taking the expectation w.r.t.  $ \mathcal{F}_K^N   ,$
		\begin{align} \label{615}
			& \mathbb{E}\Big[\int_{t_0}^T L_{\cA * \cC}(X_t,\alpha_t) dt \Big]-\delta  \sum_{i=1}^{K} \sum_{J\in [N]^K}  L_{\cA * \cC}(X_{i,J}, \alpha_{i,J}) \bP (O_{K,J}) \nonumber \\
			&\le \bE \Big[ \int_{t_{0}}^{t_{1}}  L_{\cA * \cC}(X_t,\alpha_t)dt \Big] - \delta \sum_{J\in [N]^K}    L_{\cA * \cC}(X_{K,J}, \alpha_{K,J})  \bP (O_{K,J}) \nonumber \\
			& \qquad +C\sum_{i=1}^{K-1}  \int_{t_{i}}^{t_{i+1}} \bE \Big[\sum_{J\in [N]^K}  \bE \Big[  \|X_t - X_{i,J} \|_{L^2(\cA  * \cC)}     \  \Big\vert  \   \mathcal{F}_K^N   \Big] \mathbbm{1}_{O_{K,J}}\Big] dt.
		\end{align} 
		By \eqref{611}, for every $i\in \{1,\cdots,K-1\}$ and $t\in (t_i,t_{i+1}],$ the last term above is bounded as
		\begin{align} \label{617}
			\bE \Big[&\sum_{J\in [N]^K}  \bE \Big[  \|X_t - X_{i,J} \|_{L^2(\cA  * \cC)}    \  \Big\vert  \   \mathcal{F}_K^N   \Big] \mathbbm{1}_{O_{K,J}}\Big] \nonumber \\
			&\le   
			C \beta_C  (\sqrt{\delta} +  
			\bE  |W_{t_{i+1}}^0 -W_{t_{i}}^0 | ) + C \sqrt{\delta} \beta_F  \nonumber   \\
			&+ (\delta - (t-t_i))      \sum_{J\in [N]^K} \|\alpha_{i,J} \|_{L^2(\cA  * \cC)}   \bP(O_{K,J})
			+ C\beta_C \bE \Big[ \sum_{J\in [N]^K} \zeta_{i,J}\mathbbm{1}_{O_{K,J}} \Big].
		\end{align}
		Note that  there exists $C>0$ such that for every $i$, 
		\begin{align} \label{616}
			\bE |W_{t_{i+1}}^0 -W_{t_{i}}^0 | 
			= C\sqrt{\delta}.
		\end{align}
		To bound the last term above, we define the sets
		\begin{align*}
			\textsf{bulk}:=\{ J\in [N]^K:  \text{$j_{1},\cdots,j_K  \in \{-N,\cdots,N-1\}$}\}
		\end{align*}
		and 
		\begin{align*}
			\textsf{edge}:= \{ J\in [N]^K:  \text{there exists $i \in \{1,\cdots,K\}$ such that $j_{i}  = -N-1$ or $j_{i}  = N$}\}.
		\end{align*}
		Observe that $\textsf{bulk}$ and $\textsf{edge}$ consists of the partition of $[N]^K$. To see that  $ \textsf{edge}$ is a rare set, note that for any $i \in \{1,\cdots,K\}$, by a Gaussian tail estimate,
		\begin{align*}
			\bP( o_{i,-N-1})  = \bP\Big({W^0_{t_i}-W^0_{t_{i-1}} \over \sqrt \delta} \le -\frac{1}{\sqrt{\delta}}\Big) \le Ce^{-1/(2\delta)},
		\end{align*}
		and the same probability bound holds for $O_{i,N}$ as well. Hence, 
		by a union bound, for sufficiently large $K$ (equivalently, for sufficiently small $\delta = T/K$), 
		\begin{align*}
			\bE \Big[ \sum_{J\in  \textsf{edge}} \mathbbm{1}_{O_{K,J}}\Big] &= \bP\Big(\bigcup_{J\in  \textsf{edge}} O_{K,J}  \Big) \\
			&\le  \bP( \text{$o_{i,-N-1}$ or  $o_{i,N}$ 
				occurs for some $i \in \{1,\cdots,K\}$} ) \\
			&\le 2K \cdot Ce^{-1/(2\delta)} \le e^{-K/(4T)}.
		\end{align*}
		Therefore, recalling the definition of $\zeta_{i,J}$ in \eqref{zeta},  the last term in \eqref{617} is controlled as
		\begin{align} \label{622}
			\bE \Big[ \sum_{J\in [N]^K} \zeta_{i,J}\mathbbm{1}_{O_{K,J}}\Big]  & =   \bE \Big[ \sum_{J\in  \textsf{bulk}} \zeta_{i,J}\mathbbm{1}_{O_{K,J}}\Big] +    \bE \Big[ \sum_{J\in  \textsf{edge}} \zeta_{i,J}\mathbbm{1}_{O_{K,J}}\Big]   \nonumber \\
			&\le \frac{K}{N} + \Big(  \frac{K}{N} + \sqrt{TK}\Big)e^{-K/(4T)} \le \frac{2K}{N} + \sqrt{TK} e^{-K/(4T)} .
		\end{align}
		Thus, using this and \eqref{616}, the  last summation term in  \eqref{615} is bounded by
		\begin{align} \label{619}
			C&T
			\sqrt{\delta} (\beta_C + \beta_F)dt + C \sum_{i=1}^{K-1} \sum_{J\in [N]^K} \|\alpha_{i,J} \|_{L^2(\cA  * \cC)}  \bP(O_{K,J})  \Big[\int_{t_{i}}^{t_{i+1}}   (\delta - (t-t_i))     dt \Big]\nonumber  \\
			&\qquad \qquad + C\beta_C T\Big( \frac{K}{N} + \sqrt{TK}e^{-K/(4T)} \Big)  \nonumber \\
			&\le CT\sqrt{\delta} (\beta_C + \beta_F) + C \delta^2 \sum_{i=1}^{K-1} \sum_{J\in [N]^K} \|\alpha_{i,J} \|_{L^2(\cA  * \cC)}  \bP(O_{K,J})  + C\beta_C \Big(\frac{TK}{N} +  e^{-K/(8T)}\Big).
		\end{align}
		Let us control the second term above. Since $(\alpha_{i,J})_{i,J}$ is an $\epsilon$-optimal control,
		by \eqref{eq:july23.2024.9} and \eqref{eq:july26.2024.2c},  
		\begin{align} \label{624}
			\sum_{i=1}^{K-1} \sum_{J\in [N]^K} \|\alpha_{i,J} \|^2_{L^2(\cA  * \cC)}  \bP(O_{K,J})  = \sum_{i=1}^{K}\sum_{J\in [N]^i} \| \alpha_{i,J}\|^2_{L^2(\cA  * \cC)} \bP(O_{i, J}) \le C \frac{1+\epsilon}{\delta},
		\end{align}
		where $C>0$ is a constant depending on $T,\beta_C, \beta_F$ and $\norm{x_0}_{L^2(\cA)}.$
		Now we use the following basic  inequality: For any $\beta_1,\cdots,\beta_m \ge 0$ and $p_1,\cdots,p_m \ge 0$ such that $\sum_{i=1}^m p_i=1$,
		\begin{align}
			\sum_{i=1}^m \beta_i p_i \le \sqrt{\sum_{i=1}^m \beta_i^2 p_i}.
		\end{align}
		This follows from the concavity of the square-root function.
		Applying this to \eqref{624}, we have
		\begin{align} \label{623}
			\sum_{i=1}^{K}\sum_{J\in [N]^K} \| \alpha_{i,J}\|_{L^2(\cA  * \cC)}\bP(O_{K, J}) \le C\sqrt{\frac{1+\epsilon}{\delta}} \le C \delta^{-1/2}.
		\end{align}
		In addition, recalling that $\alpha_t=0$ on $[t_0,t_1]$ and using the condition $-C_1 \le L_{\cA * \cC}(X,\alpha) \le C_1(1+\|X\|_{L^2(\cA  * \cC)} + \|\alpha \|_{L^2(\cA  * \cC)}^2)   $,
		\begin{align*}
			\bE \Big[ \int_{t_{0}}^{t_{1}}  L_{\cA * \cC}(X_t,\alpha_t)dt \Big] &- \delta \sum_{J\in [N]^K}    L_{\cA * \cC}(X_{K,J}, \alpha_{K,J})  \bP (O_{K,J})  \\
			&\le   C  \Big[ \int_{t_{0}}^{t_{1}}   \bE \|X_t\|_{L^2(\cA  * \cC)}    dt \Big] + C\delta  \overset{\eqref{626}}{\le}  C \delta (1+ \norm{x_0}_{L^2(\cA)}) + C{\delta}^{3/2} \le C'\delta,
		\end{align*}
		where $C'>0$ is a constant depending on $\norm{x_0}_{L^2(\cA)}$ and $\delta =\frac{T}{K} >0$ is sufficiently small.
		Plugging this along with \eqref{623} into \eqref{615} and \eqref{619}, we deduce that 
		\begin{align} \label{621}
			\mathbb{E}\Big[\int_{t_0}^T L_{\cA * \cC}(X_t,\alpha_t)  \Big]&-\delta \sum_{i=1}^{K} \sum_{J\in [N]^K}  L_{\cA * \cC}(X_{i,J}, \alpha_{i,J}) \bP (O_{K,J}) \nonumber \\
			&\le C'{\frac{T}{K}} +CT\sqrt{\frac{T}{K}} (\beta_C + \beta_F) + C  \Big(\frac{T}{K}\Big)^{3/2}    + C\beta_C\frac{TK}{N} + C\beta_Ce^{-K/(8T)}.
		\end{align}
		Finally we compare the terminal cost.  By \eqref{611}, for any multi-index $J \in [N]^K$,
		\begin{align*}  
			\mathbb{E} \Big[\|X_{T} - & X_{K,J}\|_{L^2(\cA  * \cC)} \mid \mathcal{F}_K^N \Big] \\
			&\le \mathbb{E}\Big[\|X_{T} - X_{K-1,J}\|_{L^2(\cA  * \cC)} \mid \mathcal{F}_K^N \Big]+\mathbb{E}\Big[\|X_{K-1,J} - X_{K,J}\|_{L^2(\cA  * \cC)} \mid \mathcal{F}_K^N \Big]  \\
			&\leq  
			C \beta_C \Big (     \zeta_{J} + \sqrt{\delta}  +
			\bE [ |W_{t_K}^0 -W_{t_{K-1}}^0 |\mid  \mathcal{F}_K^N ]  \Big) + C\sqrt{\delta} 
			\beta_F + \delta \| \alpha_{K,J}\|_{L^2(\cA  * \cC)}.
		\end{align*}
		Since
		\begin{align*}
			\mathbb{E}\big[g_{\cA * \cC}(X_T) \mid \mathcal{F}_K^N ]-g_{\cA * \cC}(X_{K,J})\leq  
			C\, \mathbb{E}\big[\|X_T-X_{K,J}\|_{L^2(\cA  * \cC)} \mid \mathcal{F}_K^N  ],
		\end{align*}
		taking the average w.r.t. $ \mathcal{F}_K^N ,$ by the similar reasoning as before,
		\begin{align*} 
			\mathbb{E}\big[g_{\cA * \cC}(X_T)]&- \sum_{J\in [N]^K} \bP(O_{K, J}) g_{\cA * \cC}(X_{K,J}) \\
			&\le C\sum_{J\in [N]^K}  \bE \Big[  \mathbb{E}\big[\|X_T-X_{K,J}\|_{L^2(\cA  * \cC)} \mid \mathcal{F}_K^N  ] \mathbbm{1}_{O_{K,J}} \Big] \\
			&\le  C\sqrt{\delta}(\beta_C + \beta_F) +C \beta_C \bE \Big[\sum_{J\in [N]^K}     \zeta_J  \mathbbm{1}_{O_{K,J}} \Big]  + \delta \sum_{J\in [N]^K} \| \alpha_{K,J}\|_{L^2(\cA  * \cC)}  \bP (O_{K,J})  \\
			&\overset{\eqref{622},\eqref{623}}{\le}   C\sqrt{\delta}(\beta_C + \beta_F) +C \beta_C \Big( \frac{2K}{N} + \sqrt{TK} e^{-K/(4T)} \Big) + C\sqrt{\delta}.
		\end{align*}
		Combining this with  \eqref{621},  we establish that for sufficiently large $K,$
		\begin{align*}
			\widetilde{V}_{\cA * \cC}(t_0,x_0) \le V^{K,N,R}_{\cA }(t_0,x_0) + \epsilon+C(T+1)\sqrt{\frac{T}{K}}(\beta_C + \beta_F)  + C\beta_C\frac{TK}{N} .
		\end{align*}
		Since $\epsilon>0$ is arbitrary, we conclude the proof.
		
		~
		
		\textbf{Matrix case: Upper bound for a discretized value function.}
		Recall that $\widehat{\mathcal{F}}^n_{i}$ denotes the $\sigma$--algebra  generated by  
		$\{\widehat{W}^n_{t_1}-\widehat{W}^n_{t_0}, \ldots, \widehat{W}^n_{t_i}-\widehat{W}^n_{t_{i-1}}\}$. For $\epsilon>0$, choose an $\epsilon$-optimal control $ (\alpha_t^n)_{t}$ in $\widehat{\bA}_{M_n(\bC)}^{t_0,T}$, and define
		\begin{align} \label{6140}
			\alpha_{i,J}^n :=
			\frac{1}{\delta}\int_{t_{i-1}}^{t_{{i}}} \overline {\mathbb{E}}\big[ \alpha^n_t \mid  O_{i,J}, \widehat{\mathcal{F}}^n_{i}\big] dt   
		\end{align}
		(recall that $(\overline \Omega, \overline \bP)$ denotes the product probability space supporting both common and GUE noise, and $\overline {\mathbb{E}}$ denotes the correspoding expectation).
		{Note that  this control belongs to $\widehat{\mathbb{A}}_{M_n(\mathbb{C}),t_0}^{K,N,R}$, a class of discretized controls for the matrix problem (see Definition \ref{defn:july09.2024.2}).}
		Then, for $i' \le i$,  
		$$
		{{\int_{t_{i'-1}}^{t_{{i'}}} \overline{ \mathbb{E}}\big[ \alpha^n_t\mid     O_{i,J}, 
				\widehat{\mathcal{F}}^n_{i}\big]dt}} = \int_{t_{i'-1}}^{t_{{i'}}} \overline {\mathbb{E}}\big[ \alpha^n_t\mid    O_{i',J}, \widehat{\mathcal{F}}^n_{i'} \big]dt = \delta\, \alpha^n_{i',J},
		$$
		where the first identity follows from the fact that $(\alpha^n_t)_{t\in [t_{i'-1},t_{i'} ]}$ is independent of $W^0_{t_j} - W^0_{t_{j-1}} $ and $\widehat W^n_{t_j} - \widehat W^n_{t_{j-1}}$ for $j \ge i'+1$. Thus
		similarly as in \eqref{612}, 
		\begin{align*}   
			&\overline{\mathbb{E}}\big[ X^n_{t_i}  \mid  O_{i,J}, \widehat{\mathcal{F}}^n_{i} ] \nonumber \\
			& =  x_0^n + \sum_{i'=1}^i \overline{\mathbb{E}} \Big[ \int_{t_{i'-1}}^{t_{i'}}\alpha_{t}^n \mid  O_{i,J},\widehat{\mathcal{F}}^n_{i} \Big]dt + \beta_C\, \mathbbm{1}_{M_n (\mathbb{C})} \sum_{i'=1}^i 
			\overline{\mathbb{E}} [W_{t_{i'}}^0 - W_{t_{i'-1}}^0  \mid O_{i,J},\widehat{\mathcal{F}}^n_{i}] + \beta_F(\widehat W_{t_i}^n-\widehat W_{t_0}^n ) \nonumber  \\
			&=  x_0^n + \delta \sum_{i'=1}^i \alpha_{i',J}^n  + \beta_C\, \mathbbm{1}_{M_n (\mathbb{C})} W_{i,J}^0 + \beta_F(\widehat W_{t_i}^n-\widehat W_{t_0}^n )  = X_{i,J}^n ,
		\end{align*}
		where the second identity follows from the independence of $ (W^0_t)_{t}$ and  $ (\widehat W^n_t)_{t}$. 
		Thus by Jensen's inequality as before,  
		\begin{align*}
			V^{K,N,R}_{M_n (\mathbb{C})}&(t_0,x_0^n)  \leq \widehat{\bE} \Big[\sum_{i=1}^K \sum_{J\in [N]^i}\bP(O_{i,J})L_{M_n (\mathbb{C})}(X_{i,J}^n, \alpha_{i,J}^n)\delta + \sum_{J\in [N]^K}\bP(O_{K,J}) g_{M_n (\mathbb{C})}(X_{K,J}^n)\Big]\\
			&=  \    \widehat{\bE}  \Big[ \sum_{i=1}^K \sum_{J\in [N]^i}
			L_{M_n (\mathbb{C})}( \overline{\bE} \big[ X^n_{t_i} \mid O_{i,J} , \widehat{\mathcal{F}}^n_{i} ], \alpha^n_{i,J})  \delta  \cdot \bP(O_{i,J})  \Big]  \\
			& \qquad \qquad \qquad \qquad  \qquad \qquad  
			+    \widehat{\bE}  \Big[ \sum_{J\in [N]^K} g_{M_n (\mathbb{C})}(  \overline{\bE} \big[ X_{T}^n \mid O_{K,J}, \widehat{\mathcal{F}}^n_{i}] ) 
			\bP(O_{i,J})    \Big] \\
			&\overset{\eqref{6140}}{\le}   \widehat{\bE} \Big[ 
			\sum_{i=1}^K\sum_{J\in [N]^i} \int_{t_{i-1}}^{t_i} 
			\overline{\bE}  [ L_{M_n (\mathbb{C})}( X^n_{t_i} , \alpha^n_{t})  \mid O_{i,J} ,\widehat{\mathcal{F}}^n_{i} ]   \cdot \bP(O_{i,J})   \Big]    \\
			& \qquad \qquad \qquad \qquad \qquad \qquad  +  \widehat{\bE} \Big[ \sum_{J\in [N]^K} \overline{\bE} \big[ g_{M_n (\mathbb{C})}(   X_{T}^n ) \mid  O_{K,J},\widehat{\mathcal{F}}^n_{i} ]  
			\bP(O_{i,J})   \Big]    \\
			&=  \overline{\bE} \bigg[\sum_{i=1}^K \int_{t_{i-1}}^{t_i} L_{M_n (\mathbb{C})}(X^n_{t_i}, \alpha^n_t)dt + g_{M_n (\mathbb{C})}(X^n_{T})\bigg]\\
			&\le  \widehat{V}_{M_n (\mathbb{C})}(t_0,x_0^n)+\frac{C}{R} + \epsilon +\overline{\bE} \Bigg[\sum_{i=1}^K \int_{t_{i-1}}^{t_i}\big| L_{M_n (\mathbb{C})}(X^n_{t_i}, \alpha^n_t) -  L_{M_n (\mathbb{C})}(X^n_{t}, \alpha^n_t) \big|dt\Bigg].
		\end{align*}
		The last term can be controlled as in the von Neumann case \eqref{321}-\eqref{632}. Only one difference is a matrix Brownian motion part:
		For $t\in [t_{i-1},t_i],$
		\begin{align*}
			\|X_{t_i}^n -  X^n_{t} \|_{ L^2 (M_n( \bC))} \le C \Big( \int_{t}^{t_i} \|\alpha^n_s  \|_{ L^2 (M_n( \bC))} ds + \beta_C |W_{t_i}^0  -W_{t}^0| + \beta_F \|\widehat W^n_{t_i}- \widehat W^n_{t}\|_{ L^2 (M_n( \bC))} \Big).
		\end{align*}
		{Using the fact that 
			\begin{align} \label{631}
				\widehat  \bE [  \|\widehat W^n_t- \widehat W^n_{s}\|_{ L^2 (M_n( \bC))}] \le C\sqrt{t-s},\qquad \forall t \ge s \ge 0,
			\end{align}
			where $C>0$ is a constant \emph{not} depending on the matrix size $n$,}
		taking the expectation, 
		\begin{align*}
			\overline{\bE} [\|X^n_{t_i} -  X^n_t \|_{ L^2 (M_n( \bC))}] \le C \Big( \int_{t}^{t_i} \overline \bE \|\alpha^n_s  \|_{ L^2 (M_n( \bC))} ds + \sqrt{\delta}\Big).
		\end{align*}
		{Note that using \eqref{631}, one can deduce a matrix-version  of \eqref{l2 bound} as follows: For an $\epsilon$-optimal control $ (\alpha_t^n)_{t}$,
			\begin{align*}
				\overline  \bE\Big[\int_{t_0}^T \|\alpha^n_t\|_{L^2(\cA)}^2dt\Big] \leq C (\| x_0^n\|_{L^2(M_n( \bC))} + T+1) .
			\end{align*}
			Hence by the argument as in \eqref{321}-\eqref{632}, we are done. }
		
		~
		
		\textbf{Matrix case: Lower bound for a discretized value function.} For $\epsilon>0$, let $ (\alpha_{i,J}^n)_{i,J}$ be an $\epsilon$-optimal discrete control.  Similarly as before,
		we set $\alpha^n_t := 0$ for $t\in [t_0,t_1]$, and for $t\in (t_i,t_{i+1}]$ with $i \in \{1,\cdots,K-1\}$,  
		$$
		\alpha^n_t( (\omega ,  \widehat{\omega})) := \alpha^n_{i,J}  (\widehat{\omega})  \ \hbox{ where $J$ is a multi-index such that }\omega\in O_{i+1,J}.
		$$
		We show that the estimate \eqref{611} holds for the matrix case as well: For any  $i\in \{1,\cdots, K-1\}$ and   $t\in (t_{i},t_{i+1}]$,  on the event $O_{i+1,J}$,  
		\begin{align} \label{322}
			\overline{\mathbb{E}}\Big[\|X^n_{t} - X^n_{i,J}\|&_{ L^2 (M_n( \bC))} \mid \mathcal{F}_K^N, \widehat{\mathcal{F}}^n_{K} \Big]  \nonumber \\
			&\leq  (\delta - (t-t_i)) \|\alpha^n_{i,J} \|_{ L^2 (M_n( \bC))}   +   
			C \beta_C \Big (     \zeta_{i,J} + \sqrt{\delta}  +
			\bE [ |W_{t_{i+1}}^0 -W_{t_{i}}^0 |\mid  \mathcal{F}_K^N ]  \Big) \nonumber \\
			&\qquad + C 
			\beta_F \big(\| \widehat W^n_{t_{i+1}} - \widehat W^n_{t_i}  \|_{ L^2 (M_n( \bC))}
			+ C\sqrt{\delta}\big ) ,
		\end{align}
		where $\zeta_{i,J}$ is defined in \eqref{zeta}. To see this, observe that we have a bound analogous to  \eqref{620}:
		\begin{align*}
			\|X^n_{t} - X^n_{i,J}&\|_{ L^2 (M_n( \bC))}  \le   (\delta - (t-t_i)) \|\alpha^n_{i,J} \|_{ L^2 (M_n( \bC))}  \\
			& +  \beta_C \Big( \sum_{i'=1}^{i}   |W_{t_{i'}}^0 -W_{t_{i'-1}}^0-\omega_{i',j_{i'}} |    +|W_t^0-W_{t_{i}}^0| \Big)  + C  \beta_F \|\widehat W^n_t- \widehat W^n_{t_i}\|_{ L^2 (M_n( \bC))}  . 
		\end{align*}
		We take the conditional expectation w.r.t.  $\mathcal{F}_K^N$ together with $ \widehat{\mathcal{F}}^n_{K}$, and then use the following analog of \eqref{610} for the GUE case:
		\[
		\mathbb{E}\bigl[\|\widehat W^n_t - \widehat W^n_{t_i} \|_{ L^2 (M_n( \bC))} \bigm| \widehat W^n_{t_{i+1}} - \widehat W^n_{t_i} \bigr]
		\le\| \widehat W^n_{t_{i+1}} - \widehat W^n_{t_i}  \|_{ L^2 (M_n( \bC))}
		+ C\sqrt{\delta},
		\]
		which follows from Lemma \ref{matrix} in Appendix. Thus
		we obtain the desired estimate \eqref{322}.

		Hence, one can proceed with the same argument as before, with only difference that we take the conditioning w.r.t.  $\mathcal{F}_K^N$ together with $ \widehat{\mathcal{F}}^n_{K}$. Indeed, using the fact
		\begin{align*}
			\widehat \bE   \| \widehat W^n_{t_{i+1}} - \widehat W^n_{t_i}  \|_{ L^2 (M_n( \bC))} \le C \sqrt{
				\delta
			}
		\end{align*}
		($C>0$ is a constant independent of $n$),
		we have an analog of \eqref{617}: For every $1\le i\le K-1,$
		\begin{align}  
			\widehat \bE \Big[&\sum_{J\in [N]^K}  \overline{\mathbb{E}} \Big[  \|X^n_t - X^n_{i,J} \|_{ L^2 (M_n( \bC))}    \  \Big\vert  \   \mathcal{F}_K^N , \widehat{\mathcal{F}}^n_{K}  \Big] \mathbbm{1}_{O_{K,J}}\Big] \nonumber \\
			&\le   
			C \beta_C  (\sqrt{\delta} +  
			\bE  |W_{t_{i+1}}^0 -W_{t_{i}}^0 | ) + C  \beta_F  
			\sqrt{\delta}
			\nonumber   \\
			&+  (\delta - (t-t_i))    \sum_{J\in [N]^K} \|\alpha^n_{i,J} \|_{ L^2 (M_n( \bC))}   \bP(O_{K,J})
			+ C\beta_C \bE \Big[ \sum_{J\in [N]^K} \zeta_{i,J}\mathbbm{1}_{O_{K,J}} \Big].
		\end{align}
		With the aid of \eqref{631}, the above quantity is bounded as in \eqref{617}. Therefore, the aforementioned proof works and we conclude the proof.

	\end{proof}

	\section{Proof of convergence} \label{sec: convergence}
	
	Recall that $\overline{V}$ denotes  the value function on the space of non-commutative laws, defined in  \eqref{eqn:value}, and   $\widehat{V}_{M_n(\bC)}$ denotes the value function for a $n\times n$ matrix control problem defined in \eqref{valuematrix}. 
	\begin{theorem}\label{thm:convergence}
		Suppose that Assumptions \textbf{A}, \textbf{B}, {and \textbf{C}} hold. For any sequence of $x_0^n\in M_n(\bC)_{\textup{sa}}^d$, such that operator norms are uniformly bounded in $n$ and the laws converge weakly* to $\lambda_0\in \Sigma_{d}^{2}$ as $n\rightarrow \infty$, we have
		$$
		\lim_{n\rightarrow \infty} \widehat{V}_{M_n(\bC)}(t_0,x_0^n) = \overline{V}(t_0,\lambda_0).
		$$
	\end{theorem}
	
	\begin{proof}
		For brevity, we shall write \(\|\cdot\|\) to denote the non-commutative \(L^2\)–norm \(\|\cdot\|_2 = \|\cdot\|_{L^2(\cA)}\) throughout the proof.  In contrast, the operator norm will always be displayed explicitly as \(\|\cdot\|_\infty\). 
		
		~

		\textbf{Upper bound:}    We first prove the upper bound  
		$$
		\limsup_{n\rightarrow \infty} \widehat{V}_{M_n(\bC)}(t_0,x_0^n) \leq \overline{V}(t_0,\lambda_0).
		$$
		
		\textbf{Upper bound, step 1 (discretization)}: We first use a discretization method (Proposition \ref{prop:discretization}) to reduce to the case of finite time intervals and a finite probability space for the common noise.  Given any $\epsilon>0$, we may find $\cA\in \bW$ containing a free Brownian motion $(S_t)_{t \in [t_0,T]}$ freely independent of $x_0 \in L^2(\cA)^d_{\textup{sa}}$ with $\lambda_{x_0} = \lambda_0$, and $N,K,R>0$ along with an admissible control $\alpha = (\alpha_{i,J})_{i,J} \in \mathbb{A}_{\cA,t_0,x_0}^{K,N,R}$ for the discretized problem such that 
		\begin{align} \label{521}
			\overline{V}(t_0,\lambda_0) =   \overline{V}_\cA (t_0,x_0) & =   \widetilde{V}_\cA (t_0,x_0)   \nonumber  \\
			&\geq  V_\cA^{K,N,R} (t_0,x_0) - \epsilon \nonumber \\
			&\ge \sum_{J\in [N]^K}\bP(O_{K,J})\big[\sum_{i=1}^K L_{\cA}(X_{i,J},\alpha_{i,J})\, \delta + g_{\cA}(X_{K,J})\big] - 2 \epsilon,
		\end{align}
		where $X_{i,J}$ is given by \eqref{eqn:discrete_dynamics}
		with the initial condition $x_0$.   
		Here we used  Lemma \ref{lem:decreasing} and chose  $N$, $K$, $R$  so that the error from the discretization of the matrix problem is smaller than $\epsilon$.

		\textbf{Upper bound, step 2 (restriction of ambient algebra)}: Using the $E$-convexity of $L$ and $g$ from  Assumption \textbf{B}, we may restrict to the algebra generated by the initial condition and the increments of the free Brownian motion.  In other words, we may assume that for each $i$ and $J$,
		\begin{align} \label{640}
			\alpha_{i,J} \in  \big(W^*(x_0, S_{t_1}{{-S_{t_0}}}, S_{t_2}- S_{t_1},\ldots , S_{t_{i}}- S_{t_i-1})\big)^d=:\cB_i.
		\end{align}
		Indeed, let $E_{\cB_i}: \cA \to \cB_i$ be the adjoint of the embedding of $\cB_i$ into $\cA$, and set $\alpha 
		'_{i,J}:=E_{\cB_i}(\alpha_{i,J}).$ We have $\|\alpha'_{i,J}\|_{L^\infty(\cA)} \leq\|\alpha_{i,J}\|_{L^\infty(\cA)} \le R$. Being a $L^2$-projection onto $L^2(\cB_i),$  $(\alpha_{i,J}')_{i, J}$ is an admissible control in the sense of Definition \ref{defn:july09.2024.1}. {{By the free independence of increments of free semi-circular process $(S_t)_t$,}} we deduce that the discretized process  associated to this new policy $(\alpha_{i,J}')_{i, J}$  is  given by $X_{i,J}'=E_{\cB_i}(X_{i,J})$.  Thanks to the $E$-convexity assumption on $L$ and $g$ (see Assumption \textbf{B}), this new policy 
		can only reduce the value of the discretized cost, thus satisfying \eqref{521} as well.

		
		\textbf{Upper bound, step 3 (polynomial control policy):}  The control policy $\alpha_{i,J}$ now takes values in the algebra $\cB_i$ generated by the initial condition and semi-circular increments thanks to \eqref{640}.  We will next approximate it by a non-commutative polynomial in the initial conditions and semi-circular increments.  Hence, by applying the same non-commutative polynomial to the matrix initial conditions and matrix Brownian motion, we will obtain a candidate control policy for the discretized matrix problem.
		
		Although it is immediate that $\alpha_{i,J}$ can be approximated by a polynomial of the generators, we also need to arrange that the evaluation of the polynomial on the matrices is bounded in operator norm.  Fix $M \geq 3 \sqrt{T}$ such that $\sup_n \norm{x_0^n} \leq M$; here $3 \sqrt{T}$ is chosen since  $\norm{S_{t_j} - S_{t_{j-1}}} \le  2 \sqrt{t_j - t_{j-1}} < 3 \sqrt{T}$. Using some standard facts in operator algebras \cite[Lemma 2.2]{jekel2024chi}, for every $\kappa > 0$, there is a non-commutative polynomial $p_{i,J}$ (in variables corresponding to $x_0$ and $S_{t_j} - S_{t_{j-1}}$ for $j=1,\cdots,i$) such that
		\[
		\norm{\alpha_{i,J} - p_{i,J}(x_0,S_{t_j} - S_{t_{j-1}}: j=1,\cdots,i )}< \kappa,
		\]
		and also
		\[
		\norm{p_{i,J}(y)} \leq R
		\]
		whenever $y = (y_j)_j$ is any tuple of variables with $\norm{y_j} \leq M$. This is what will allow us to achieve operator norm boundedness of the matrix approximations.  Meanwhile, because $L$ and $g$ are Lipschitz with respect to $\norm{\cdot}_2$ on the operator norm ball of radius $M$, if $\kappa>0$ is chosen sufficiently small, then the control policy
		\begin{align} \label{522}
			\widetilde{\alpha}_{i,J} := p_{i,J}(x_0,S_{t_j} - S_{t_{j-1}}: j=1,\cdots,i )
		\end{align}
		satisfies 
		\begin{align} \label{523} 
			\sum_{J\in [N]^K}\bP(O_{K,J})& \Big[\sum_{i=1}^K L_{\cA}(\widetilde X_{i,J},\widetilde \alpha_{i,J})\, \delta + g_{\cA}( \widetilde X_{K,J})\Big] \nonumber \\
			&\le  \sum_{J\in [N]^K}\bP(O_{K,J})\Big[\sum_{i=1}^K L_{\cA}(X_{i,J},\alpha_{i,J})\, \delta + g_{\cA}(X_{K,J})\Big] +  \epsilon,
		\end{align}
		where $\widetilde X_{i,j}$ is given by \eqref{eqn:discrete_dynamics} with $\alpha_{i,J}$ replaced by $\widetilde\alpha_{i,J}$.

		\textbf{Upper bound, step 4 (construction of matrix approximations)}:
		The polynomials $p_{i,J}$ may now be applied to the matrix initial conditions $x_0^n$ and GUE($n$) Brownian motion increments $(\widehat{W}^n_{t_j} - \widehat{W}^n_{t_{j-1}})_{j=1,\cdots,i}$.  Let
		\begin{align} \label{644}
			\alpha_{i,J}^n := \mathbbm{1}_{\sup_{j=1,\cdots,K}\norm{\widehat{W}_{t_j}^n - \widehat{W}_{t_{j-1}}^n} \leq M} \  p_{i,J}(x_0^n, \widehat{W}_{t_j}^{n} - \widehat{W}_{t_{j-1}}^n: j=1,\cdots,i).
		\end{align}
		Note that the resulting process $X_{i,J}^n$ in the state space is also given by a polynomial in the initial conditions and increments, thanks to (\ref{eqn:discrete_dynamics_matrix}), and the process $\widetilde X_{i,J}$ induced by $\widetilde{\alpha}_{i,J}$ is given by the same polynomial in the initial condition and semi-circular increments.  Because of the indicator function in \eqref{644}, $\alpha_{i,J}^n$ will be zero unless $\norm{\widehat{W}_{t_j}^n - \widehat{W}_{t_{j-1}}^n} \leq M$ for all $j=1,\cdots,K$, which forces the $\norm{\alpha_{i,J}^n} \leq R$ based on our choice of the polynomial $p_{i,J}$.

		By Theorem \ref{thm: asymptotic freeness} and our choice of $M$, we have $\limsup_{n \to \infty} \sup_{j=1,\cdots,K} \norm{\widehat{W}_{t_j}^n - \widehat{W}_{t_{j-1}}^n} \leq M$ almost surely, and so almost surely $\alpha_{i,J}^n$ eventually agrees with $p_{i,J}$ applied to the matrix increments.  Moreover, Theorem \ref{thm: asymptotic freeness} implies that the joint law of  $(x_0^n, \widehat{W}_{t_1}^{n} - \widehat{W}_{t_{0}}^n, \cdots, \widehat{W}_{t_K}^{n} - \widehat{W}_{t_{K-1}}^n)$  converges almost surely to the law of the freely independent variables  $(x_0, S_{t_1} - S_{t_{0}},\cdots, S_{t_K} - S_{t_{K-1}})$.  Since convergence in law is preserved by application of non-commutative polynomial functions, it follows that the joint non-commutative law of $p_{i,J}(x_0^n, \widehat{W}_{t_j}^{n} - \widehat{W}_{t_{j-1}}^n: j  = 1,\cdots,i)$ together with $x_0^n$ converges almost surely to the law of $p_{i,J}(x_0, S_{t_j} - S_{t_{j-1}}: j  = 1,\cdots,i)$ together with $x_0$.  Hence also the joint law of $\alpha_{i,J}^n$ together with $x_0^n$ converges almost surely to the joint law of $\widetilde \alpha_{i,J}$ together with $x_0$  because $\mathbbm{1}_{\sup_{j=1,\cdots,K}\norm{\widehat{W}_{t_j}^n - \widehat{W}_{t_{j-1}}^n} \leq M}$ converges to $1$  almost surely.  This finally implies almost sure convergence of the joint law of the matrix initial condition $x_0^n$, controls $\alpha_{i,J}^n$, and the output process $X_{i,J}^n$ to the joint law of $x_0$, $\widetilde{\alpha}_{i,J}$, and  $\widetilde X_{i,J}$.  As the Lagrangian is  continuous w.r.t. weak* topology (see Assumption \textbf{C}),  the corresponding cost will converge, i.e.,  
		\begin{align*}
			\lim_{n\rightarrow \infty} & \sum_{J\in [N]^K}\bP(O_{K,J})\widehat {\mathbb{E}}\Big[\sum_{i=1}^K L_{M_n(\bC)}({X}_{i,J}^n,a^n_{i,J})
			\delta + g_{M_n(\bC)}({X}^n_{K,J})   \Big]\\
			=&\ \sum_{J\in [N]^K} \bP(O_{K,J})  \Big[\sum_{i=1}^K L_{\cA}(\widetilde X_{i,J},\widetilde a_{i,J})
			\delta + g_{\cA}(\widetilde X_{K,J})\Big].
		\end{align*}
		
		\textbf{Upper bound, step 5 (conclusion):}  Therefore, we obtain that
		\begin{align*}
			\limsup_{n \to \infty} \widehat{V}_{M_n(\bC)}^{K,N,R}(t_0,x_0^n) &\leq     \limsup_{n \to \infty} 
			\sum_{J\in [N]^K}\bP(O_{K,J})\widehat {\mathbb{E}}\Big[\sum_{i=1}^K L_{M_n(\bC)}({X}_{i,J}^n,a^n_{i,J})
			\delta + g_{M_n(\bC)}({X}^n_{K,J})   \Big]\\ &=\sum_{J\in [N]^K} \bP(O_{K,J})  \Big[\sum_{i=1}^K L_{\cA}(\widetilde X_{i,J},\widetilde a_{i,J})
			\delta + g_{\cA}(\widetilde X_{K,J})\Big] \\
			&\overset{\eqref{523}}{\le} \sum_{J\in [N]^K} \bP(O_{K,J})  \Big[\sum_{i=1}^K L_{\cA}( X_{i,J},  a_{i,J})
			\delta + g_{\cA}(X_{K,J})\Big]  + \epsilon \\
			&\overset{\eqref{521}}{\le} \overline V(t_0,\lambda_0) + 3\epsilon.
		\end{align*}
		As the error produced by a discretization of the matrix problem is at most $\epsilon$ (uniformly in $n$),  we deduce that
		\[
		\limsup_{n \to \infty} \widehat{V}_{M_n(\bC)}(t_0,x_0^n) \leq \overline{V}(t_0,\lambda_0) + 4\epsilon.
		\]
		Since $\epsilon>0$ was arbitrary, the proof is complete.
		
		~
		
		
		
		\textbf{Lower bound:}  We now prove the lower bound  
		$$
		\liminf_{n\rightarrow \infty} \widehat{V}_{M_n(\bC)}(t_0,x_0^n) \geq \overline{V}(t_0,\lambda_0).
		$$    
		The proof structure is similar, starting with near optimizers of the finite-dimensional problem, with some key differences. Convergence can no longer be obtained by the polynomial approximation used in the upper bound argument because the approximations would need to be uniform in the matrix dimension for different control policies. Instead, we should rely on compactness in the space of laws combined with a selection principle to guarantee the concentration of the controls; the discretization of the classical probability space makes this much easier technically.
		
		\textbf{Lower bound, step 1 (discretization):}
		We again use  Proposition \ref{prop:discretization} to reduce to the case of finite time intervals and a finite probability space for the common noise, this time starting with the $n\times n$ matrix problems. Given any $\epsilon > 0$, we may find $N$, $K$, $R$, and admissible controls $(\alpha_{i,J}^{n})_{i,J}$ for the $n\times n$ matrix version of  discretized problems such that for all $ n\in \mathbb{N}$,
		\begin{align} \label{660}
			\widehat{V}_{M_n(\bC)}(t_0,x_0^n) \geq \sum_{J\in [N]^K} \bP(O_{K,J})\widehat{\mathbb{E}}\big[\sum_{i=1}^K L_{M_n(\bC)}({X}_{i,J}^{n},\alpha_{i,J}^{n})\, \delta + g_{M_n(\bC)}(X_{K,J}^n)\big] - \epsilon,
		\end{align}
		and at the same time  for all $\cA \in \bW$ {containing  a discretized semi-circular process} and $x_0 \in \cA$ such that $\lambda_{x_0} = \lambda_0$, 
		\begin{align} \label{661}
			V_{\cA}^{K,N,R}(t_0,x_0) \geq \widetilde{V}_{\cA * \mathcal{C}}(t_0,x_0) - \epsilon,
		\end{align}
		where  $\mathcal{C}$ denotes a tracial von Neumann algebra generated by an infinite free semi-circular family.

		\textbf{Lower bound, step 2 (asymptotic freeness):}  We fix a multi-index $J \in [N]^K$.  Our goal is construct a limiting near-optimal control for the problem with fixed values of $(W_{i,J}^0)_i$. 
		Recall that $\widehat{\mathcal{F}}_i^n$ denotes the $\sigma$-algebra generated by $\{\widehat{W}^n_{t_j} - \widehat{W}^n_{t_{j-1}}\}_{j=1}^{i}$. 
		In particular, $\widehat{W}_{t_i}^n - \widehat{W}_{t_{i-1}}^n$ is probabilistically independent of $\widehat{\mathcal{F}}_{i-1}^n$ and thus independent  of $\alpha^n_{1,J},\dots,\alpha^n_{i-1,J}$ as well. 
		
		Thus by the asymptotic freeness theorem (Theorem \ref{thm: asymptotic freeness}), we have $\widehat{\bP}$-almost surely that
		\begin{enumerate}
			\item $\lim_{n \to \infty} \norm{\widehat{W}_{t_i}^n - \widehat{W}^n_{t_{i-1}}}_{\infty} = 2(t_i - t_{i-1})^{1/2}$; 
			\item the non-commutative law of $\widehat{W}_{t_i}^n - \widehat{W}_{t_{i-1}}^n$ converges to that of a free semi-circular family where each coordinate has variance $t_i - t_{i-1}$;
			\item $\widehat{W}_{t_i}^n - \widehat{W}_{t_{i-1}}^n$ is asymptotically free from $(x_0^n,\alpha^n_{1,J},\dots,\alpha^n_{i-1,J}, \widehat{W}_{t_1}^n-\widehat{W}_{t_0}^n,\dots,\widehat{W}_{t_{i-1}}^n-\widehat{W}_{t_{i-2}}^n)$.
		\end{enumerate}
		
		\textbf{Lower bound, step 3 (selection):}  
		First, fix an increasing sequence of integers $\{n_k\}_{ k \ge 1}$ such that
		\begin{align} \label{711}
			\lim_{k \to \infty} \widehat{V}_{M_{n_k}(\bC)}^{K,N,R}(t_0,x_0^{n_k}) = \liminf_{n\rightarrow \infty} \widehat{V}_{M_n(\bC)}^{K,N,R}(t_0,x_0^n),
		\end{align}
		which follows because every sequence of real numbers has a subsequence converging to its $\liminf$.  Then we aim to select a further subsequence and a sample $\widehat \omega$ in the probability space $(\widehat{\Omega}, \widehat{\bP})$ where the properties (1) - (3) from  the above asymptotic freeness hold \emph{and} the value of 
		\begin{align*}
			Z^n_J (\widehat \omega):=\sum_{i=1}^K L_{M_n(\bC)}({X}_{i,J}^{n} (\widehat \omega),\alpha_{i,J}^{n}  (\widehat \omega) )\, \delta + g_{M_n(\bC)}(X_{K,J}^n  (\widehat \omega))
		\end{align*}
		is less than its expectation plus $\epsilon$. Note that $J\in [N]^K$ is still fixed.
		We claim that there exists $M>0$, depending only on $R$, $K$, $N$ and the Lipschitz constants of $g$ and $L^0$ where $L(X,\alpha) = L^0(X,\alpha) + c \norm{\alpha}_2^2$, such that   for all $n\in \mathbb{N}$ and $J \in [N]^K$,
		\begin{align} \label{662}
			\widehat{\mathbb{E}}[|Z^n_J|^2] \leq M.
		\end{align} 
		Indeed, $X_{i,J}^n$ is composed of a linear combination of $\{\alpha_{i',J}^n\}_{i'=1,\cdots,i}$, which is bounded in operator norm, {along with $\{w_{i,J}\}_{i'=1,\cdots,i}$, which satisfies $\sup_{i,J} |w_{i,J}| \le 2$ (see \eqref{605}),} and $\{\widehat{W}_{t_i}^n - \widehat{W}_{t_{i-1}}^n\}_{i'=1,\cdots,i}$ which satisfies $\sup_n \sup_{1\le i\le K} \widehat{\mathbb{E}}[ \norm{\widehat{W}_{t_i}^n - \widehat{W}^n_{t_{i-1}}}_{\infty}^4] < \infty$.  It follows from the Lipschitzness of $L^0$ and $g$ that the application of these functions to $X^n_{i,J}$ and $\alpha^n_{i,J}$ yields
		\eqref{662}. 
		
		Let $p: = \widehat{\mathbb{P}}(Z^n_J - \widehat{\mathbb{E}} [Z^n_J] < \epsilon)$. We suppress the notations $n$ and $J$ in $p$ for the sake of readability. Then, using H\"older's inequality  and \eqref{662},
		\begin{align*}
			0 &= \widehat{\mathbb{E}}[(Z^n_J - \widehat{\mathbb{E}} [Z^n_J]) \mathbbm{1}_{Z^n_J - \widehat{\mathbb{E}}[Z^n_J] \geq \epsilon}] + \widehat{\mathbb{E}}[(Z^n_J - \widehat{\mathbb{E}} [Z^n_J]) \mathbbm{1}_{Z^n_J - \widehat{\mathbb{E}}[Z^n_J] <\epsilon}] \\
			&\geq \epsilon (1 - p) - \big( \widehat{\mathbb{E}}[(Z_J^n- \widehat{\mathbb{E}}[Z^n_J])^2]\big)^{1/2} p^{1/2} \\
			&\geq \epsilon(1-p) - M^{1/2} p^{1/2}.
		\end{align*}
		Thus, $\epsilon(1 - p) \leq M^{1/2} p^{1/2}$, and hence $\epsilon^2 (1-p)^2 \leq M p$.  By rearranging, we get
		\[
		\epsilon^2 \leq Mp + 2p \epsilon^2 - p^2 \epsilon^2 \leq (M+2\epsilon^2) p.
		\]
		Hence,
		\[
		\widehat\bP(Z^n_J - \widehat{\mathbb{E}} [Z^n_J] < \epsilon) = p \geq \frac{\epsilon^2}{M + 2 \epsilon^2}.
		\]
		Note that this bound is independent of $n$. 
		
		{Now let
			\[
			T_k := \frac{1}{k} \sum_{j=1}^k \mathbbm{1}_{Z^{n_j}_J - \widehat{\mathbb{E}} [Z^{n_j}_J] < \epsilon},
			\]
			and
			\[
			T := \limsup_{k \to \infty} T_k.
			\]
			Note that $T > 0$ implies that $Z^{n_k}_J - \widehat{\mathbb{E}} [Z^{n_k}_J] < \epsilon$ for infinitely many $k$.  Since $0 \leq T_k \leq 1$, Fatou's lemma applied to $1 - T_k$ implies that
			\[
			\mathbb{E} T \geq \limsup_{k \to \infty} \mathbb{E} T_k = \limsup_{k \to \infty} \frac{1}{k} \sum_{j=1}^k \bP (Z^{n_j}_J - \widehat{\mathbb{E}} [Z^{n_j}_J] < \epsilon) \geq \frac{\epsilon^2}{M + 2 \epsilon^2}.
			\]
			Therefore, $T > 0$ with positive probability, and hence with positive probability, we have that $Z^{n_k}_J - \widehat{\mathbb{E}} [Z^{n_k}_J] < \epsilon$ infinitely often.}
		Because of the almost sure statements about asymptotic freeness in the previous step, there will be some sample $\widehat \omega_J \in \widehat{\Omega}$ where these statements hold and also $Z^{n_k}_J (\widehat \omega_J) - \widehat{\mathbb{E}} [Z^{n_k}_J] < \epsilon$ infinitely often.   In particular, taking such an $\widehat \omega_J$ and passing to a further subsequence $\{n_k\}_{k \ge 1}$, we can arrange that 
		\begin{align} \label{633}
			\sum_{i=1}^K L_{M_{n_k}(\bC)}({X}_{i,J}^{n_k}(\widehat \omega_J),\alpha_{i,J}^{n_k} (\widehat \omega_J) )\, \delta & + g_{M_{n_k}(\bC)}(X_{K,J}^{n_k}(\widehat \omega_J)) \nonumber \\
			&\leq \widehat{\mathbb{E}}  \Big[\sum_{i=1}^K L_{M_{n_k}(\bC)}({X}_{i,J}^{n_k},\alpha_{i,J}^{n_k})\, \delta + g_{M_{n_k}(\bC)}(X_{K,J}^{n_k})  \Big]+ \epsilon,
		\end{align}
		and in addition
		\begin{enumerate}
			\item $\lim_{k \to \infty} \norm{\widehat{W}_{t_i}^{n_k}(\widehat \omega_J)-\widehat{W}_{t_{i-1}}^{n_k}(\widehat \omega_J)}_{\infty} = 2(t_i-t_{i-1})^{\frac{1}{2}}$;
			\item the non-commutative law of $\widehat{W}_{t_i}^{n_k}(\widehat \omega_J)-\widehat{W}_{t_{i-1}}^{n_k}(\widehat \omega_J)$ converges to that of a free semi-circular family;
			\item $\widehat{W}_{t_i}^{n_k}(\widehat \omega_J)-\widehat{W}_{t_{i-1}}^{n_k}(\widehat \omega_J)$ is asymptotically free from
			\[
			(\alpha_{1,J}^{n_k}(\widehat \omega_J),\dots,\alpha_{i-1,J}^{n_k}(\widehat \omega_J), \widehat{W}_{t_1}^{n_k}(\widehat \omega_J) - \widehat{W}_{t_0}^{n_k}(\widehat \omega_J),\dots,\widehat{W}_{t_{i-1}}^{n_k}(\widehat \omega_J)-\widehat{W}_{t_{i-2}}^{n_k}(\widehat \omega_J)).
			\]
		\end{enumerate}
		We furthermore note that we can arrange that the subsequence $(n_k)_{k \in \bN}$ is the same for all $J \in [N]^K$, by iterating through the choices of $J$ and choosing for each $J$ a further subsequence of the current one.
		
		\textbf{Lower bound, step 4 (compactness)}: Let $J\in [N]^K$ be fixed. Note that the matrices $\widehat{W}_{t_i}^{n_k}(\widehat \omega_J)-\widehat{W}_{t_{i-1}}^{n_k}(\widehat \omega_J)$ and $\alpha_{i,J}^{n_k}(\widehat \omega_J)$ are uniformly bounded in operator norm as $k \to \infty$.  Therefore, by compactness of the space of non-commutative laws whose  norms are bounded by $R+1$, by passing to further subsequence, we can arrange that the sequence of joint laws of $(x_0^{n_k},\widehat{W}_{t_1}^{n_k}(\widehat \omega_J)-\widehat{W}_{t_0}^{n_k}(\widehat \omega_J),\dots,\widehat{W}_{t_K}^{n_k}(\widehat \omega_J)-\widehat{W}_{t_{K-1}}^{n_k}(\widehat \omega_J), \alpha_{1,J}^{n_k}(\widehat \omega_J),\dots,\alpha_{K,J}^{n_k}(\widehat \omega_J))$ has some limit point which is the law of some tuple $(x_0,S_{t_1}-S_{t_0},\dots,S_{t_K}-S_{t_{K-1}},\alpha_{1,J},\dots,\alpha_{K,J})$ in some tracial von Neumann algebra $\mathcal{A}_J$.  Note that as the law of $x_0^n$ converges weakly* to  $\lambda_0$, we have $\lambda_{x_0} = \lambda_0.$ 
		Here $S_{t_i} - S_{t_{i-1}}$ is a semi-circular increment freely independent of $(x_0,S_{t
			_1} - S_{t_{0}},\dots,S_{t
			_{i-1}} - S_{t_{i-2}},\alpha_{1,J},\dots,\alpha_{i-1,J})$.  Thus, by the assumption that $L_{{\cA}_J}$ and $g_{{\cA}_J}$ are continuous with respect to convergence in law,  denoting  by $\widetilde X_{i,j}$ the solution defined in \eqref{eqn:discrete_dynamics},
		\begin{align} \label{666}
			\sum_{i=1}^K L_{\cA_J}({X}_{i,J},\alpha_{i,J})\, \delta &+ g_{\cA_J} (X_{K,J}) \nonumber \\
			&\leq   \lim_{k\to \infty} \Big[ \sum_{i=1}^K L_{M_{n_k}(\bC)}({X}_{i,J}^{n_k}(\widehat \omega_J),\alpha_{i,J}^{n_k}(\widehat \omega_J)) \, \delta + g_{M_{n_k}(\bC)}(X_{K,J}^{n_k}(\widehat \omega_J)) \Big]\nonumber\\
			&\overset{\eqref{633}}{\le} \liminf_{k \to \infty}\widehat{\mathbb{E}}  
			\Big[\sum_{i=1}^K L_{M_{n_k}(\bC)}({X}_{i,J}^{n_k},\alpha_{i,J}^{n_k})\, \delta + g_{M_{n_k}(\bC)}(X_{K,J}^{n_k})\Big] + \epsilon.
		\end{align}

		\textbf{Lower bound, step 5 (patching):}  The preceding argument considered a fixed $J \in [N]^K$ and found a candidate for the discretized free problem for that value of $J$ in a tracial von Neumann algebra $\mathcal{A}_J$.  Next, we want to show that these can all be taken in the same tracial von Neumann algebra with the same free Brownian motion.  To this end, let $\mathcal{A}_{i,J}$ be the tracial von Neumann algebra generated by $x_0$, $S_{t_1} - S_{t_{0}}$, \dots, $S_{t_i} - S_{t_{i-1}}$, $\alpha_{1,J}$, \dots, $\alpha_{i,J}$ in $\mathcal{A}_J$.   Let $\mathcal{A}_0:= \mathrm{W}^*(x_0)$.  Then define $\mathcal{A}_i$ inductively as follows.  Assume that $\mathcal{A}_0 \subseteq 
		\mathcal{A}_1 \subseteq \dots \subseteq \mathcal{A}_{i-1}$ have been chosen so that for each $1\le j \le i-1 $, the algebra $\mathcal{A}_j$ contains $S_{t_j} - S_{t_{j-1}}$ and $\alpha_{j,J}$ for all $J \in [N]^K$, and $S_{t_j} - S_{t_{j-1}}$ is freely independent  of $\mathcal{A}_{j-1}$.  We obtain $\mathcal{A}_i$ by an iterated amalgamated free product construction with iterations indexed by $J \in [N]^K$ (ordered lexicographically for concreteness, although the particular order we choose does not matter).  We start out with $\mathcal{A}_{i-1} * \mathrm{W}^*(S_{t_i} - S_{t_{i-1}})$.  At each stage, we take the amalgamated free product of the existing algebra with $\mathcal{A}_{i,J}$ for all $J\in [N]^K$ with amalgamation over $\mathcal{A}_{i-1,J} * \mathrm{W}^*(S_{t_i} - S_{t_{i-1}})$ which can be viewed as a subalgebra of
		$\mathcal{A}_{i,J}$.  This produces a common algebra $\mathcal{A}_i$ with the desired properties.
		
		\textbf{Lower bound, step 6 (conclusion)}: We now have a common filtration of von Neumann algebras  $\mathcal{A}_1 \subseteq \dots \subseteq \mathcal{A}_K$ such  that  $\mathcal{A}_K$ contains all $x_0$, $S_{t_1} - S_{t_{0}}$, \dots, $S_{t_K} - S_{t_{K-1}}$ and $\alpha_{1,J}$, \dots, $\alpha_{K,J}$ for all $J \in [N]^K$ in a natural way.  Then consider the discrete probability space $[N]^K$ with the natural classical filtration of subproducts.  We can view $\alpha_{i,J}$ as a random element here as $J$ ranges over the discrete classical probability space $[N]^K$, and call this random variable $\alpha_i$.  Then the semi-circular increments $S_{t_1} - S_{t_{0}}$, \dots, $S_{t_K} - S_{t_{K-1}}$ and the controls $\alpha_1$, \dots, $\alpha_K$ provide a candidate for the discretized problem in the free setting, so that letting $(X_{i,J})_{i,J}$ be the associated stochastic process,
		\begin{align} \label{412}
			V_{\mathcal{A}_K}^{K,N,R}(t_0,x_0) &\leq  \sum_{J \in [N]^K}\bP(O_{K,J}) \left[ \sum_{i=1}^N L_{\cA_K}(X_{i,J},\alpha_{i,J})\, \delta + g_{\cA}(X_{K,J}) \right] \nonumber  \\
			&\overset{\eqref{666}}{\le} \sum_{J \in [N]^K} \bP(O_{K,J})\liminf_{k \to \infty} \widehat{\mathbb{E}} \left[ \sum_{i=1}^K L_{M_{n_k}(\bC)}({X}_{i,J}^{n_k},\alpha_{i,J}^{n_k})\, \delta + g_{M_{n_k}(\bC)}(X_{K,J}^{n_k}) \right] + \epsilon \nonumber  \\
			&\overset{\eqref{660}}{\le}  \liminf_{k \to \infty} \widehat{V}_{M_{n_k}(\bC)}(t_0,x_0^{n_k}) + 2 \epsilon \overset{\eqref{711}}{=}  
			\liminf_{n \to \infty} \widehat{V}_{M_{n}(\bC)}(t_0,x_0^{n}) + 2 \epsilon  
		\end{align}
		(recall that $\{n_k\}_{k  \ge 1}$ are equal for all  $J\in [N]^K$).
		Therefore, noting that $\lambda_{x_0} =\lambda_0,$ 
		\begin{align*}
			\overline{V}(t_0,\lambda_0) =  \overline{V}_{\mathcal{A}_K * \cC }(t_0,x_0)& \le  \widetilde{V}_{\mathcal{A}_K * \cC}(t_0,x_0)  \\
			&
			\overset{\eqref{661}}{\le}  V_{\mathcal{A}_K}^{K,N,R}(t_0,x_0) + \epsilon \overset{\eqref{412}}{\le}  \liminf_{n \to \infty} \widehat{V}_{M_n(\bC)}(t_0,x_0^n) + 3 \epsilon.
		\end{align*}
		Since $\epsilon>0$ was arbitrary, the proof is complete.
	\end{proof}
	
	\begin{remark}
		The $E$-convexity of $L$ and $g$ in Assumption \textbf{B} plays a pivotal role in establishing the upper bound on $n\times n$ matrix-version value function. In a general von Neumann algebra $\mathcal B$—one that contains both the initial datum and a free semicircular Brownian motion—controls need not be approximable by matrices (e.g.\ when $\mathcal B$ fails to be Connes-embeddable). However, $E$-convexity allows us to restrict attention to the free product of the subalgebras generated by the initial condition and by the semicircular motion. This reduced algebra admits matrix approximations, which in turn yields the desired upper bound.
		
		On the other hand, no $E$-convexity condition on $L$ and $g$ is needed to prove the corresponding lower bound  on the $n\times n$ matrix-value function. Indeed any near-optimal controls in the $n\times n$ matrix   problem already provides a candidate in the free control formulation. 
	\end{remark}

	\section{Application to Large Deviations} \label{sec: large deviations}
	
	In this section, we offer a new perspective that \emph{heuristically} derives a large deviation principle  of the $d$-tuple of  GUE($n$) random matrices  $ \widehat{W}^n_1 $ on $M_n(\bC)_{\sa}^d$ for general $d \ge 1$ (compare \cite{biane2003large}), by leveraging the convergence of the value function associated with the Hamilton-Jacobi equation.  The following Bryc's theorem serves as a key ingredient in the large deviation analysis. We refer to \cite[\S 3.3]{rassoul2015course} for the proof.
	
	\begin{proposition}[Bryc's Theorem]
		Let $\{X_n\}_{n  \ge 1}$ be a family of random variables taking values in a metric space $\mathcal{X}$.
		Assume that   $\{X_n\}_{n  \ge 1}$ is exponentially tight, i.e.  for every $L \in \mathbb R$, there exists a compact set $K_L \subseteq \mathcal{X}$ such that
		\[
		\limsup_{n \to \infty}  \frac{1}{n}  \log \mathbb{P}(X_n \notin K_L) \le -L,
		\]
		and for every bounded continuous function $f \in C_b(\mathcal{X})$, the limit
		\begin{align} \label{expmoment}
			\Lambda(f) := \lim_{n \to \infty}  \frac{1}{n} \log  \mathbb E e ^{nf(X_n)} 
		\end{align}
		exists. Then, the law of $\{X_n\}_{n  \ge 1}$ on $\mathcal{X}$  satisfies a large deviation principle  with speed $n$ and a  rate function $I: \mathcal{X} \to [0, \infty]$ given by the Legendre–Fenchel transform  
		\[
		I(x) := \sup_{f \in C_b(\mathcal{X})} \left\{ f(x) - \Lambda(f) \right\}.
		\]
		In other words, for any  Borel measurable set  $E \subseteq \mathcal X,$
		\[
		-\inf_{x \in \operatorname{int}(E)} I(x)
		\le 
		\liminf_{n\to\infty} \frac{1}{n} \log \mathbb{P}(X_n \in E)
		\le 
		\limsup_{n\to\infty} \frac{1}{n} \log \mathbb{P}(X_n \in E)
		\le 
		-\inf_{x \in \overline{E}} I(x).
		\]
		
	\end{proposition}
	
	
	With the aid of Bryc's theorem, we \emph{heuristically} obtain a large deviation rate function for $\arctan\widehat{W}^n_1$. The motivation for working with $\arctan\widehat{W}^n_1$ rather than $\widehat{W}^n_1$ itself is that $\arctan\widehat{W}^n_1$ is uniformly bounded in operator norm and therefore takes values in a metrizable space, allowing us to apply Bryc's theorem.
	Let $\mathcal X$   be the space of non-commutative laws $\Sigma_{d,\pi/2}$ endowed with the weak-* topology (see Section \ref{subsecNon-commutativeLaws} for details). Note that this space is metrizable (see \eqref{metric}) and compact.  The computation of the exponential moment \eqref{expmoment} can be carried out by invoking our convergence result for the matrix value function. 
	%

	To  analyze the exponential moment,
	let $\varphi=(\varphi_\cA)_{\cA \in \bW}$ be a Lipschitz and locally bounded $W^*$-tracial function 
	such that $\varphi \circ \arctan$ is $E$-convex. In other words, $\varphi \circ \arctan$ is regarded as a terminal cost in a value function, satisfying Assumptions \textbf{A}, \textbf{B} and \textbf{C}.
	Define
	\begin{align}
		\widehat \Lambda_n(\varphi):= -\frac{1}{n^2}\log   \mathbb{E}\big[ e^{-n^2  \varphi_n(\arctan \widehat{W}^n_1) }  \big],
	\end{align}
	where we abbreviate the notation $\varphi_n := \varphi_{M_n(\mathbb C)^d_{\text{sa}}}.$

	To evaluate this quantity, we use the following formula, which can be deduced from classical stochastic control theory or from the Bou{\'e}--Dupuis formula \cite{BoueDupuis1998variational}.  Since we are working on the matrix space $M_n(\bC)_{\sa}^d$ while the formula is usually stated for Euclidean space, we explain the normalization need to translate between the two settings.
	
	\begin{lemma}\label{bd}
		Let $\psi_n: M_n(\bC)_{\sa}^d \to \bR$ be bounded and Borel-measurable.  Let $(\widehat{W}_t^n)_{t \in [0,1]}$ be a $\textup{GUE}(n)$ Brownian motion on $M_n(\bC)_{\sa}^d$.  Then
		\[
		-\frac{1}{n^2} \log \mathbb{E} [e^{-n^2 \psi_n(\widehat{W}_1^n)}] = \inf_{\alpha^n} \mathbb{E} \left[ \frac{1}{2} \int_0^1 \norm{\alpha^n_t}_2^2\,dt + \psi_n\left(\widehat{W}_1^n + \int_0^1 \alpha^n_t\,dt \right) \right],
		\]
		where the infimum is over all controls $\alpha^n: [0,1] \to M_n(\bC)_{\sa}^d$ that are progressively measurable with respect to the filtration generated by $(\widehat{W}_t^n)_{t \in [0,1]}$.
	\end{lemma}
	
	\begin{proof}
		Recall that $M_n(\bC)_{\sa}$ is equipped with the inner product $\ip{A,B} = \tr_n(AB)$.  We define a linear isometry $\bR^{n^2} \to M_n(\bC)_{\sa}$ using the orthonormal basis in \S \ref{subsec: value function definitions}, and let $\mathcal T: \bR^{dn^2} \to M_n(\bC)_{\sa}^d$ be the direct sum of $d$ copies.  Recall from \S \ref{subsec: value function definitions} that the GUE($n$) Brownian motion $(\widehat{W}_t^n)_{t \in [0,1]}$ is given by summing independent standard Brownian motions multiplied by each basis element and then dividing by $n$.  The result is that $B^n_t = n \mathcal T^{-1} (\widehat{W}_t^n)$ is a standard Brownian motion on $\bR^{dn^2}$.  We can apply \cite[Theorem 1]{BoueDupuis1998variational} to $n^2 \psi_n(n^{-1} \mathcal T (B^n_1))$ viewed as a functional of $(B^n_t)_{t \in [0,1]}$ to obtain that
		\[
		-\log \mathbb{E} [e^{-n^2 \psi_n(n^{-1} \mathcal T (B^n_1))}] = \inf_{\beta^n} \mathbb{E} \left[ \frac{1}{2}\int_0^1 \norm{\beta^n_t}_2^2\,dt + n^2 \psi_n\left(n^{-1}\mathcal T\left(B^n_t + \int_0^1 \beta^n_t\,dt \right) \right) \right],
		\]
		where $\beta^n: [0,1] \to \bR^{dn^2}$ is progressively measurable.  Letting $\alpha^n_t := n^{-1} \mathcal T (\beta^n_t)$, we thus obtain
		\[
		-\log \mathbb{E} [e^{-n^2 \psi_n(\widehat{W}_1^n)}] = \inf_{\alpha^n} \mathbb{E} \left[ \frac{1}{2} \int_0^1 n^2 \norm{\alpha^n_t}_2^2\,dt + n^2 \psi_n\left(\widehat{W}_1^n + \int_0^1 \alpha^n_t\,dt \right) \right],
		\]
		and after dividing both sides by $n^2$, we obtain the asserted formula.
	\end{proof}
	
	Motivated by this, for $\cA \in \bW$, define
	\begin{align}
		\widetilde \Lambda_{\cA} (\varphi):= \inf_{\alpha} \Bigl[ \frac{1}{2} \int_0^1 \|\alpha_t \|_{L^2(\cA)}^2 dt  +    (\varphi\circ \arctan) \Big(S_1  + \int_0^1 \alpha _t dt \Big) \Bigr],
	\end{align}
	where the control belongs to $\bA_{\cA}^{0,1}$ and is classically deterministic with initial condition $0$. Observe that this corresponds to the value function $\widetilde{V}_\cA(t_0,x_0)$ in \eqref{eqn:l2_value} with a quadratic Lagrangian and a terminal cost function $\varphi$ (at terminal time $T=1$), evaluated at $t_0=0$ and $x_0=0$.  Moreover, in light of Lemma \ref{lem: reducing to deterministic controls}, since there is no common noise $(W^0_t)_{t \ge 0}$, the value function can be evaluated using only deterministic controls, which is why we have omitted the expectation in the equation above.
	
	Now, as in \eqref{def:bar}, set
	\begin{align}
		\overline{\Lambda}_{\cA}(\varphi):= \inf_{\iota: \mathcal A \to \mathcal B}  \widetilde  \Lambda_{ \cB}(\varphi),
	\end{align}
	where the infimum is performed over   tracial $W^*$--embeddings  $\iota: \mathcal A \to \mathcal B$. Then $(\overline{\Lambda}_{\cA})_{\cA \in \bW}$ is tracial, and thus we regard this as a function $\overline{\Lambda}$ only depending on the test function $\varphi$.
	
	By Lemma \ref{bd} (with $\psi: = \varphi \circ \arctan$) and the convergence result (Theorem \ref{main thm}),
	\begin{align}\label{bryc}
		\lim _{n\rightarrow \infty}  \left(-\frac{1}{n^2}\log   \mathbb{E}\big[ e^{-n^2  \varphi_n(\arctan \widehat{W}^n_1) }  \big]   \right)   = \lim _{n\rightarrow \infty} \widehat \Lambda_n (\varphi) = \overline{\Lambda} (\varphi).
	\end{align}

	Observe that the validity of Bryc's theorem requires the exponential moment estimates to hold for all bounded and continuous test functions. However, in our case, the estimates are currently established only when $\varphi \circ \arctan$ is  $E$-convex. Nevertheless, assuming that the exponential  estimate \eqref{bryc}  can be extended to all bounded and continuous functions, we are naturally led to the following \emph{candidate} for the large deviation rate function:
	\[
	I(x) = \sup_{\varphi \in C_b(\Sigma_d^\infty )} \left\{ \varphi(x) + \overline{\Lambda} (-\varphi ) \right\},\qquad \forall x \in \Sigma_{d,\pi/2},
	\]
	in other words, we expect that
	for any  Borel measurable set  $E \subseteq \Sigma_{d,\pi/2} $,
	\begin{align*}
		-\inf_{x \in \operatorname{int}(E)} I(x)
		&\le 
		\liminf_{n\to\infty} \frac{1}{n^2} \log \mathbb{P}(  \arctan \widehat{W}^n_1 \in E) \\
		&  \le 
		\limsup_{n\to\infty} \frac{1}{n^2} \log \mathbb{P}(  \arctan \widehat{W}^n_1 \in E)
		\le 
		-\inf_{x \in \overline{E}} I(x).
	\end{align*}
	
	Let us comment briefly on how the approach here compares with previous work.  Biane, Capitaine, and Guionnet \cite[\S 5.1]{biane2003large} defined a rate function for non-commutative Brownian motion by a variational formula.  They proved a general large deviation upper bound as well as a lower bound that holds under strong regularity assumptions.  Dabrowski later was able to show the corresponding Laplace principle holds with two-sided bounds but only for convex functionals \cite{Dabrowski2017Laplace}.  By specializing \cite{biane2003large} to test functions that only look at the Brownian motion at time $1$, one obtains a conjectural large deviation rate function for the GUE-tuple itself.  The variational formula in \cite{biane2003large} is expressed in terms of the Malliavin calculus; the infimum is taken over non-commutative polynomials in the generators and the Malliavin derivative of the polynomial plays a role similar to the control $\alpha_t$ here.
	
	As we have stressed in \cite{2025viscosity}, one of the difficulties of optimization in the non-commutative setting is that one must decide which von Neumann algebra to use as the ambient algebra, since there are many non-isomorphic options.  Our framework allows the ambient algebra to become as large as possible, and in an arbitrary algebra without the finer structure of the non-commutative Wiener space, we cannot expect the optimizer to always be expressible through Malliavin derivatives, and so a more general framework for variational problems was needed.  In fact, the assumptions that \cite{biane2003large} used for the lower bound actually imply that the process lives in the von Neumann algebra generated by a free Brownian motion as a consequence of the uniqueness result in \cite[Theorem 6.1]{biane2003large}.  It is quite possible that the ``correct'' choice of ambient algebra for the large deviation principle in general is something in between the smallest possible algebra (that of a free Brownian motion and the initial condition) and the largest possible (arbitrary larger algebras).  Our work here shows that the choice of ambient algebra makes little difference when the functions in question are $E$-convex, giving another heuristic for why the convex functional case studied by Dabrowski is easier than the general case.  Without such convexity assumptions, the problem seems to be much more subtle; this is similar to the case of mean field games, where convergence arguments become much more difficult without convexity.  We hope that the parallel developments in these two areas will provide insight for future work.

	\appendix
	
	\section{Asymptotic freeness theorem} \label{sec: asymptotic freeness proof}
	
	This section gives the proof for the version of the asymptotic freeness theorem used in this paper (Theorem \ref{thm: asymptotic freeness}).  We remark that all the results in this section are well-known in random matrix theory, so we make no claim of originality, but we include some detail for the sake of readers from other areas.
	
	It will be convenient to use Haar random unitary matrices for the argument.  Recall that the Haar measure on $\cU_n$ (also sometimes called the uniform measure on $\cU_n$) is the unique Borel probability measure $\eta_n$ on $\cU_n$, satisfying the left-invariance condition
	\[
	\eta_n(US)=\eta_n(S)
	\]
	for every $U\in \cU_n$ and all Borel sets $S \subset \cU_n$.  The Haar measure also satisfies right-invariance $\eta_n(SU) = \eta_n(S)$ (for background, see e.g.~\cite{Meckes2019}).  A \emph{Haar random unitary matrix} is a random element of $\cU_n$ whose probability distribution is the Haar measure; more explicitly it is a function $V$ from a probability space $(\Omega,\mathbb P)$ into $\cU_n$ such that for every Borel set $S \subset \cU_n$ we have 
	\[
	\mathbb P\big[ V \in S \big]=\eta_n(S).
	\]
	For the proof of Theorem \ref{thm: asymptotic freeness}, we will use the observation that if $U_1^{(n)}$, \dots, $U_{m'}^{(n)}$ are independent $n \times n$ random unitary matrices chosen according to the Haar measure on the unitary group, then $(U_1^{(n)} S_1^{(n)} (U_1^{(n)})^*,\dots,U_{m'}^{(n)} S_{m'}^{(n)} (U_{m'}^{(n)})^*)$ has the same  distribution as $(S_1^{(n)},\dots,S_{m'}^{(n)})$ as random variables in $M_n(\bC)_{\sa}^{m'}$ (see the proof below for details). 
	Then we will apply concentration of measure estimates for the unitaries together with the following version of Voiculescu's asymptotic freeness theorem.
	
	\begin{theorem}[{Voiculescu's asymptotic freeness \cite[Corollary 2.5]{voiculescu1998strengthened}}] \label{thm: unitary asymptotic freeness}
		Let $m,k$ be  positive integers.
		Let $U_1^{(n)}$, \dots, $U_m^{(n)}$ be independent $n \times n$ Haar unitary random matrices. Let $X_1^{(n)}$, \dots, $X_k^{(n)}$ be  $n\times n$ deterministic matrices with $\sup_n \sup_{j=1,\cdots,k} \norm{X_j^{(n)}} \leq r$ for some $r<\infty$.  Then for any indices $i_1 \neq i_2 \neq \dots  \neq i_k$ in $\{1,\cdots,m\}$,
		\[
		\lim_{n \to \infty} \Big\vert \mathbb{E} \tr_n \left[U_{i_1}^{(n)}(X_1^{(n)} - \tr_n(X_1^{(n)}))(U_{i_1}^{(n)})^* \dots U_{i_k}^{(n)}(X_k^{(n)} - \tr_n(X_k^{(n)}))(U_{i_k}^{(n)})^* \right]\Big\vert = 0.
		\]
	\end{theorem}

	To obtain almost sure convergence to zero for $\tr_n [U_{i_1}^{(n)}(X_1^{(n)} - \tr_n(X_1^{(n)}))(U_{i_1}^{(n)})^* \dots U_{i_k}^{(n)}(X_k ^{(n)} - \tr_n(X_k^{(n)}))(U_{i_k}^{(n)})^*]$, we will estimate the difference between this quantity and its expectation using concentration of measure, specifically the following concentration bound for several independent unitary matrices. Let $ \mathbb{U}_n^m$ be the $m$-tuple product space of $n\times n$ unitary group $ \mathbb{U}_n$ equipped with the Riemannian metric associated to the inner product $\ip{\cdot,\cdot}_{\tr_n}$.
	
	\begin{lemma} \label{lem: unitary concentration}
		Let $U_1^{(n)}$, \dots, $U_m^{(n)}$ be $n \times n$ independent Haar random unitaries, and let $f: \mathbb{U}_n^m \to \mathbb{C}$ be Lipschitz with respect to $\norm{\cdot}_2$. 
		Then  for any $\delta>0,$
		\[
		\bP(|f(U_1^{(n)},\dots,U_m^{(n)}) - \mathbb{E}[f(U_1^{(n)},\dots,U_m^{(n)})]| \geq \delta) \leq 4 e^{-n^2 \delta^2 / 12 \norm{f}_{\Lip}^2}.
		\]
	\end{lemma}
	
	This bound arises from the following facts.  The unitary group $\mathbb{U}_n$ with the metric associated to the unnormalized trace $\Tr_n$ satisfies the log-Sobolev inequality with constant $6/n$ \cite[Theorem 15]{Meckes2013}.  One can easily deduce the log-Sobolev inequality for the product of several copies of $\mathbb{U}_n$; see e.g.\ \cite[Corollary 5.7]{Ledoux2001}, \cite[Theorem 5.9]{Meckes2019}.  This in turn implies that it satisfies the Herbst concentration estimate; see e.g.\ \cite[Lemma 2.3.3]{anderson2010introduction}, \cite[Theorem 5.5]{Meckes2019}.  After renormalizing the inner product from $\ip{\cdot,\cdot}_{\Tr_n}$ to $\ip{\cdot,\cdot}_{\tr_n}$, one deduces Lemma \ref{lem: unitary concentration}.
	
	~
	
	Note that if $X_1^{(n)},\dots,X_k^{(n)}$ are deterministic matrices whose norms are bounded by $r$, then
	\begin{multline*}
		f(U_1^{(n)},\dots,U_m^{(n)},X_1^{(n)},\dots,X_k^{(n)}) \\
		:= \tr_n \left[U_{i_1}^{(n)}(X_1^{(n)} - \tr_n(X_1^{(n)}))(U_{i_1}^{(n)})^* \dots U_{i_k}^{(n)}(X_k^{(n)} - \tr_n(X_k^{(n)}))(U_{i_k}^{(n)})^* \right]
	\end{multline*}
	is a $2k r^k$-Lipschitz function of $(U_1^{(n)},\dots,U_m^{(n)})$ with respect to $\norm{\cdot}_2$.  Hence,
	\begin{multline*}
		\bP(|f(U_1^{(n)},\dots,U_m^{(n)},X_1^{(n)},\dots,X_k^{(n)}) - \mathbb{E}[f(U_1^{(n)},\dots,U_m^{(n)},X_1^{(n)},\dots,X_k^{(n)})]| \geq \delta) \\
		\leq 4 e^{-n^2 \delta^2 / 12(2k)^2r^{2k}}.
	\end{multline*}
	Because the right-hand-side is summable in $n$, the Borel-Cantelli lemma implies the following result.
	
	\begin{lemma}
		Let $m,k$ be  positive integers.
		Let $U_1^{(n)}$, \dots, $U_m^{(n)}$ be independent $n \times n$ Haar unitary random matrices. Let $X_1^{(n)}$, \dots, $X_k^{(n)}$ be  $n\times n$ deterministic matrices with $\sup_n \sup_{j=1,\cdots,k} \norm{X_j^{(n)}} \leq r$ for some $r<\infty$.  Then for any indices $i_1 \neq i_2 \neq \dots  \neq i_k$ in $\{1,\cdots,m\}$, we have almost surely
		\[
		\lim_{n \to \infty} \tr_n \left[U_{i_1}^{(n)}(X_1^{(n)} - \tr_n(X_1^{(n)}))(U_{i_1}^{(n)})^* \dots U_{i_k}^{(n)}(X_k^{(n)} - \tr_n(X_k^{(n)}))(U_{i_k}^{(n)})^* \right] = 0.
		\]
	\end{lemma}
	
	Next, we upgrade this statement to allow the $X_j^{(n)}$'s to be random matrices independent of the $U_j^{(n))}$'s.
	
	\begin{lemma} \label{lem: unitary moment almost sure convergence}
		Let $m,k$ be  positive integers.  Let $X_1^{(n)}$, \dots, $X_k^{(n)}$ be  $n\times n$ random matrices such that almost surely $\sup_n \sup_{j=1,\cdots,k} \norm{X_j^{(n)}} < \infty$. 
		Let $U_1^{(n)}$, \dots, $U_m^{(n)}$ be independent $n \times n$ Haar unitary random matrices, which are also independent of $(X_1^{(n)},\dots,X_k^{(n)})$.  Then for any indices $i_1 \neq i_2 \neq \dots  \neq i_k$ in $\{1,\cdots,m\}$, we have almost surely
		\[
		\lim_{n \to \infty} \tr_n \left[ \prod_{j=1}^k U_{i_j}^{(n)}(X_j^{(n)} - \tr_n(X_j^{(n)}))(U_{i_j}^{(n)})^* \right] = 0.
		\]
	\end{lemma}
	
	\begin{proof}
		Recall that probabilistic independence means that random variables can be represented on a measure space $(\Omega,\mathbb P)$ that splits as a product space $(\Omega_1,\mathbb P_1) \times (\Omega_2,\mathbb P_2)$, such that $X_j^{(n)}$'s only depend on the first coordinate $\omega_1 \in \Omega_1$ and the $U_j^{(n)}$'s only depend on the second coordinate $\omega_2 \in \Omega_2$. By assumption, for almost every $\omega_1$, we have $\sup_n \sup_{j=1,\cdots,k} \norm{X_j^{(n)}(\omega_1)} < \infty$, and hence for almost every $\omega_2$, we have
		\[
		\lim_{n \to \infty} \tr_n \left[ \prod_{j=1}^k U_{i_j}^{(n)}(\omega_2)(X_j^{(n)}(\omega_1) - \tr_n(X_j^{(n)}(\omega_1)))(U_{i_j}^{(n)}(\omega_2))^* \right] = 0.
		\]
		In other words, for almost every $\omega_1$, for almost every $\omega_2$, the statement that we want holds, and so by the Fubini-Tonelli theorem the statement holds for almost every $(\omega_1,\omega_2)$, since the set where it holds is clearly measurable.
	\end{proof}
	
	\begin{proposition} \label{prop: conjugation for freeness}
		Consider self-adjoint random matrices $Y_{i,i'}^{(n)}$ for $i = 1$, \dots, $m$ and $i' = 1$, \dots, $d_i$, such that $\sup_n \max_{i,i'} \norm{Y_{i,i'}^{(n)}} < \infty$ almost surely.  Let $U_1^{(n)}$, \dots, $U_m^{(n)}$ be independent Haar random unitary matrices that are independent of $(Y_{i,i'}^{(n)}: i = 1, \dots, m; i' = 1, \dots, d_i)$.  Let
		\[
		\widetilde{Y}_i^{(n)} = \left( U_i^{(n)} Y_{i,i'} (U_i^{(n)})^* \right)_{i'=1}^{d_i},\qquad i=1,\dots,m.
		\]
		Then $\widetilde{Y}_1^{(n)}$, \dots, $\widetilde{Y}_m^{(n)}$ are almost surely freely independent.
	\end{proposition}
	
	\begin{proof}
		Fix indices $i_1 \neq \dots \neq i_k$ in $\{1,\cdots,m\}$. For each $j = 1$, \dots, $k$, let $p_j$ be a  $d_{i_j}$-variable non-commutative polynomial.  We show that almost surely
		\[
		\lim_{n \to \infty} \tr_n \left[ \prod_{j=1}^k p_j(\widetilde{Y}_{i_j}^{(n)}) - \tr_n[p_j(\widetilde{Y}_{i_j}^{(n)})] \right] = 0.
		\]
		Let $X_j^{(n)} = p_j((Y_{i_j,i'}^{(n)})_{i'=1}^{d_{i_j}})$.  Note that
		\[
		p_j(\widetilde{Y}_{i_j}^{(n)}) = p_j((U_{i_j}^{(n)} Y_{i,i'}^{(n)} (U_{i_j}^{(n)})^*)_{i'=1}^{d_{i_j}}) = U_{i_j}^{(n)} p_j(( Y_{i,i'}^{(n)} )_{i'=1}^{d_{i_j}}) (U_{i_j}^{(n)})^* = U_{i_j}^{(n)} X_j^{(n)} (U_{i_j}^{(n)})^*.
		\]
		Moreover, for each $j$, we have almost surely $\sup_n \norm{X_j^{(n)}} < \infty$ since $X_j^{(n)}$ is a polynomial function of the $Y_{i,i'}^{(n)}$'s.  Of course, the $U_i^{(n)}$'s are probabilistically independent of the $X_j^{(n)}$'s.  Therefore, by Lemma \ref{lem: unitary moment almost sure convergence}, we have
		\[
		\lim_{n \to \infty} \tr_n \left[ \prod_{j=1}^k p_j(\widetilde{Y}_{i_j}^{(n)}) - \tr_n[p_j(\widetilde{Y}_{i_j}^{(n)})] \right] = \lim_{n \to \infty} \tr_n \left[ \prod_{j=1}^k U_{i_j}^{(n)}(X_j^{(n)} - \tr_n(X_j^{(n)}))(U_{i_j}^{(n)})^* \right] = 0
		\]
		almost surely.
	\end{proof}

	Now, we proceed with the proof of Theorem \ref{thm: asymptotic freeness}, but first we remark one probabilistic subtlety.  In many random matrix theorems, the relationship between the matrix models for different values of $n$ is not explicitly given; for instance, the matrix models for $n = 10$ and $n = 100$ may be independent from each other or they may generate the same $\sigma$-algebra.  The almost sure statements proved above through the (first) Borel-Cantelli lemma work no matter what the relationship between $X^{(n)}$ and $X^{(n')}$ is for $n \neq n'$.  In particular, they work when the $n \times n$ matrices under consideration are mutually independent for different values of $n$.  The second Borel-Cantelli lemma says that if events $A_n$ are independent and $\sum \mathbb P(A_n) = \infty$, then almost surely infinitely many of the $A_n$'s happen; hence, by the contrapositive, if random variables $Z_n$ are independent and $Z_n \to 0$ almost surely, then for each $\epsilon > 0$, we must have that
	\[
	\sum_{n=1}^\infty \mathbb P(|Z_n| \geq \epsilon) < \infty.
	\]
	Then by the first Borel-Cantelli lemma, this implies that if $\widetilde{Z}_n$ is a random variable on a different probability space, where $\widetilde{Z}_n$ has the same distribution as $Z_n$ individually, but the $\widetilde{Z}_n$'s are not necessarily independent for different values of $n$, then in fact we still have $\widetilde{Z}_n \to 0$ almost surely.  The upshot is that if a random matrix theorem stating almost sure convergence holds when the random matrices are assumed to be mutually independent for different values of $n$, then it holds without this assumption as well.
	
	\begin{proof}[Proof of Theorem \ref{thm: asymptotic freeness}]
		As explained above, we assume without loss of generality that the matrix models for different values of $n$ are mutually independent.
		
		Let $Y^{(n)} = (Y_1^{(n)}, \dots, Y_m^{(n)})$ be random self-adjoint matrices with $\sup_n \max_j \norm{Y_j^{(n)}} < \infty$ almost surely, and let $S_1^{(n)}$, \dots, $S_{m'}^{(n)}$ be independent GUE matrices which are independent of $Y^{(n)}$.  Recall from Lemma \ref{lem: GUE operator norm} that $\sup_n \norm{S_j^{(n)}} < \infty$ almost surely for each $j$.  Now let $U_0^{(n)}$, \dots, $U_{m'}^{(n)}$ be independent Haar random unitary matrices independent of $(Y_1^{(n)}, \dots, Y_m^{(n)}, S_1^{(n)},\dots,S_{m'}^{(n)})$, and let
		\[
		\widetilde{Y}^{(n)} = (U_0^{(n)} Y_i^{(n)} (U_0^{(n)})^*)_{i=1}^m, \qquad \widetilde{S}_j^{(n)} = U_j^{(n)} S_j^{(n)} (U_j^{(n)})^* \text{ for } j = 1, \dots, m'.
		\]
		By Proposition \ref{prop: conjugation for freeness}, $\widetilde{Y}^{(n)}$, $\widetilde{S}_1^{(n)}$, \dots, $\widetilde{S}_{m'}^{(n)}$ are almost surely asymptotically free.
		
		This in turn implies that
		\[
		Y^{(n)}, (U_0^{(n)})^* U_1^{(n)} S_1^{(n)} (U_1^{(n)})^* U_0^{(n)}, \dots, (U_0^{(n)})^* U_{m'}^{(n)} S_{m'}^{(n)} (U_{m'}^{(n)})^* U_0^{(n)}
		\]
		are almost surely asymptotically free.  The reason for this is that the trace of any polynomial expression $(\widetilde{Y}_1^{(n)},\dots,\widetilde Y_m^{(n)},S_1^{(n)},\dots,S_{m'}^{(n)})$ is unchanged when we conjugate each of the matrices by $U_0^{(n)}$, and the expressions that are required to vanish for asymptotic free independence are all composed from traces of such polynomials.
		
		Finally, we note that the joint probability distribution of
		\[
		((U_0^{(n)})^* U_1^{(n)} S_1^{(n)} (U_1^{(n)})^* U_0^{(n)}, \dots, (U_0^{(n)})^* U_{m'}^{(n)} S_{m'}^{(n)} (U_{m'}^{(n)})^* U_0^{(n)})
		\]
		is the same as that of $(S_1^{(n)},\dots,S_m^{(n)})$.  Indeed, the joint density of $(S_1^{(n)},\dots,S_{m'}^{(n)})$ with respect to Lebesgue measure on $M_n(\bC)_{\sa}^{m'}$ is a constant times $\exp(-n^2 \sum_{j=1}^{m'} \tr_n(X_j^2))$, which is invariant under the substitution of $U_j X_j U_j^*$ for any fixed unitary matrices $U_1$, \dots, $U_{m'}$.  Hence, for any fixed unitaries, the joint distribution of $(U_j S_1^{(n)} U_j^*,\dots,U_{m'} S_{m'}^{(n)} U_{m'}^*)$ is the same as $(S_1^{(n)},\dots,S_m^{(n)})$.  Using the Fubini-Tonelli theorem in a somewhat similar way as in Lemma \ref{lem: unitary moment almost sure convergence}, the same holds if the $U_j$'s are random and independent of $X_j$.
		
		Therefore, $((U_0^{(n)})^* U_1^{(n)} S_1^{(n)} (U_1^{(n)})^* U_0^{(n)}, \dots, (U_0^{(n)})^* U_{m'}^{(n)} S_{m'}^{(n)} (U_{m'}^{(n)})^* U_0^{(n)})$ has the same distribution as $(S_1^{(n)},\dots,S_m^{(n)})$, and of course it is still probabilistically independent of $Y^{(n)}$.  It follows in turn that $(Y_1^{(n)},\dots,Y_m^{(n)},S_1^{(n)},\dots,S_{m'}^{(n)})$ has the same probability distribution as
		\[
		(Y_1^{(n)},\dots,Y_m^{(n)},(U_0^{(n)})^* U_1^{(n)} S_1^{(n)} (U_1^{(n)})^* U_0^{(n)}, \dots, (U_0^{(n)})^* U_{m'}^{(n)} S_{m'}^{(n)} (U_{m'}^{(n)})^* U_0^{(n)}).
		\]
		In fact, this statement is true even when considering the joint distribution of the matrix models over different values of $n$, since we assumed that all our matrix models for different values of $n$ are mutually independent.  Thus, we obtain that $(Y^{(n)}, S_1^{(n)},\dots, S_{m'}^{(n)})$ are almost surely asymptotically freely independent as desired.
	\end{proof}

	\section{Estimates for normal random variables} \label{sec: normal estimates}
	
	This section records properties of Gaussian random variables used for discretizing the stochastic control problems.  The first lemma provides the conditional variance of Gaussian random variables.
	\begin{lemma} \label{cond variance} 
		Let $Z$ be the standard Gaussian random variable. Then for any $z >0$,
		\begin{align*}
			\textup{Var} (Z \mid Z \ge z) \le 1
		\end{align*}
		and
		\begin{align*}
			\textup{Var} (Z \mid Z \le -z) \le 1.
		\end{align*}
	\end{lemma}
	\begin{proof}
		The density function of $Z$, conditioned on $Z \ge z$ is 
		\begin{align*}
			h_z(x):= \frac{1}{\bP(Z \ge z)}  \frac{1}{\sqrt{2\pi}}e^{-x^2/2},\qquad x \ge z.
		\end{align*}
		Thus
		\begin{align} \label{900}
			\textup{Var} (Z \mid Z \ge z) &= \int_z^\infty x^2h_z(x)  dx  - \Big(\int_z^\infty xh_z(x)  dx\Big)^2 \nonumber  \\
			&= \frac{1}{\bP(Z \ge z)}   \frac{1}{\sqrt{2\pi}} \Big(ze^{-z^2/2} + \int_z^\infty e^{-x^2/2} dx \Big) - \Big(\frac{1}{\bP(Z \ge z)}  \frac{1}{\sqrt{2\pi}}e^{-z^2/2} 
			\Big)^2  \nonumber   \\
			&= 1 +\frac{1}{\bP(Z \ge z)}   \frac{1}{\sqrt{2\pi}}   ze^{-z^2/2} -\Big(\frac{1}{\bP(Z \ge z)}  \frac{1}{\sqrt{2\pi}}e^{-z^2/2} 
			\Big)^2   \nonumber   \\
			&= 1 +\frac{1}{\bP(Z \ge z)^2}   \frac{1}{\sqrt{2\pi}}   ze^{-z^2/2} \cdot \Big[ \bP(Z \ge z) -  \frac{1}{\sqrt{2\pi}}\frac{1}{z} e^{-z^2/2}  \Big].
		\end{align}
		Note that by the integration by parts,
		\begin{align*}
			\bP(Z \ge z) =  \frac{1}{\sqrt{2\pi}}   \int_z^\infty e^{-x^2/2} dx = \frac{1}{\sqrt{2\pi}}  \Big(  \frac{1}{z}e^{-z^2/2} - \int_z^\infty \frac{e^{-x^2/2}}{x^2}dx  \Big) \le \frac{1}{\sqrt{2\pi}}   \frac{1}{z}e^{-z^2/2}.
		\end{align*}
		Hence the quantity \eqref{900} is at most 1, proving the first statement of the lemma. The second statement is obtained by symmetry.
	\end{proof}

	\begin{lemma} \label{gauss}
		Let $Z$ be the standard Gaussian random variable. Then for $ K \ge 1,$
		\begin{align*}
			\bE (Z \mid Z \ge K) \le 2K.
		\end{align*}
	\end{lemma}
	\begin{proof}
		This easily follows from the argument in the proof of  Lemma \ref{cond variance}. Indeed, denoting by $h_K(x)$  the density function of $Z$ conditioned on $Z \ge K,$ 
		\begin{align*}
			\int_K^\infty xh_K(x)  dx =  \frac{1}{\bP(Z \ge K)}  \frac{1}{\sqrt{2\pi}}e^{-K^2/2}  \le 2K.
		\end{align*}
		Here we used a lower bound   $\bP(Z \ge K) \ge \frac{1}{2\sqrt{2\pi}} K^{-1} e^{-K^2/2}   $ for $K \ge 1.$
	\end{proof}
	

	The following two lemmas provide  conditional information for the Brownian increments and GUE increments. 
	\begin{lemma}\label{cond brownian}
		Let $(W_t)_t$ be  a standard Brownian motion. Then, there exists a constant $C>0$ such that for any    $0\le a\le b\le c$ with $a < c$,
		\begin{align*}
			\bE [ |W_{b}-W_a|   \mid  W_{c}  -W_{a}  ] &\le  \frac{b-a}{c-a} |W_{c}  -W_{a} |+ C\sqrt{b-a}.
		\end{align*}
	\end{lemma}
	\begin{lemma} \label{matrix}
		Let $(\widehat W^n_t)_{t\ge0}$ be the \textup{GUE}$(n)$–Brownian motion.
		Then there exists a constant $C>0$, independent of $n$, such that for any $0\le a\le b\le c$ with $a<c$,
		\[
		\mathbb{E}\bigl[\|\widehat W^n_b - \widehat W^n_a\| \bigm| \widehat W^n_c - \widehat W^n_a\bigr]
		\le
		\frac{b-a}{c-a}\,\|\widehat W^n_c - \widehat W^n_a\|
		+ C\,\sqrt{b-a}.
		\]
	\end{lemma}


	Since Lemma \ref{cond brownian} is a special case of Lemma \ref{matrix}, we prove the latter.
	\begin{proof}[Proof of Lemma \ref{matrix}]
		Define
		\[
		\Delta^n := \widehat W^n_c - \widehat W^n_a,
		\quad
		\Xi^n := (\widehat W^n_b - \widehat W^n_a)
		- \frac{b-a}{c-a} \Delta^n.
		\]
		Then we write
		\[
		\widehat W^n_b - \widehat W^n_a
		= \frac{b-a}{c-a}\Delta^n + \Xi^n.
		\] 
		Let $(W_t)_t$ be  a standard Brownian motion.
		Using the fact
		\begin{align*}
			\text{Cov}\Big (W_b-W_a - \frac{b-a}{c-a} (W_c-W_a), \  W_c-W_a\Big)= (b-a) - \frac{b-a}{c-a} (c-a)= 0,
		\end{align*}
		we deduce that for any $1\le i,j,k,\ell\le n,$
		\begin{align*}
			\text{Cov}(\Xi^n_{ij},\Delta^n_{k\ell})=0.
		\end{align*} 
		As $(\Xi^n,\Delta^n)$ is jointly Gaussian, $\Xi^n$ and $\Delta^n$ are independent.
		Also,
		\begin{align*}
			&\text{Var}\Big(W_b-W_a - \frac{b-a}{c-a} (W_c-W_a)\Big) \\
			&
			= \text{Var}(W_b-W_a)
			+ \Bigl(\frac{b-a}{c-a}\Bigr)^2\text{Var}(W_c-W_a)
			- 2\frac{b-a}{c-a}\text{Cov}(W_b-W_a,W_c-W_a)
			= \frac{(c-b)(b-a)}{c-a}.
		\end{align*}
		Thus  $\Xi^n$ has the same distribution as
		\[
		\frac{(c-b)(b-a)}{c-a}G^n,
		\]
		where $G^n$ denotes a unit‐time GUE$(n)$ matrix. As $\sup_n \mathbb E\|G^n\|<\infty$, we have 
		\[
		\mathbb E\Big[ \Big\| \frac{b-a}{c-a}\Delta^n + \Xi^n \Big\| \big\vert  \Delta^n \Big] \le \frac{b-a}{c-a} \|\Delta^n  \| + \mathbb E \| \Xi^n \| \le 
		\frac{b-a}{c-a}\,\|\widehat W^n_c - \widehat W^n_a\|
		+ C\sqrt{b-a},
		\]
		concluding the proof.
	\end{proof}

	\section{Comparison of GUE Laplacian and free Laplacian} \label{sec: Laplacians}
	
	The Laplacian operator corresponding to the GUE($n$) Brownian motion can be expressed as
	\begin{align}\label{eqn:gue_laplacian}
		\widehat{\Theta}_{M_n(\mathbb{C})} U(X) =  \frac{1}{ n^2}\sum_{l=1}^{d}\sum_{\widehat{i}=1}^n\sum_{\widehat{j}=1}^n{\rm Hess}\, U_{M_n(\mathbb{C})}(X)[  \mathbf{e}^l\, E_{\widehat{i}\widehat{j}},  \mathbf{e}^l\, E_{\widehat{i}\widehat{j}}].
	\end{align}
	Recall the normalized basis $E_{\widehat{i}\widehat{j}}$ introduced in \eqref{eqn:n_basis}. 
	We note that this is a standard Laplacian on the space of $d$-tuples of self-adjoint matrices equipped with the \emph{normalized} inner product associated to $\tr_n$, with a further dimensional normalization by $\frac{1}{n^2}$.  We remark that the same operator is obtained if we compute the Laplacian with respect to the inner product associated to $\Tr_n$ and then normalize by $1/n$, which is the normalization used in \cite{jekel2020elementary}, for instance.  For further discussion of the normalization, see \cite[Appendix B]{ST2022}.
	
	The free Laplacian is known to describe the large-$n$ behavior of the normalized Laplacian on $n \times n$ matrices for certain natural classes of functions; see for instance \cite[\S 3.2]{driver2013large}, \cite[\S 4.5]{jekel2022tracial}. A useful class of test functions  are cylindrical functions of the form
	$$
	U_{\cA}(X) = g\Big( \tau\big((\phi_1\circ\psi) (X)\big),\ldots,  \tau\big((\phi_m\circ \psi)(X)\big)\Big)
	$$
	for $g\in C^{2,1}(\bR^m;\bR)$, self-adjoint $\phi_1,\ldots,\phi_m\in {\rm NCP}_d$, and $\psi(x)={\rm arctan}(x)$ applied component-wise. We refer to \cite[Appendix B]{2025viscosity} for the details, some of which we repeat here for completeness. Clearly, $(U_{\cA})_{\cA\in \bW}$ are tracial $W^*$-functions. For $1\le o,q\le m,$
	$g_o$   denotes the partial derivative with respect to the $o$-th component of $g$, and $g_{oq}$  denotes the second partial derivative with respect to the $o$-and $q$-th components. We write
	$$
	g_o = g_o\Big( \tau\big((\phi_1\circ\psi) (X)\big),\ldots,  \tau\big((\phi_m\circ \psi)(X)\big)\Big)
	$$
	and
	$$
	g_{oq} = g_{oq}\Big( \tau\big((\phi_1\circ\psi) (X)\big),\ldots,  \tau\big((\phi_m\circ \psi)(X)\big)\Big).
	$$  
	Voiculescu's $j$th \emph{free difference quotient} is the map $\partial_{x_j}: \NCPd \to \NCPd \otimes \NCPd$ given, for a monomial, by
	$$
	\partial_{x_j} x_{i_1}x_{i_2}\ldots x_{i_m} := \sum_{i_k=j} (x_{i_1}\ldots x_{i_{k-1}})\otimes (x_{i_{k+1}} \ldots x_{i_m}).
	$$
	The \emph{cyclic derivative} $\mathcal{D}_{x_j}^\circ: \NCPd \to \NCPd$ is defined for non-commutative monomials by
	$$
	\mathcal{D}_{x_j}^\circ x_{i_1}x_{i_2}\ldots x_{i_m} := \sum_{i_k=j} \big(x_{i_{k+1}} \ldots x_{i_m} x_{i_1}\ldots x_{i_{k-1}} \big), 
	$$
	and extended linearly to non-commutative polynomials. Both of these derivatives extend to the ${\rm arctan}$ function. As a result or these definitions and the chain rule, the gradient of a cylindrical function is computed as, for a tracial $W^*$ algebra $\cA$ and $X\in L^2(\cA)_{sa}^d$,
	\begin{align*}
		(\nabla U_{\cA})^j(X) = \sum_{o=1}^m g_o\, \mathcal{D}_{x_j}^\circ  (\phi_o\circ \psi)(X) \quad j\in \{1,\ldots,d\} .
	\end{align*}
	We define the non-commutative derivative, for $i,j\in \{1,\ldots, d\}$, naturally to be
	$$
	\partial_{x_i} (\nabla U_{\cA})^j(X) = \sum_{o=1}^m g_o\, \partial_{x_i}\mathcal{D}_{x_j}^o (\phi_o\circ\psi)(X)  \in L^2(\cA\otimes \cA,\tau\otimes \tau).
	$$
	We also collect terms involving the second derivatives of $g$ as
	$$
	(\nabla^2 U_{\cA})^{ij}(X) = \sum_{o=1}^m\sum_{q=1}^m g_{oq}\, \mathcal{D}_{x_i}^o (\phi_o\circ\psi)(X)\otimes \mathcal{D}^o_{x_j} (\phi_q\circ\psi)(X).
	$$
	These terms make up the Hessian of a cylindrical function
	\begin{align}\label{eqn:cylindrical_hessian}
		&\ {\rm Hess}\, U_{\cA}(X)\big[A,B\big]\nonumber\\
		=&\ \sum_{i=1}^d\sum_{j=1}^d\Big( \langle(\nabla^2 U_\cA)^{ij}(X), A^i\otimes B^j\rangle_{L^2(\cA\otimes \cA)} + \langle \partial_{x_i} (\nabla U_\cA)^j(X)\# A^i,B^j\rangle_{L^2(\cA)}\Big),
	\end{align}
	recalling the $\#$ operation such that $(a\otimes b)\# c = a\, c\, b$ and extended linearly.
	
	Finally, see Proposition B.5 of \cite{2025viscosity}, we compute the free Lapacian of a cylindrical function to be
	\begin{align}\label{eqn:free_laplacian}
		\Theta_{\cA} U(X) = \sum_{i=1}^{d}(\tau\otimes \tau)\Big( \partial_{x_i} (\nabla U_{\cA})^i(X)\Big).
	\end{align}

	We now provide a relationship between GUE Laplacian $\widehat{\Theta}_{M_n(\mathbb{C})} $ and free Laplacian ${\Theta}_{M_n(\mathbb{C})} $.
	
	\begin{proposition}\label{cylindrical_comparison}
		Suppose that $U$ is a cylindrical function.  Then  for $X\in  M_n(\mathbb{C})_{\textup{sa}}^d$,
		$$
		\widehat{\Theta}_{M_n(\mathbb{C})} U(X) = \Theta_{M_n(\mathbb{C})} U(X) + \frac{1}{n^2}\sum_{l=1}^d\sum_{o=1}^m\sum_{q=1}^m g_{oq}\, \langle \mathcal{D}_{x_l}^o (\phi_o\circ\psi)(X),  \mathcal{D}^o_{x_l} (\phi_q\circ\psi)(X)\rangle_{M_n(\mathbb{C})}.
		$$
		Suppose that $X^n\in M_n(\mathbb{C})_{\textup{sa}}^d$ are a sequence of matrices converging in law to $X\in L^2(\cA)_{\textup{sa}}^d$ for $\cA\in \bW$. Then,
		$$
		\lim_{n\rightarrow \infty}\widehat{\Theta}_{M_n(\mathbb{C})} U(X^n)  = {\Theta}_{\cA} U(X).
		$$
	\end{proposition}
	\begin{proof}
		
		
		When computing \eqref{eqn:gue_laplacian}, the first term in the expression for the Hessian \eqref{eqn:cylindrical_hessian} has no counterpart for the free Laplacian and remains as the error term. For the GUE Laplacian, this term may be expressed as
		\begin{align*}
			&\ \frac{1}{n^2}\sum_{l=1}^d\sum_{\widehat{i}=1}^n\sum_{\widehat{j}=1}^n \langle(\nabla^2 U)^{ll}(X),  E_{\widehat{i}\widehat{j}}\otimes  E_{\widehat{i}\widehat{j}}\rangle_{{M_n(\mathbb{C})}\otimes {M_n(\mathbb{C})}}\\
			=&\ \frac{1}{n^2}\sum_{l=1}^d\sum_{\widehat{i}=1}^n\sum_{\widehat{j}=1}^n\sum_{o=1}^m\sum_{q=1}^m g_{oq}\, \langle \mathcal{D}_{x_l}^o (\phi_o\circ\psi)(X), E_{\widehat{i}\widehat{j}}\rangle_{M_n(\mathbb{C})}\, \langle \mathcal{D}^o_{x_l} (\phi_q\circ\psi)(X),E_{\widehat{i}\widehat{j}}\rangle_{M_n(\mathbb{C})}\\
			=&\ \frac{1}{n^2}\sum_{l=1}^d\sum_{o=1}^m\sum_{q=1}^m g_{oq}\, \langle \mathcal{D}_{x_l}^o (\phi_o\circ\psi)(X),  \mathcal{D}^o_{x_l} (\phi_q\circ\psi)(X)\rangle_{M_n(\mathbb{C})}.
		\end{align*}
		The quantities $g_{oq}\, \langle \mathcal{D}_{x_l}^o (\phi_o\circ\psi)(X),  \mathcal{D}^o_{x_l} (\phi_q\circ\psi)(X)\rangle_{M_n(\mathbb{C})}$ are $W^*$-tracial functions and thus the prefactor $\frac{1}{n^2}$ forces each of these terms to vanish in the limit $n\rightarrow \infty$.


		
		For the second term in (\ref{eqn:cylindrical_hessian}), we fix $o\in \{1,\ldots, m\}$, $l\in \{1,\ldots, d\}$, and $\hat{i},\hat{j}\in \{1,\ldots, n\}$  and consider a simple term as if $\partial_{x_l}(\nabla U)^l(X) = a\otimes b$. We have from the second Hessian term, 
		\begin{align*}
			\langle (a\otimes b)\# E_{\widehat{i}\widehat{j}},E_{\widehat{i}\widehat{j}}\rangle_{M_n(\mathbb{C})}=&\ \langle a\, E_{\widehat{i}\widehat{j}}\,  b, E_{\widehat{i}\widehat{j}}\,\rangle_{{M_n(\mathbb{C})}}\\
			=&\ \trn\big( E_{\widehat{i}\widehat{j}}\,  a\, E_{\widehat{i}\widehat{j}}\, b\big).
		\end{align*}
		When $\widehat{i}=\widehat{j}$ the above quantity is equal to
		$$
		a^{\widehat{i}\widehat{i}}\, b^{\widehat{i}\widehat{i}}.
		$$
		When $\widehat{i}<\widehat{j}$, we get
		\begin{align*}
			\frac{1}{2} (a^{\widehat{j}\widehat{i}}b^{\widehat{j}\widehat{i}} + a^{\widehat{j}\widehat{j}}b^{\widehat{i}\widehat{i}}+a^{\widehat{i}\widehat{j}}b^{\widehat{i}\widehat{j}}+a^{\widehat{i}\widehat{i}}b^{\widehat{j}\widehat{j}}),
		\end{align*}
		and when $\widehat{i}>\widehat{j}$, we get
		\begin{align*}
			\frac{1}{2}(-a^{\widehat{j}\widehat{i}}b^{\widehat{j}\widehat{i}} + a^{\widehat{j}\widehat{j}}b^{\widehat{i}\widehat{i}}-a^{\widehat{i}\widehat{j}}b^{\widehat{i}\widehat{j}}+a^{\widehat{i}\widehat{i}}b^{\widehat{j}\widehat{j}}).
		\end{align*}
		Summing over the basis we have
		$$
		\frac{1}{n^2}\sum_{\hat{i}=1}^n\sum_{\hat{j}=1}^n\langle (a\otimes b)\# E_{\widehat{i}\widehat{j}},E_{\widehat{i}\widehat{j}}\rangle_{M_n(\mathbb{C})} = {\rm tr}_n(a)\, {\rm tr}_n(b) = ({\rm tr}_n \otimes {\rm tr}_n)(a\otimes b).
		$$
		By linearity, this coincides with the formula \eqref{eqn:free_laplacian} for the free Laplacian.
	\end{proof}

	\bibliography{FreeViscBibJan21}
	\bibliographystyle{plain}

\end{document}